\documentclass[a4paper,12pt]{article}
\usepackage{mathrsfs,amsthm,graphicx,color,verbatim,bm,bbm,amsmath,amsfonts,amssymb,newclude,nicefrac,amsfonts,graphicx,enumerate,hyperref}
\usepackage[latin1]{inputenc}

\newtheorem{lemma}{Lemma}[section]
\newtheorem{theorem}[lemma]{Theorem}

\newtheorem{corollary}[lemma]{Corollary}

\providecommand{\N}{{\ensuremath{\mathbbm{N}}}}
\providecommand{\R}{{\ensuremath{\mathbbm{R}}}}
\providecommand{\E}{{\ensuremath{\mathbb{E}}}}
\renewcommand{\P}{{\ensuremath{\mathbb{P}}}}
\providecommand{\1}{{\ensuremath{\mathbbm{1}}}}

\newcommand{\diffns}[1]{d#1}
\providecommand{\Cb}[1]{{C_b^{#1}}}
\providecommand{\Cp}[2]{{C^{#1}_{#2}}}

\newcommand{\trans}{\phi}
\newcommand{\transmol}[1]{\phi_{#1}}
\newcommand{\lpn}[3]{\mathcal{L}^{#1}(#2;#3)}
\newcommand{\lpnb}[3]{L^{#1}(#2;#3)}
\newcommand{\smallsum}{\textstyle\sum}
\newcommand{\smallprod}{\textstyle\prod}

\newcommand{\deltaset}[1]{\mathbb{D}_{#1}}
\newcommand{\nzspace}[1]{{#1}\setminus\{0\}}
\newcommand{\supertheta}{\Upsilon}

\title{Regularity properties for solutions of infinite dimensional Kolmogorov equations in Hilbert spaces}

\author{Adam Andersson, Mario Hefter, Arnulf Jentzen, and Ryan Kurniawan}

\begin{document}

\maketitle

\begin{abstract}
In this article we establish regularity properties for solutions of infinite dimensional Kolmogorov equations.
We prove that if the nonlinear drift coefficients, the nonlinear diffusion coefficients, and the initial conditions of the considered Kolmogorov equations are $n$-times continuously Fr\'{e}chet differentiable, then so are the generalized solutions at every positive time. In addition, a key contribution of this work is to prove suitable enhanced regularity properties for the derivatives of the generalized solutions of the Kolmogorov equations in the sense that the dominating linear operator in the drift coefficient of the Kolmogorov equation regularizes the higher order derivatives of the solutions. Such enhanced regularity properties are of major importance for establishing weak convergence rates for spatial and temporal numerical approximations of stochastic partial differential equations. 
\end{abstract}

\tableofcontents

\section{Introduction}

In this article we establish regularity properties for solutions of infinite dimensional Kolmogorov equations.
Infinite dimensional Kolmogorov equations are the Kolmogorov equations associated to stochastic partial differential equations (SPDEs) and such equations have been intensively studied in the literature in the last three decades
(cf., e.g., 
Ma \& R{\"o}ckner~\cite{MaRoeckner1992}, 
R{\"o}ckner~\cite{Roeckner1999}, 
Zabczyk~\cite{Zabczyk1999}, 
Cerrai~\cite{c01}, 
Da Prato \& Zabczyk~\cite{dz02b}, 
R{\"o}ckner \& Sobol~\cite{RoecknerSobol2004}, 
Da Prato~\cite{DaPrato2004}, 
R{\"o}ckner~\cite{Roeckner2005}, 
R{\"o}ckner \& Sobol~\cite{RoecknerSobol2006}, 
R{\"o}ckner \& Sobol~\cite{RoecknerSobol2007}, 
Da Prato~\cite{DaPrato2007}, 
and the references mentioned therein).
In Theorem~\ref{thm:intro} below we summarize some of the main findings of this paper.
In our formulation of Theorem~\ref{thm:intro} we employ the following notation.
For every $ n \in \N=\{1,2,\ldots,\} $ and every non-trivial
$ \R $-Banach space
$ ( V , \left\| \cdot \right\|_V ) $
we denote by 
$\Cb{n}(V,\R)$ 
the set of all $n$-times continuously Fr\'{e}chet differentiable functions $f\colon V \to \R$ with globally bounded derivatives,
we denote by
$
  \left\|\cdot\right\|_{\Cb{n}( V, \R )}
$
the associated norm on $\Cb{n}(V,\R)$
(cf.~\eqref{eq:Cb.def} below),
we denote by 
$
  \operatorname{Lip}^n( V, \R )
$ 
the set of all functions $f\colon V \to \R$
in $\Cb{n}(V,\R)$ 
which have globally Lipschitz continuous derivatives, 
and we denote by 
$\left|\cdot\right|_{\operatorname{Lip}^n( V, \R )}$ 
an associated semi-norm on $\operatorname{Lip}^n( V, \R )$ 
(cf.~\eqref{eq:Lip.def} below). 
\begin{theorem}
\label{thm:intro}
Let
$
  \left(
  H,
  \left\| \cdot \right\|_H,
  \left< \cdot, \cdot \right>_H
  \right)
$
and 
$
  \left(
  U,
  \left\| \cdot \right\|_U ,
  \left< \cdot, \cdot \right>_U
  \right)
$ 
be non-trivial separable $ \R $-Hilbert spaces, 
let $\mathbb{U}\subseteq U$ be an orthonormal basis of $U$, 
let
$
A \colon D(A)
\subseteq
H \rightarrow H
$
be a generator of a strongly continuous analytic semigroup,
and let 
$ T \in (0,\infty) $, $ n \in \N $, 
$F\in\Cb{n}(H,H)$, $B\in\Cb{n}(H,HS(U,H))$. 
Then 
\begin{enumerate}[(i)]
\item
it holds that there exist unique functions 
$ P_t \colon \Cb{1}(H,\R) \to \Cp{}{}(H,\R) $, 
$ t \in [0,T] $, 
such that for every 
$ \varphi \in \Cb{1}(H,\R) $ 
it holds that $(P_t \varphi)(x) \in \R$, $ (t,x) \in [0,T] \times H $, is a generalized solution of
\begin{equation}
\label{eq:2nd.order.PDE}
\begin{split}
  \tfrac{\partial}{\partial t} (P_t \varphi)(x)
=
  \tfrac{1}{2}
  \smallsum\limits_{u\in\mathbb{U}}
  (P_t \varphi)''(x)
  ( B(x)u, B(x)u )
  +
  (P_t \varphi)'(x)
  [Ax + F(x)]
\end{split}
\end{equation}
for 
$ (t,x) \in (0,T] \times D(A) $ 
with 
$
  (P_0 \varphi)(x)=\varphi(x)
$
for $ x \in H $
(cf., e.g., \cite[page~127]{dz02b}),
\item 
\label{item:intro.Cb}
it holds for all 
$ k \in \{1,\ldots,n\} $, 
$t\in[0,T]$
that 
$
  P_t(\Cb{k}(H,\R))
  \subseteq \Cb{k}(H,\R)
$, 
\item 
\label{item:intro.regularity}
it holds for all 
$ k \in \{1,\ldots,n\} $, 
$t\in[0,T]$
with
$
|F|_{\operatorname{Lip}^k(H,H)}
+
|B|_{\operatorname{Lip}^k(H,HS(U,H))}
< \infty
$ 
that 
$
P_t(\operatorname{Lip}^k(H,\R))
\subseteq \operatorname{Lip}^k(H,\R)
$, 
\item
\label{item:intro.trans.apriori}
it holds for all 
$ k \in \{1,\ldots,n\} $,
$
\delta_1,\dots,\delta_k\in[0,\nicefrac{1}{2})
$
with 
$
\sum_{i=1}^k\delta_i
< \nicefrac{1}{2}
$ 
that 
\begin{equation}
\label{eq:intro.trans.apriori}
\begin{split}
&
\sup_{ \varphi \in \nzspace{\Cb{k}(H,\R)} }
\sup_{ x \in H }
\sup_{
	u_1, \dots, u_k \in 
	\nzspace{H}
}
\sup_{
	t \in (0,T]
}
\Bigg[
\frac{
	t^{\sum^k_{i=1} \delta_i}
	\,
	|
	(P_t \varphi)^{(k)}(x)
	(u_1,\ldots,u_k)
	|
}{
\|\varphi\|_{\Cb{k}(H,\R)}
\prod^k_{ i=1 }
\|
u_i
\|_{
	H_{ - \delta_i }
}
}
\Bigg]
<\infty,
\end{split}
\end{equation}
and 
\item
\label{item:intro.Lip}
it holds for all 
$ k \in \{1,\ldots,n\} $,
$
\delta_1, \ldots, \delta_k \in
[0,\nicefrac{1}{2})
$
with 
$
\sum_{i=1}^k\delta_i
< \nicefrac{1}{2}
$ 
and 
$
|F|_{\operatorname{Lip}^k(H,H)}
+
|B|_{\operatorname{Lip}^k(H,HS(U,H))}
< \infty
$ 
that 
\begin{equation}
\label{eq:intro.trans.lip}
\begin{split}
&
\sup_{ \varphi \in \nzspace{\operatorname{Lip}^k(H,\R)} }
\sup_{\substack{x, y \in H, \\ x \neq y}}
\sup_{\substack{ u_1,\ldots,u_k \\ \in \nzspace{H} }}
\sup_{t \in (0,T] }
\left[
\frac{
	t^{\sum^k_{i=1} \delta_i}
	|[
	(P_t \varphi)^{(k)}(x)
	-
	(P_t \varphi)^{(k)}(y)
	]
	(u_1,\ldots,u_k)	
	|
}{
\|\varphi\|_{\operatorname{Lip}^k(H,\R)} \,
\|v-w\|_H
\prod^k_{i=1} \|u_i\|_{H_{-\delta_i}}
}
\right]
\\&< \infty.
\end{split}
\end{equation}
\end{enumerate}
\end{theorem}
In the case $n=2$, item~\eqref{item:intro.Cb} in Theorem~\ref{thm:intro} is a generalization of Theorem~6.7 in Zabczyk~\cite{Zabczyk1999}
and Theorem~7.4.3 in Da Prato \& Zabczyk~\cite{dz02b} 
(in this paper $\Cb{2}$-functions do not necessarily need to be globally bounded; compare the sentence above Lemma~3.4 in~\cite{Zabczyk1999} and item~(ii) on page~31 in~\cite{dz02b} with~\eqref{eq:Cb.def} in this paper).
Theorem~\ref{thm:intro} is a straightforward consequence of Theorem~\ref{thm:derivative_formulas} in Section~\ref{sec:regularity_extended_transition_semigroup} below.
In Theorem~\ref{thm:derivative_formulas} below we also specify 
for every natural number $n\in\N$ 
and every $t\in[0,T]$ an explicit formula for the $n$-th
derivative of the generalized solution 
$
  H \ni x \mapsto
  ( P_t \varphi )(x) \in \R
$ 
of~\eqref{eq:2nd.order.PDE} at time $t\in[0,T]$.
Moreover, Theorem~\ref{thm:derivative_formulas} below provides explicit bounds for the left hand sides of~\eqref{eq:intro.trans.apriori} and~\eqref{eq:intro.trans.lip} (see items~\eqref{item:thm2.derivative.bound} and~\eqref{item:thm.trans.lip} in Theorem~\ref{thm:derivative_formulas} below).
Next we would like to emphasize that Theorem~\ref{thm:intro} and Theorem~\ref{thm:derivative_formulas}, respectively, prove finiteness of~\eqref{eq:intro.trans.apriori} and~\eqref{eq:intro.trans.lip} even though the denominators in~\eqref{eq:intro.trans.apriori} and~\eqref{eq:intro.trans.lip} contain rather weak norms from negative Sobolev-type spaces for the multilinear arguments of the derivatives of the generalized solution.
In particular, item~\eqref{item:intro.trans.apriori} in Theorem~\ref{thm:intro} above 
and item~\eqref{item:thm2.derivative.bound} in Theorem~\ref{thm:derivative_formulas} below, respectively, 
reveal for every
$ p \in [1,\infty) $, 
$ k \in \{1,2,\ldots,n\} $, 
$ \delta_1, \delta_2, \ldots, \delta_k \in [0,\nicefrac{1}{2}) $, 
$ x \in H $,
$ t \in (0,T] $
that the $k$-th derivative
$
	(P_t \varphi)^{(k)}(x)
$
even takes values in the continuously embedded subspace 
\begin{equation}
\label{eq:rough.space}
L( \otimes^k_{i=1} H_{-\delta_i}, \R )
\end{equation} 
of 
$
L( H^{\otimes k}, \R )
$
provided that the hypothesis 
\begin{equation}
\label{eq:condition.delta}
\smallsum^k_{i=1} \delta_i
< \nicefrac{1}{2}
\end{equation}
is satisfied.
In addition, we employ items~\eqref{item:intro.trans.apriori}--\eqref{item:intro.Lip} in Theorem~\ref{thm:intro} above and items~\eqref{item:thm2.derivative.bound} and~\eqref{item:thm.trans.lip} in Theorem~\ref{thm:derivative_formulas} below, respectively, to establish similar a priori bounds as~\eqref{eq:intro.trans.apriori}--\eqref{eq:intro.trans.lip} for a family of appropriately mollified solutions of~\eqref{eq:2nd.order.PDE} which hold uniformly in the mollification parameter; 
see items~\eqref{item:cor.mollified.transition.bound}--\eqref{item:cor.mollified.transition.lip} in Corollary~\ref{thm:mollified.transition} below for details.
Items~\eqref{item:intro.trans.apriori}--\eqref{item:intro.Lip} in Theorem~\ref{thm:intro} above, 
items~\eqref{item:thm2.derivative.bound} and~\eqref{item:thm.trans.lip} in Theorem~\ref{thm:derivative_formulas} below,
and, especially, items~\eqref{item:cor.mollified.transition.bound}--\eqref{item:cor.mollified.transition.lip} in Corollary~\ref{thm:mollified.transition} below, respectively, are of major importance for establishing essentially sharp probabilistically \emph{weak convergence rates} for numerical approximation processes
as the analytically weak norms for the multilinear arguments of the derivatives of the generalized solution (cf.\ the denominators in~\eqref{eq:intro.trans.apriori} and~\eqref{eq:intro.trans.lip} above) translate in analytically weak norms for the approximation errors in the probabilistically weak error analysis which, in turn, result in essentially sharp probabilistically weak convergence rates for the numerical approximation processes
(cf., e.g., 
Theorem~2.2 in Debussche~\cite{Debussche2011},
Theorem~2.1 in Wang \& Gan~\cite{WangGan2013_Weak_convergence}, 
Theorem~1.1 in Andersson \& Larsson~\cite{AnderssonLarsson2015},
Theorem~1.1 in Br\'{e}hier~\cite{Brehier2014},
Theorem~5.1 in Br\'{e}hier \& Kopec~\cite{BrehierKopec2016},
Corollary~1 in Wang~\cite{Wang2016481},
Corollary~5.2 in Conus et al.~\cite{ConusJentzenKurniawan2014arXiv}, 
Theorem~6.1 in Kopec~\cite{Kopec2014_PhD_Thesis},
and
Corollary~8.2 in~\cite{JentzenKurniawan2015arXiv}).

\subsection{Notation}
\label{sec:notation}

In this section we introduce some of the notation which we employ throughout the article
(cf., e.g., \cite[Section~1.1]{AnderssonJentzenKurniawan2016arXiv}).
For two sets $ A $ and $ B $ we denote by 
$ \mathbb{M}(A,B) $ 
the set of all mappings from $A$ to $B$.
For two
measurable spaces
$
  ( A, \mathcal{A} )
$
and
$
  ( B, \mathcal{B} )
$
we denote by
$
  \mathcal{M}( \mathcal{A}, \mathcal{B} )
$
the set of
$ \mathcal{A} $/$ \mathcal{B} $-measurable
functions.
For a set $ A $ 
we denote by 
$ \mathcal{P}(A) $ the power set of $ A $
and we denote by 
$ \#_A \in \N_0 \cup \{\infty\} $ 
the number of elements of $ A $.
For a Borel measurable set $ A \in \mathcal{B}(\R) $ 
we denote by $ \mu_A \colon \mathcal{B}(A) \to [0,\infty] $ 
the Lebesgue-Borel measure on $A$. 
We denote by 
$ \lfloor \cdot \rfloor \colon \R \to \R $
and 
$ \lceil \cdot \rceil \colon \R \to \R $
the functions which satisfy for all
$ t \in \R $
that
$
  \lfloor t \rfloor =
  \max\!\left(
    ( -\infty, t ]
    \cap
    \{ 0, 1 , - 1 , 2 , - 2 , \dots \}
  \right)
$
and 
$
  \lceil t \rceil =
  \min\!\left(
    [ t, \infty )
    \cap
    \{ 0, 1 , - 1 , 2 , - 2 , \dots \}
  \right)
$. 
For $ \R $-Banach spaces
$ ( V , \left\| \cdot \right\|_V ) $
and
$ ( W , \left\| \cdot \right\|_W ) $
with 
$ \#_V > 1 $
and a natural number $ n \in \N $
we denote by
$
  \left| \cdot \right|_{ \Cb{n}( V, W ) }
  \colon
  \Cp{n}{}( V, W ) \to [0,\infty]
$
and
$
  \left\| \cdot \right\|_{ \Cb{n}( V, W ) }
  \colon
  \Cp{n}{}( V, W ) \to [0,\infty]
$
the functions which satisfy
for all $ f \in \Cp{n}{}( V, W ) $
that
\begin{equation}
\label{eq:Cb.def}
\begin{split}
  \left| f \right|_{
    \Cb{n}( V, W )
  }
& =
  \sup\nolimits_{
    x \in V
  }
  \left\|
    f^{ (n) }( x )
  \right\|_{
    L^{ (n) }( V, W )
  }
  ,
\qquad
  \left\| f \right\|_{
    \Cb{n}( V, W )
  }
  =
  \|f(0)\|_W
  +
  \smallsum_{ k = 1 }^n
  \left| f \right|_{ \Cb{k}(V,W) }
\end{split}
\end{equation}
and we denote by
$
  \Cb{n}( V, W )
$
the set given by
$
  \Cb{n}( V, W ) =
  \{ f \in \Cp{n}{}( V, W ) \colon \left\| f \right\|_{ \Cb{n}( V, W ) } < \infty \}
$.
For $ \R $-Banach spaces
$ ( V , \left\| \cdot \right\|_V ) $
and
$ ( W , \left\| \cdot \right\|_W ) $
with
$ \#_V > 1 $
and a nonnegative integer $ n \in \N_0 $
we denote by
$
  \left| \cdot \right|_{
    \operatorname{Lip}^n( V, W )
  }
  \colon 
  \Cp{n}{}( V, W )
  \to [0,\infty]
$
and 
$
  \left\| \cdot \right\|_{
    \operatorname{Lip}^n( V, W )
  }
  \colon 
  \Cp{n}{}( V, W )
  \to [0,\infty]
$
the functions which satisfy
for all $ f \in \Cp{n}{}( V, W ) $
that
\begin{equation}
\label{eq:Lip.def}
\begin{split}
  \left| f \right|_{ 
    \operatorname{Lip}^n( V, W )
  }
  &=
\begin{cases}
  \sup_{ 
    \substack{
      x, y \in V ,\,
      x \neq y
    }
  }
  \left(
  \frac{
    \left\| f( x ) - f( y ) \right\|_W
  }{
    \left\| x - y \right\|_V
  }
  \right)
&
  \colon
  n = 0
\\
  \sup_{ 
    \substack{
      x, y \in V ,\,
      x \neq y
    }
  }
  \left(
  \frac{
    \| f^{ (n) }( x ) - f^{ (n) }( y ) \|_{ L^{ (n) }( V, W ) }
  }{
    \left\| x - y \right\|_V
  }
  \right)
&
  \colon
  n \in \N
\end{cases}
  ,
\\
  \left\| f \right\|_{
    \operatorname{Lip}^n( V, W )
  }
  &
  =
  \|f(0)\|_W
  +
  \smallsum_{ k = 0 }^n
  \left| f \right|_{ \operatorname{Lip}^k(V,W) }
\end{split}
\end{equation}
and we denote by
$
  \operatorname{Lip}^n( V, W )
$
the set given by
$
  \operatorname{Lip}^n( V, W ) =
  \{ f \in \Cp{n}{}( V, W ) \colon \left\| f \right\|_{ \operatorname{Lip}^n( V, W ) } < \infty \}
$.
We denote by  
$ \Pi_k, \Pi^*_k \in 
\mathcal{P}\big(\mathcal{P}\big(
\mathcal{P}( \N )
\big)\big)
$, 
$ k \in \N_0 $,
the sets which satisfy for all $ k \in \N $
that
$ \Pi_0 = \Pi_0^{ * } = \emptyset $,
$
\Pi_k^{ * } =
\Pi_k \backslash
\big\{ 
\{ \{ 1, 2, \dots, k \} \}
\big\}
$,
and
\begin{equation}
\Pi_k =
\big\{
A \subseteq \mathcal{P}( \N )
\colon
\left[
\emptyset \notin A
\right]
\wedge
\left[
\cup_{ a \in A }
a
=
\left\{ 1, 2, \dots, k \right\}
\right]
\wedge
\left[
\forall \, a, b \in A \colon
\left(
a \neq b
\Rightarrow
a \cap b = \emptyset
\right)
\right]
\big\}
\end{equation}
(see, e.g., (10) in Andersson et al.~\cite{AnderssonJentzenKurniawan2016a}).
For a natural number $ k \in \N $
and a set $ \varpi \in \Pi_k $
we denote by
$
  I^\varpi_1 , I^\varpi_2, \dots , I^\varpi_{ \#_\varpi }
  \in
  \varpi
$
the sets which satisfy that
$
  \min( I^\varpi_1 ) < 
  \min( I_2^\varpi ) < \dots < 
  \min( 
    I_{ \#_\varpi }^{ \varpi } 
  )
$.
For 
a natural number
$ k \in \N $,
a set
$ \varpi \in \Pi_k $,
and a natural number
$ i \in \{ 1, 2, \dots, \#_\varpi \} $
we denote by
$
  I_{ i, 1 }^\varpi ,
  I_{ i, 2 }^\varpi ,
  \dots,
  I_{ i, \#_{ I_i^{ \varpi } } }^\varpi
  \in
  I_i^{ \varpi }
$
the natural numbers which satisfy that
$
  I_{ i, 1 }^\varpi < I_{ i, 2 }^\varpi < \dots < I_{ i, \#_{ I_i^{ \varpi } } }^\varpi
$.
For a measure space $ ( \Omega , \mathcal{F}, \mu ) $,
a measurable space $ ( S , \mathcal{S} ) $,
a set $ R $, 
and a function
$ f \colon \Omega \to R $
we denote by
$
   \left[ f \right]_{
     \mu, \mathcal{S}
   }
$
the set given by
\begin{equation}
   \left[ f \right]_{
     \mu, \mathcal{S}
   }
   =
   \left\{
     g \in \mathcal{M}( \mathcal{F}, \mathcal{S} )
     \colon
     (
     \exists \, A \in \mathcal{F} \colon
     \mu(A) = 0 
     \text{ and }
     \{ \omega \in \Omega \colon f(\omega) \neq g(\omega) \}
     \subseteq A
     )
   \right\}
   .
\end{equation}

\subsection{Setting}
\label{sec:global_setting}

Throughout this article the following setting is frequently used.
Let
$ T \in (0,\infty) $,
$ \eta \in \R $,
let
$
  \left(
    H,
    \left\| \cdot \right\|_H,
    \left< \cdot, \cdot \right>_H
  \right)
$
and 
$
  \left(
    U,
    \left\| \cdot \right\|_U ,
    \left< \cdot, \cdot \right>_U
  \right)
$ 
be separable $ \R $-Hilbert spaces 
with $\#_H > 1$, 
let 
$
  \left(
    V,
    \left\| \cdot \right\|_V
  \right)
$ 
be a separable $ \R $-Banach space, 
let
$
( \Omega , \mathcal{F}, \P )
$
be a probability space with a normal filtration 
$
( \mathcal{F}_t )_{ t \in [0,T] }
$,
let
$
  ( W_t )_{ t \in [0,T] }
$
be an $ \operatorname{Id}_U $-cylindrical 
$ ( \Omega , \mathcal{F}, \P, ( \mathcal{F}_t )_{ t \in [0,T] } ) $-Wiener process,
let
$
  A \colon D(A)
  \subseteq
  H \rightarrow H
$
be a generator of a strongly continuous analytic semigroup
with 
$
  \operatorname{spectrum}( A )
  \subseteq
  \{
    z \in \mathbb{C}
    \colon
    \text{Re}( z ) < \eta
  \}
$,
let
$
  (
    H_r
    ,
    \left\| \cdot \right\|_{ H_r }
    ,
    \left< \cdot , \cdot \right>_{ H_r }
  )
$,
$ r \in \R $,
be a family interpolation spaces associated to
$
  \eta - A
$
(cf., e.g., \cite[Section~3.7]{sy02}),
for every $ k \in \N $,
$ \varpi \in \Pi_k $,
$
  i \in \{ 1, 2, \dots, \#_\varpi \} 
$
let 
$
  [ \cdot ]_i^\varpi
  \colon
  H^{ k + 1 }
  \to 
  H^{ 
    \#_{I_i^\varpi} + 1
  }
$
be the mapping which satisfies for all 
$
  \mathbf{u} = (u_0, u_1, \dots, u_k)
  \in 
  H^{ k + 1 }
$
that
$
  [ \mathbf{u} ]_i^\varpi
  = ( u_0, u_{ I_{ i, 1 }^\varpi } , u_{ I_{ i, 2 }^\varpi } , \dots , u_{ I_{ i, \#_{I_i^\varpi} }^\varpi } )
$, 
for every 
$ k \in \N $, 
$ \bm{\delta} = (\delta_1,\delta_2,\ldots,\delta_k) \in \R^k $, 
$ \alpha \in [0,1) $, 
$ \beta \in [0,\nicefrac{1}{2}) $, 
$ J \in \mathcal{P}(\R) $ 
let 
$
  \iota^{\bm{\delta},\alpha,\beta}_J \in \R
$
be the real number given by 
$
  \iota^{\bm{\delta},\alpha,\beta}_J
  =
  \sum_{i\in J\cap\{1,2,\ldots,k\}} \delta_i
  -
  \1_{[2,\infty)}(\#_{J\cap\{1,2,\ldots,k\}}) \,
  \min\{1-\alpha,\nicefrac{1}{2}-\beta\}
$, 
and for every separable $ \R $-Banach space $ ( J, \left\| \cdot \right\|_J ) $
and every 
$ a \in \R $, $ b \in (a,\infty) $, $ I \in \mathcal{B}( \R ) $,
$
  X \in \mathcal{M}( \mathcal{B}( I ) \otimes \mathcal{F} , \mathcal{B}( J ) )
$ 
with  
$ (a,b) \subseteq I $
let 
$
  \int_a^b X_s \, {\bf ds}
  \in 
  L^0( \P ; J )
$
be the set given by
$
  \int_a^b X_s \, {\bf ds}
  =
  \big[
    \int_a^b \mathbbm{1}_{ \{ \int_a^b \| X_u \|_J \, du < \infty \} } X_s \, ds
  \big]_{
    \P , \mathcal{B}( J )
  }
$.

\section{Some auxiliary results for the differentiation of random fields}
\begin{lemma}[A chain rule for random fields]
\label{lem:stoch.chain.rule}
Let 
$
  \left(
    U,
    \left\| \cdot \right\|_U
  \right)
$ 
be an $\R$-Banach space with $\#_U > 1$, 
let 
$
  \left(
    V,
    \left\| \cdot \right\|_{V}
  \right)
$  
and 
$
  \left(
    W,
    \left\| \cdot \right\|_W
  \right)
$ 
be separable $\R$-Banach spaces, 
let $ (\Omega,\mathcal{F},\P) $ be a probability space, 
let 
$ X^{k,\mathbf{u}} \in \cap_{ p \in [1,\infty) } \lpn{p}{\P}{V} $, $\mathbf{u} \in U^{k+1}$, 
$ k \in \{0,1\} $, satisfy for all 
$ p \in [1,\infty) $, 
$ x,u \in U $
that 
$ 
 \big( U \ni y \mapsto [X^{0,y}]_{\P,\mathcal{B}(V)} \in \lpnb{p}{\P}{V} \big) \in 
 \Cp{1}{}(U,\lpnb{p}{\P}{V}) 
$ 
and 
$
  \big( \frac{d}{dx} [X^{0,x}]_{\P,\mathcal{B}(V)} \big) u
  =
  [X^{1,(x,u)}]_{\P,\mathcal{B}(V)}
$,
and let $ \varphi \in \Cp{1}{}(V,W) $ 
satisfy that 
$
  \limsup_{ p \nearrow \infty }
  \sup_{ x \in V }
  \frac{
    \|\varphi'(x)\|_{L(V,W)}
  }{
    |\!\max\{1,\|x\|_{V}\}|^p
  }
  < \infty
$.
Then 
\begin{enumerate}[(i)]
\item
\label{item:stoch.derivative.integrability}
it holds for all $ x, u \in U $ that 
$
  \E[\|\varphi(X^{0,x})\|_W+\|\varphi'(X^{0,x})X^{1,(x,u)}\|_W]
  < \infty
$, 
\item
\label{item:stoch.derivative} 
it holds that 
$
  \big(
    U \ni x \mapsto \E[\varphi(X^{0,x})] \in W
  \big) 
  \in \Cp{1}{}(U,W)
$, 
and 
\item
\label{item:stoch.derivative.representation}
it holds for all $ x, u \in U $ that 
$
  \big(
  \frac{d}{dx}
  \E[\varphi(X^{0,x})]
  \big) u
  =
  \E[\varphi'(X^{0,x})X^{1,(x,u)}]
$.
\end{enumerate}
\end{lemma}
\begin{proof}
Throughout this proof 
let 
$
  c_{k,r} \in [0,\infty]
$, $ r \in (0,\infty) $, $ k \in \{0,1\} $, 
be the extended real numbers which satisfy for all
$ r \in (0,\infty) $ that 
\begin{equation}
\begin{split}
&
  c_{0,r}
  =
  \sup_{ x \in V }
  \bigg[
  \frac{
  	\|\varphi(x)\|_W
  }{
  |\!\max\{1,\|x\|_{V}\}|^r
  }
  \bigg]
  \qquad\text{ and }\qquad
  c_{1,r}
  =
  \sup_{ x \in V }
  \bigg[
  \frac{
  	\|\varphi'(x)\|_{L(V,W)}
  }{
  |\!\max\{1,\|x\|_{V}\}|^r} 
  \bigg]
\end{split}
\end{equation} 
and let $ p \in [1,\infty) $ 
be a real number which satisfies that 
$
  c_{1,p}
  < \infty
$. 
We note that the fundamental theorem of calculus implies that for all 
$ x \in V $ 
it holds that 
\begin{equation}
\begin{split}
&
  \|\varphi(x) - \varphi(0)\|_W
  =
  \left\|
  \int^1_0
  \varphi'(\rho x) x
  \, d\rho
  \right\|_W
\leq
  \int^1_0
  \|\varphi'(\rho x)\|_{L(V,W)} \,
  \|x\|_{V}
  \, d\rho
\\&\leq
  c_{1,p} \,
  \|x\|_{V}
  \sup\nolimits_{\rho\in[0,1]}
  |\!\max\{1,\|\rho x\|_{V}\}|^p
=
  c_{1,p} \,
  \|x\|_{V} \,
  |\!\max\{1,\|x\|_{V}\}|^p
\\&\leq
  c_{1,p} \,
  |\!\max\{1,\|x\|_{V}\}|^{(p+1)}
  .
\end{split}
\end{equation}
This ensures that 
$
  c_{0,p+1}
  < \infty
$.
H\"{o}lder's inequality 
and the fact that 
$ c_{1,p} < \infty $
therefore show that for all 
$ x, u \in U $ 
it holds that 
\begin{equation}
\label{eq:holder.stoch.derivative}
\begin{split}
&
  \E[\|\varphi'(X^{0,x}) X^{1,(x,u)}\|_W]
  \leq
  c_{1,p} \,
  \E[|\!\max\{1,\|X^{0,x}\|_{V}\}|^p \, \|X^{1,(x,u)}\|_{V}]
\\&\leq
  c_{1,p} \,
  \|\!\max\{1,\|X^{0,x}\|_{V}\}\|^p_{\lpn{2p}{\P}{\R}} \,
  \|X^{1,(x,u)}\|_{\lpn{2}{\P}{V}}  
  < \infty
\end{split}
\end{equation}
and 
\begin{equation}
  \E[\|\varphi(X^{0,x})\|_W]
  \leq
  c_{0,p+1} \,
  \E[|\!\max\{1,\|X^{0,x}\|_{V}\}|^{(p+1)}]
  < \infty.
\end{equation}
This proves item~\eqref{item:stoch.derivative.integrability}.
Next note that~\eqref{eq:holder.stoch.derivative} and the fact that 
$
  \forall \, q \in [1,\infty), \,  x \in U
  \colon
  \big(
  U \ni u \mapsto
  [X^{1,(x,u)}]_{\P,\mathcal{B}(V)}
  \in \lpnb{q}{\P}{V}
  \big) 
  \in L( U, \lpnb{q}{\P}{V} )
$ 
ensure that for every $ x \in U $ 
it holds
\begin{enumerate}[a)]
\item
that 
\begin{equation}
\begin{split}
&
  \sup_{ u \in U, \, \|u\|_U=1 }
    \|
    \E[\varphi'(X^{0,x})X^{1,(x,u)}]
    \|_W
\\&\leq
  c_{1,p} \,
  \|\!\max\{1,\|X^{0,x}\|_{V}\}\|^p_{\lpn{2p}{\P}{\R}}
  \sup_{ u \in U, \, \|u\|_U=1 }
  \|X^{1,(x,u)}\|_{\lpn{2}{\P}{V}}
  < \infty 
\end{split}
\end{equation} 
and 
\item
that the function 
$
  \big(
    U \ni u \mapsto \E[\varphi'(X^{0,x})X^{1,(x,u)}] \in W
  \big)
$ 
is linear. 
\end{enumerate} 
Hence, we obtain that
\begin{equation}
\label{eq:derivative.composition}
  \big(
    U \ni u \mapsto \E[\varphi'(X^{0,x})X^{1,(x,u)}] \in W
  \big)
  \in L( U, W )
  .
\end{equation}
In the next step we demonstrate that for all 
$ x \in U $ 
it holds that 
\begin{equation}
\label{eq:1st.derivative}
  \limsup_{ U\setminus\{0\} \ni u \to 0 }
  \bigg(
  \frac{
	\|
	\E[ \varphi(X^{0,x+u}) ] - \E[ \varphi(X^{0,x}) ] -
	\E[ \varphi'(X^{0,x}) X^{1,(x,u)} ]
	\|_W
  }{
    \|u\|_U
  }
  \bigg)
  =0.
\end{equation}
For this we first observe that for all 
$ x, u \in U $
it holds that 
\begin{equation}
\label{eq:1st.derivative.decompose}
\begin{split}
&
  	\|
  	\E[ \varphi(X^{0,x+u}) ] - \E[ \varphi(X^{0,x}) ] -
  	\E[ \varphi'(X^{0,x}) X^{1,(x,u)} ]
  	\|_W
\\&\leq
  	\|
  	\E[ \varphi(X^{0,x+u}) - \varphi(X^{0,x}) - \varphi'(X^{0,x})(X^{0,x+u}-X^{0,x}) ]
  	\|_W
\\&+
  	\|
  	\E[ \varphi'(X^{0,x})( X^{0,x+u} - X^{0,x} - X^{1,(x,u)} ) ]
  	\|_W
   .
\end{split}
\end{equation}
Moreover, we note that H\"{o}lder's inequality and the fact that 
$ c_{1,p} < \infty $ 
ensure that for all 
$ x \in U $ 
it holds that 
\begin{equation}
\label{eq:1st.derivative.inner}
\begin{split}
&
  \limsup_{ U\setminus\{0\} \ni u \to 0 }
  \bigg(
  \frac{
  	\|
  	\E[ \varphi'(X^{0,x})( X^{0,x+u} - X^{0,x} - X^{1,(x,u)} ) ]
  	\|_W
  }{
  \|u\|_{U}  
  }
  \bigg)
\\&\leq
  \|\varphi'(X^{0,x})\|_{\lpn{2}{\P}{L(V,W)}}
  \limsup_{ U\setminus\{0\} \ni u \to 0 }
  \bigg(
  \frac{
  	\|
  	X^{0,x+u} - X^{0,x} - X^{1,(x,u)}
  	\|_{\lpn{2}{\P}{V}}
  }{
  \|u\|_{U} } 
  \bigg)
\\&\leq
  c_{1,p} \,
  \|\!\max\{1,\|X^{0,x}\|_{V}\}\|^p_{\lpn{2p}{\P}{\R}}
  \limsup_{ U\setminus\{0\} \ni u \to 0 }
  \bigg(
  \frac{
  	\|
  	X^{0,x+u} - X^{0,x} - X^{1,(x,u)}
  	\|_{\lpn{2}{\P}{V}}
  }{
  \|u\|_{U} } 
  \bigg)
  = 0.
\end{split}
\end{equation} 
Furthermore, we observe that the fundamental theorem of calculus shows that for all 
$ x, u \in U $ it holds that 
\begin{equation}
\begin{split}
&
  \|
    \varphi(X^{0,x+u}) - \varphi(X^{0,x}) - \varphi'(X^{0,x})(X^{0,x+u}-X^{0,x})
  \|_W
\\&=
  \left\|
    \int^1_0
    \big[\varphi'( X^{0,x} + \rho [ X^{0,x+u}-X^{0,x} ] ) - \varphi'(X^{0,x})\big] (X^{0,x+u}-X^{0,x})
    \, d\rho
  \right\|_W
\\&\leq
  \|
    X^{0,x+u}-X^{0,x}
  \|_{V}  
  \int^1_0
  \| \varphi'( X^{0,x} + \rho [ X^{0,x+u}-X^{0,x} ] ) - \varphi'(X^{0,x}) \|_{ L(V,W) }
  \, d\rho
  .
\end{split}
\end{equation}
H\"{o}lder's inequality and Jensen's inequality therefore imply that for all 
$ x, u \in U $ it holds that 
\begin{equation}
\label{eq:2nd.term.bound}
\begin{split}
&
  	\|
  	\E[ \varphi(X^{0,x+u}) - \varphi(X^{0,x}) - \varphi'(X^{0,x})(X^{0,x+u}-X^{0,x}) ]
  	\|_W
\\&\leq
  \left\{
  \E\!\left[
  \int^1_0
  \| \varphi'( X^{0,x} + \rho [ X^{0,x+u}-X^{0,x} ] ) - \varphi'(X^{0,x}) \|^2_{ L(V,W) }
  \, d\rho  
  \right]
  \right\}^{1/2}    
\\&\cdot
  \|
    X^{0,x+u}-X^{0,x}
  \|_{\lpn{2}{\P}{V}}
  .  
\end{split}
\end{equation}
Moreover, note that for all 
$ q \in (2,\infty) $, 
$ \rho \in [0,1] $, 
$ x, y \in U $ 
it holds that 
\begin{equation}
\label{eq:integrand.Lp.bounded}
\begin{split}
&
  \E\big[\| \varphi'( X^{0,x} + \rho [ X^{0,y}-X^{0,x} ] ) \|^q_{ L(V,W) }\big]
\\&\leq
  |c_{1,p}|^q \,
  \E\big[|\!
    \max\{1,\|X^{0,x} + \rho [ X^{0,y}-X^{0,x} ]\|^p_{V}\} 
  |^q\big]
\\&\leq
  |c_{1,p}|^q \, 
  \E\big[|\! 
    \max\{1,\|X^{0,x}\|_{V},\|X^{0,y}\|_{V}\} 
  |^{pq}\big]
\\&\leq
  |c_{1,p}|^q  
  \big(
    1
    +
    \E\big[\|X^{0,x}\|^{pq}_{V}\big]  
    +
    \E\big[\|X^{0,y}\|^{pq}_{V}\big]    
  \big) 
  .
\end{split}
\end{equation}
This and the fact that 
$
  \forall \, q \in [1,\infty)
  \colon
  \big( U \ni x \mapsto [X^{0,x}]_{\P,\mathcal{B}(V)} \in \lpnb{q}{\P}{V} \big)
  \in \Cp{}{}(U,\lpnb{q}{\P}{V})
$ 
ensure that for all 
$ q \in (2,\infty) $, 
$ x \in U $ 
it holds that 
\begin{equation}
\label{eq:UI}
\begin{split}
&
  \limsup_{ U \ni u \to 0 }
  \int^1_0
  \E\big[\| \varphi'( X^{0,x} + \rho [ X^{0,x+u}-X^{0,x} ] ) \|^q_{ L(V,W) }\big]
  \, d\rho  
\leq
  |c_{1,p}|^q \,
  \big(
    1
    +
    2 \, \E\big[\|X^{0,x}\|^{pq}_{V}\big]  
  \big) 
\\&< \infty.
\end{split}
\end{equation}
In addition, observe that the fact that 
$
  \forall \, q \in [1,\infty)
  \colon
  \big(
    U \ni x \mapsto
    [X^{0,x}]_{ \P, \mathcal{B}(V) } \in \lpnb{q}{\P}{V}
  \big) 
  \in \Cp{}{}(U,\lpnb{q}{\P}{V})
$
shows that for all 
$ x \in U $ it holds that 
\begin{equation}
  \limsup\nolimits_{ U \ni y \to x }
  \E\big[\!
  \min\{1,
    \|X^{0,x}-X^{0,y}\|_V
  \}
  \big]
  =0.
\end{equation}
This implies that for all 
$ \rho \in [0,1] $, 
$ x \in U $ 
it holds that 
\begin{equation}
  \limsup\nolimits_{ U \ni y \to x }
  \E\big[\!
  \min\{1,
    \|(X^{0,x} + \rho [ X^{0,y} - X^{0,x} ]) - X^{0,x}\|_V
  \}
  \big]
  =0.
\end{equation}
The fact that 
$ \varphi' \in \Cp{}{}( V, L(V,W) ) $ 
hence ensures that for all 
$ \rho \in [0,1] $, 
$ x \in U $ 
it holds that 
\begin{equation}
\label{eq:integrand.CIP}
  \limsup\nolimits_{ U \ni y \to x }
  \E\big[\!\min\{
    1,
    \| \varphi'( X^{0,x} + \rho [ X^{0,y}-X^{0,x} ] ) - \varphi'(X^{0,x}) \|_{ L(V,W) }
  \}\big]
  = 0.
\end{equation}
This and Lebesgue's theorem of dominated convergence imply that for all 
$ x \in U $ 
it holds that 
\begin{equation}
\label{eq:CIP}
  \limsup_{ U \ni u \to 0 }
  \int^1_0
  \E\big[\!\min\{
    1,
    \| \varphi'( X^{0,x} + \rho [ X^{0,x+u}-X^{0,x} ] ) - \varphi'(X^{0,x}) \|_{ L(V,W) }
  \}\big]
  \, d\rho
  = 0.
\end{equation}
Combining this and, e.g., Lemma~4.2 in 
Hutzenthaler et al.~\cite{HutzenthalerJentzenSalimova2016arXiv} 
(with 
$ I = \{\emptyset\} $, 
$ 
  (\Omega,\mathcal{F},\P) =
  ([0,1] \times\Omega,\mathcal{B}([0,1])\otimes\mathcal{F},
  \mu_{[0,1]}\otimes\P) 
$, 
$ c=1 $, 
$
  X^n(\emptyset,(\rho,\omega)) = 
  \|\varphi'( X^{0,x}(\omega) + \rho [ X^{0,x+u_n}(\omega) - X^{0,x}(\omega) ] ) - \varphi'( X^{0,x}(\omega) )\|_{L(V,W)}
$
for 
$ (\rho,\omega) \in [0,1] \times \Omega $, 
$ n \in \N $, 
$ x \in U $, 
$
  (u_m)_{ m \in \N }
  \in 
  \{
    v \in \mathbb{M}( \N, U )
    \colon
    \limsup_{ m \to \infty }
    \|v_m\|_U
    =0
  \}
$
in the notation of Lemma~4.2 in 
Hutzenthaler et al.~\cite{HutzenthalerJentzenSalimova2016arXiv}) 
establishes that for all 
$ \varepsilon \in (0,\infty) $, 
$ x \in U $
and all sequences $ (u_n)_{ n \in \N } \subseteq U $ 
with 
$
    \limsup_{ n \to \infty }
    \|u_n\|_U
    =0
$ 
it holds that 
\begin{multline}
  \limsup\nolimits_{ n \to \infty } \,
  ( \mu_{[0,1]} \otimes \P )\big(\big\{
    (\rho,\omega) \in [0,1] \times \Omega \colon
    \| \varphi'( X^{0,x}(\omega) + \rho [ X^{0,x+u_n}(\omega)-X^{0,x}(\omega) ] ) 
\\
    - \varphi'(X^{0,x}(\omega)) \|_{ L(V,W) }
    \geq \varepsilon
  \big\}\big)
  =0.
\end{multline}
This, \eqref{eq:UI}, and, e.g., Proposition~4.5 in 
Hutzenthaler et al.~\cite{HutzenthalerJentzenSalimova2016arXiv} 
(with 
$ I = \{\emptyset\} $, 
$ 
  (\Omega,\mathcal{F},\P) =
  ([0,1] \times\Omega,\mathcal{B}([0,1])\otimes\mathcal{F},
  \mu_{[0,1]}\otimes\P) 
$, 
$ p = q $, 
$ V = \R $, 
$
  X^n(\emptyset,(\rho,\omega)) = 
  \|\varphi'( X^{0,x}(\omega) + \rho [ X^{0,x+u_n}(\omega) - X^{0,x}(\omega) ] ) - \varphi'( X^{0,x}(\omega) )\|_{L(V,W)}
$
for 
$ (\rho,\omega) \in [0,1] \times \Omega $, 
$ n \in \N_0 $, 
$ x \in U $, 
$ q \in (2,\infty) $, 
$
  (u_m)_{ m \in \N_0 }
  \in 
  \{
    v \in \mathbb{M}( \N_0, U )
    \colon
    \limsup_{ m \to \infty }
    \|v_m\|_U
    =
    \|v_0\|_U
    =0
  \}
$
in the notation of Proposition~4.5 in 
Hutzenthaler et al.~\cite{HutzenthalerJentzenSalimova2016arXiv})
yield that for all 
$ x \in U $ 
and all sequences 
$ (u_n)_{ n \in \N_0 } \subseteq U $ 
with 
$
  \limsup_{ n \to \infty }
  \|u_n\|_U = \|u_0\|_U = 0
$ 
it holds that 
\begin{equation}
\label{eq:double.integral.convergence}
  \limsup_{ n \to \infty }
  \int^1_0
  \E\big[\| \varphi'( X^{0,x} + \rho [ X^{0,x+u_n}-X^{0,x} ] ) - \varphi'(X^{0,x}) \|^2_{ L(V,W) }\big]
  \, d\rho  
  = 0.
\end{equation}
Moreover, observe that the triangle inequality and the fact that 
$
  \forall \, q \in [1,\infty), \,  x \in U
  \colon
  \big(
  U \ni u \mapsto
  [X^{1,(x,u)}]_{\P,\mathcal{B}(V)}
  \in \lpnb{q}{\P}{V}
  \big) 
  \in L( U, \lpnb{q}{\P}{V} )
$ 
assure that for all 
$ x \in U $ 
it holds that 
\begin{equation}
\label{eq:derivative.boundedness}
\begin{split}
&
  \limsup_{ U \setminus \{0\} \ni u \to 0 }
  \bigg[
  \frac{
    \| X^{0,x+u} - X^{0,x} \|_{ \lpn{2}{\P}{V} }
  }{
    \|u\|_U
  }
  \bigg]
\\&\leq
  \limsup_{ U \setminus \{0\} \ni u \to 0 }
  \bigg[
  \frac{
    \| X^{0,x+u} - X^{0,x} - X^{1,(x,u)} \|_{ \lpn{2}{\P}{V} }
  }{
    \|u\|_U
  }
  \bigg]
  +
  \sup_{ u \in U \setminus \{0\} }
  \bigg[
  \frac{
    \| X^{ 1,(x,u) } \|_{ \lpn{2}{\P}{V} }
  }{
    \|u\|_U
  }
  \bigg]
\\&=
  \sup_{ u \in U \setminus \{0\} }
  \bigg[
  \frac{
    \| X^{ 1,(x,u) } \|_{ \lpn{2}{\P}{V} }
  }{
    \|u\|_U
  }
  \bigg]
  < \infty.
\end{split}
\end{equation}
Putting~\eqref{eq:double.integral.convergence}--\eqref{eq:derivative.boundedness} 
into~\eqref{eq:2nd.term.bound} yields that for all 
$ x \in U $ 
it holds that 
\begin{equation}
\label{eq:2nd.term.vanish}
\begin{split}
&
  	\limsup_{  U \setminus \{0\} \ni u \to 0 }
  	\bigg(
  	\frac{\|
  	\E[ \varphi(X^{0,x+u}) - \varphi(X^{0,x}) - \varphi'(X^{0,x})(X^{0,x+u}-X^{0,x}) ]
  	\|_W  }{ \|u\|_U }
  	\bigg)
\\&\leq
  \limsup_{  U \setminus \{0\} \ni u \to 0 }
  \bigg(
  \frac{
    \| X^{0,x+u} - X^{0,x} \|_{ \lpn{2}{\P}{V} }
  }{
    \|u\|_U
  }
  \bigg)
\\&\cdot
  \left[
  \limsup_{  U \setminus \{0\} \ni u \to 0 }
  \int^1_0
  \E\big[\| \varphi'( X^{0,x} + \rho [ X^{0,x+u}-X^{0,x} ] ) - \varphi'(X^{0,x}) \|^2_{ L(V,W) }\big]
  \, d\rho  
  \right]^{1/2}  
  =0.  
\end{split}
\end{equation}
Combining~\eqref{eq:1st.derivative.decompose}, \eqref{eq:1st.derivative.inner}, and~\eqref{eq:2nd.term.vanish} proves~\eqref{eq:1st.derivative}.
In the next step we demonstrate that 
\begin{equation}
\label{eq:derivative.continuity}
\big(
U \ni x \mapsto
\big(
U \ni u \mapsto \E[\varphi'(X^{0,x})X^{1,(x,u)}] \in W
\big)
\in L( U, W )
\big)
\in
\Cp{}{}(U,L(U,W))
.
\end{equation}
Observe that~\eqref{eq:integrand.Lp.bounded} and the fact that 
$
  \forall \, q \in [1,\infty)
  \colon
  \limsup_{ U \ni y \to x }
  \E[ \|X^{0,y}\|^q_{V} ]
  =
  \E[ \|X^{0,x}\|^q_{V} ]
  < \infty
$ 
ensure that for all 
$ q \in (2,\infty) $, 
$ \rho \in [0,1] $, 
$ x \in U $ it holds that 
\begin{equation}
\begin{split}
&
  \limsup\nolimits_{ U \ni y \to x }
  \E\big[\|\varphi'(X^{0,x} + \rho [X^{0,y}-X^{0,x}])\|^q_{L(V,W)}\big]
\\&\leq
  |c_{1,p}|^q \, 
  \big(
    1
    +
    \E\big[ \|X^{0,x}\|^{pq}_V \big]
    +
    \limsup\nolimits_{ U \ni y \to x }
    \E\big[ \|X^{0,y}\|^{pq}_V \big]
  \big)
\\&=
  |c_{1,p}|^q \, 
  \big(
    1
    +
    2 \, \E\big[ \|X^{0,x}\|^{pq}_V \big]
  \big)
  <\infty.
\end{split}
\end{equation}
Hence, we obtain that for all 
$ q \in (2,\infty) $, 
$ x \in U $ 
it holds that 
\begin{equation}
\label{eq:integrand.Lp.bounded.II}
\begin{split}
&
  \limsup_{ U \ni y \to x }
  \E\big[
  \|\varphi'(X^{0,x})-\varphi'(X^{0,y})\|^q_{L(V,W)}
  \big]
\\&\leq
  2^q
  \limsup_{ U \ni y \to x }
  \max_{ \rho \in \{0,1\} }
  \E\big[
  \|\varphi'(X^{0,x}+\rho[X^{0,y}-X^{0,x}])\|^q_{L(V,W)}
  \big]
  < \infty.
\end{split}
\end{equation}
Moreover, note that~\eqref{eq:integrand.CIP} 
(with $\rho=1$ in the notation of~\eqref{eq:integrand.CIP})
and, e.g., Lemma~4.2 in 
Hutzenthaler et al.~\cite{HutzenthalerJentzenSalimova2016arXiv} 
(with 
$ I = \{\emptyset\} $, 
$ 
  (\Omega,\mathcal{F},\P) =
  (\Omega,\mathcal{F},\P) 
$, 
$ c=1 $, 
$
  X^n(\emptyset,\omega) = 
  \|\varphi'( X^{0,u_n}(\omega)  ) - \varphi'( X^{0,u_0}(\omega) )\|_{L(V,W)}
$
for 
$ \omega \in \Omega $, 
$ n \in \N $, 
$
  (u_m)_{ m \in \N_0 }
  \in 
  \{
    v \in \mathbb{M}( \N_0, U )
    \colon
    \limsup_{ m \to \infty }
    \|v_m-v_0\|_U
    =0
  \}
$
in the notation of Lemma~4.2 in 
Hutzenthaler et al.~\cite{HutzenthalerJentzenSalimova2016arXiv}) 
establishes that for all 
$ \varepsilon \in (0,\infty) $ 
and all sequences 
$ (u_n)_{ n \in \N_0 } \subseteq U $ 
with 
$
  \limsup_{ n \to \infty }
  \| u_n - u_0 \|_U = 0
$ 
it holds that 
\begin{equation}
  \limsup\nolimits_{ n \to \infty }
  \P\big(
    \|\varphi'( X^{0,u_n}  ) - \varphi'( X^{0,u_0} )\|_{L(V,W)}
    \geq \varepsilon
  \big)
  =0.
\end{equation}
Combining this, \eqref{eq:integrand.Lp.bounded.II}, and, e.g.,
Proposition~4.5 in 
Hutzenthaler et al.~\cite{HutzenthalerJentzenSalimova2016arXiv} 
(with 
$ I = \{\emptyset\} $, 
$ 
  (\Omega,\mathcal{F},\P) =
  (\Omega,\mathcal{F},\P) 
$, 
$ p = q $, 
$ V = \R $, 
$
  X^n(\emptyset,\omega) = 
  \|\varphi'( X^{0,u_n}(\omega)  ) - \varphi'( X^{0,u_0}(\omega) )\|_{L(V,W)}
$
for 
$ \omega \in \Omega $, 
$ q \in (2,\infty) $, 
$ n \in \N_0 $, 
$
  (u_m)_{ m \in \N_0 }
  \in 
  \{
    v \in \mathbb{M}( \N_0, U )
    \colon
    \limsup_{ m \to \infty }
    \|v_m-v_0\|_U
    =0
  \}
$
in the notation of Proposition~4.5 in 
Hutzenthaler et al.~\cite{HutzenthalerJentzenSalimova2016arXiv})
yields that for all sequences 
$
  (u_n)_{ n \in \N_0 }
  \subseteq U
$ 
with 
$
    \limsup_{ n \to \infty }
    \|u_n-u_0\|_U
    =0
$
it holds that 
\begin{equation}
\limsup\nolimits_{ n \to \infty }
\E\big[
\|\varphi'(X^{0,u_n})-\varphi'(X^{0,u_0})\|^2_{L(V,W)}
\big]
= 0.
\end{equation}
Next observe that the fact that for every 
$ q \in [1,\infty) $ 
it holds that the function 
$
  U \ni x \mapsto
  \big(
    U \ni u \mapsto
    [ X^{1,(x,u)} ]_{ \P, \mathcal{B}(V) }
    \in \lpnb{q}{\P}{V}
  \big)
  \in L( U, \lpnb{q}{\P}{V} )
$
is continuous shows that for all $ x \in U $ 
it holds that 
\begin{equation}
\limsup_{ U \ni y \to x }\sup_{ u \in U, \, \|u\|_U=1 }
\|X^{1,(y,u)}\|_{\lpn{2}{\P}{V}}
=
\sup_{ u \in U, \, \|u\|_U=1 }
\|X^{1,(x,u)}\|_{\lpn{2}{\P}{V}}
< \infty  
\end{equation} 
and 
\begin{equation}
  \limsup_{ U \ni y \to x }
  \sup_{ u \in U, \, \|u\|_U=1 }
  \|X^{1,(x,u)} - X^{1,(y,u)}\|_{\lpn{2}{\P}{V}}
  =0.
\end{equation}
H\"{o}lder's inequality and~\eqref{eq:integrand.Lp.bounded.II} 
hence ensure that for all $ x \in U $ it holds that 
\begin{equation}
\begin{split}
&
  \limsup_{ U \ni y \to x }
  \sup_{ u \in U, \, \|u\|_U=1 }
  \big\|
  \E[ \varphi'(X^{0,x}) X^{1,(x,u)} ]
  -
  \E[ \varphi'(X^{0,y}) X^{1,(y,u)} ]
  \big\|_W
\\&\leq
  \limsup_{ U \ni y \to x }
  \sup_{ u \in U, \, \|u\|_U=1 }
  \E\big[ \|\varphi'(X^{0,x}) (X^{1,(x,u)} - X^{1,(y,u)})\|_W \big]
\\&+
  \limsup_{ U \ni y \to x }
  \sup_{ u \in U, \, \|u\|_U=1 }
  \E\big[ \|[\varphi'(X^{0,x})-\varphi'(X^{0,y})] X^{1,(y,u)}\|_W \big]
\\&\leq
  \|\varphi'(X^{0,x})\|_{\lpn{2}{\P}{L(V,W)}}
  \limsup_{ U \ni y \to x }
  \sup_{ u \in U, \, \|u\|_U=1 }
  \|X^{1,(x,u)} - X^{1,(y,u)}\|_{\lpn{2}{\P}{V}}
\\&+
  \left[\limsup_{ U \ni y \to x }\|\varphi'(X^{0,x})-\varphi'(X^{0,y})\|_{\lpn{2}{\P}{L(V,W)}}\right]
  \limsup_{ U \ni y \to x }\sup_{ u \in U, \, \|u\|_U=1 }
  \|X^{1,(y,u)}\|_{\lpn{2}{\P}{V}}
\\&\leq
  c_{1,p} \,
  \|\!\max\{1,\|X^{0,x}\|_V\}\|^p_{ \lpn{2p}{\P}{\R} }
  \limsup_{ U \ni y \to x }
  \sup_{ u \in U, \, \|u\|_U=1 }
  \|X^{1,(x,u)} - X^{1,(y,u)}\|_{\lpn{2}{\P}{V}}
\\&+
  \left[\limsup_{ U \ni y \to x }\|\varphi'(X^{0,x})-\varphi'(X^{0,y})\|_{\lpn{2}{\P}{L(V,W)}}\right]
  \sup_{ u \in U, \, \|u\|_U=1 }
  \|X^{1,(x,u)}\|_{\lpn{2}{\P}{V}}
  =0. 
\end{split}
\end{equation} 
This proves~\eqref{eq:derivative.continuity}. 
Combining~\eqref{eq:derivative.composition}, \eqref{eq:1st.derivative}, 
and~\eqref{eq:derivative.continuity} establishes item~\eqref{item:stoch.derivative} and item~\eqref{item:stoch.derivative.representation}.
The proof of Lemma~\ref{lem:stoch.chain.rule} is thus completed.
\end{proof}
\begin{lemma}[Pointwise differentiation]
\label{lem:partial.diff}
Let 
$
  \left(
    V,
    \left\| \cdot \right\|_V
  \right)
$
and 
$
  \left(
    W,
    \left\| \cdot \right\|_W
  \right)
$ 
be $\R$-Banach spaces with $\#_V > 1$ 
and let $ n \in \N $, 
$ f \in \Cp{n}{}(V,W) $, 
$ g \in \Cp{}{}(V,L^{(n+1)}(V,W)) $ 
satisfy for all 
$ \mathbf{u}=(u_1,u_2,\ldots,u_n) \in V^n $, 
$ x \in V $ 
that 
\begin{equation}
\label{eq:single.diff}
  \limsup_{ V\setminus\{0\} \ni h \to 0 }
  \bigg[
  \frac{
    \|
      f^{(n)}(x+h) \mathbf{u}
      -
      f^{(n)}(x) \mathbf{u}
      -
      g(x)(u_1,u_2,\ldots,u_n,h)
    \|_W
  }{
    \|h\|_V
  }
  \bigg]
  =0.
\end{equation}
Then it holds that 
$ f \in \Cp{n+1}{}(V,W) $ 
and 
$
  f^{(n+1)}=g
$.
\end{lemma}
\begin{proof}
We first note that~\eqref{eq:single.diff} 
and the fact that 
$
  \forall \, x, u_1, u_2, \ldots, u_n \in V
  \colon
  \big(
    V \ni h \mapsto g(x)(u_1, u_2, \ldots, 
$
$
    u_n,h) \in W
  \big) \in L(V,W)
$ 
and 
$
  \big(
    V \ni y \mapsto
    \big(
      V \ni h \mapsto g(y)(u_1, u_2, \ldots, u_n,h) \in W
    \big) \in L(V,W)
  \big) 
  \in \Cp{}{}(V,L(V,W))
$ 
imply that 
for all 
$ \mathbf{u} = (u_1, u_2, \ldots, u_n) \in V^n $, 
$ x, h \in V $
it holds that 
$
  \big(
  V \ni y \mapsto
  f^{(n)}(y) \mathbf{u} \in W
  \big) 
  \in \Cp{1}{}(V,W)
$ 
and 
$
  \big(\frac{d}{dx} \big(f^{(n)}(x) \mathbf{u}\big)\big) h
  =
  g(x)(u_1, u_2, \ldots, u_n,h)
$. 
This and the fundamental theorem of calculus imply that for all 
$ \mathbf{u} = (u_1,u_2,\ldots,u_n) \in V^n $, 
$ x, h \in V $
it holds that 
\begin{equation}
\label{eq:TFK}
\begin{split}
&
    \|
      f^{(n)}(x+h) \mathbf{u}
      -
      f^{(n)}(x) \mathbf{u}
      -
      g(x)(u_1, u_2, \ldots, u_n,h)
    \|_W
\\&=
  \big\|\textstyle\int^1_0
    \big[
      g(x+\rho h) - g(x)
    \big]
    (u_1, u_2, \ldots, u_n,h)
  \, d\rho\big\|_W
\\&\leq
  \|h\|_V
  \big[
  \smallprod^n_{i=1}
  \|u_i\|_V
  \big]
  \int^1_0
  \|g(x+\rho h) - g(x)\|_{ L^{(n+1)}(V,W) }
  \, d\rho
  .
\end{split}
\end{equation}
In addition, observe that the assumption that 
$ g \in \Cp{}{}(V,L^{(n+1)}(V,W)) $
ensures that for all 
$ x \in V $ 
it holds that 
\begin{equation}
  \limsup_{ V \ni h \to 0 }
  \sup_{ \rho \in [0,1] }
  \| g(x+\rho h) \|_{ L^{(n+1)}(V,W) }
  < \infty.
\end{equation}
Lebesgue's theorem of dominated convergence therefore ensures that for all 
$ x \in V $ 
it holds that 
\begin{equation}
\label{eq:Cb.convergence}
  \limsup\nolimits_{ V \ni h \to 0 }
  \textstyle\int^1_0
  \|g(x+\rho h) - g(x)\|_{ L^{(n+1)}(V,W) }
  \, d\rho
  =0.
\end{equation}
Combining~\eqref{eq:TFK} with~\eqref{eq:Cb.convergence} yields that 
for all $ x \in V $ it holds that 
\begin{equation}
  \limsup_{ V \setminus \{0\} \ni h \to 0 }
  \sup_{ \mathbf{u}=(u_1,u_2,\ldots,u_n) \in (V \setminus \{0\})^n }
  \bigg[
  \frac{
    \|
      f^{(n)}(x+h) \mathbf{u}
      -
      f^{(n)}(x) \mathbf{u}
      -
      g(x)(u_1, u_2, \ldots, u_n,h)
    \|_W
  }{
    \|h\|_V
    \prod^n_{i=1}
    \|u_i\|_V
  }
  \bigg]
  =0.
\end{equation}
This and the assumption that 
$ g \in \Cp{}{}(V,L^{(n+1)}(V,W)) $ 
complete the proof of Lemma~\ref{lem:partial.diff}.
\end{proof}

\section{Regularity of transition semigroups for stochastic evolution equations}
\label{sec:regularity_extended_transition_semigroup}

This section establishes regularity properties of the 
transition semigroup.

\begin{lemma}
\label{lem:test.derivative}
Let 
$
  (V, \left\|\cdot\right\|_V)
$ 
and 
$
  (W, \left\|\cdot\right\|_W)
$ 
be $\R$-Banach spaces with $ \#_V > 1 $, 
let $ n \in \N $, 
$
  \varphi \in \Cp{n+1}{}(V,W)
$, 
and let 
$
  \Phi \colon V^{n+1} \to W
$ 
be the function which satisfies for all 
$
  \mathbf{v}=(v_1,v_2,\ldots,v_{n+1}) \in V^{n+1}
$ 
that 
$
  \Phi(\mathbf{v})
  =
  \varphi^{(n)}(v_{n+1})(v_1,v_2,\ldots,v_n)
$. 
Then it holds for all 
$
  \mathbf{v}=(v_1,v_2,\ldots,v_{n+1}), \,
  \mathbf{h}=(h_1,h_2,\ldots,
$
$
  h_{n+1})
  \in V^{n+1}
$ 
that 
$ \Phi \in \Cp{1}{}(V^{n+1},W) $ 
and 
\begin{equation}
\label{eq:total.derivative}
\begin{split}
  \Phi'(\mathbf{v}) \mathbf{h}
  =&
      \varphi^{(n+1)}(v_{n+1})(v_1,v_2,\ldots,v_n,h_{n+1})
\\&+
      \smallsum^n_{i=1}
      \varphi^{(n)}(v_{n+1})(v_1,v_2,\ldots,v_{i-1},h_i,v_{i+1},v_{i+2},\ldots,v_n)
      .
\end{split}
\end{equation}
\end{lemma}
\begin{proof}
Throughout this proof let 
$ P \colon V^{n+1} \to L^{(n)}(V,W) \times V^n $, 
$ \beta \colon L^{(n)}(V,W) \times V^n \to W $, 
and 
$ \phi \colon V^2 \to L^{(n)}(V,W) $ 
be the functions which satisfy for all 
$ A \in L^{(n)}(V,W) $, 
$ \mathbf{v}=(v_1,v_2,\ldots,v_n) \in V^n $, 
$ v, h \in V $ 
that 
\begin{equation}
\begin{split}
&
  P( v_1,v_2,\ldots,v_n, v )
  =
  ( \varphi^{(n)}(v), \mathbf{v} )
  ,
  \qquad
  \beta( A, \mathbf{v} )
  =
  A(v_1,v_2,\ldots,v_n)
  ,
\end{split}
\end{equation}
and
\begin{equation}
\begin{split}
\phi(v,h)\mathbf{v}
=
\varphi^{(n+1)}(v)(v_1,v_2,\ldots,v_n,h)
.
\end{split}
\end{equation}
We note that for all 
$ \mathbf{v}=(v_1,v_2,\ldots,v_n) \in V^n $, 
$ v \in V $ 
it holds that 
\begin{equation}
\label{eq:composition}
  \Phi(v_1,v_2,\ldots,v_n,v)
  =
  \beta(P(v_1,v_2,\ldots,v_n, v))
  .
\end{equation}
Furthermore, observe that the assumption that $ \varphi \in \Cp{n+1}{}(V,W) $ 
ensures that 
for all 
$ \mathbf{v}=(v_1,v_2,\ldots,v_n) $, $ \mathbf{h}=(h_1,h_2,\ldots,h_n) \in V^n $, 
$ v, h \in V $ 
it holds that 
$ P \in \Cp{1}{}( V^{n+1}, L^{(n)}(V,W) \times V^n ) $
and 
\begin{equation}
\label{eq:compose.derivative.1}
  P'(v_1,v_2,\ldots,v_n, v)(h_1,h_2,\ldots,h_n,h)
  =
  (\phi(v,h),\mathbf{h})
  .
\end{equation}
Moreover, the fact that 
$\beta$ is an $(n+1)$-multilinear and continuous function
and, e.g., Theorem~3.7 in Coleman~\cite{Coleman2012} 
assure that for all 
$ A, \tilde{A} \in L^{(n)}(V,W) $, 
$ \mathbf{v}=(v_1,v_2,\ldots,v_n) $, $ \mathbf{h}=(h_1,h_2,\ldots,h_n) \in V^n $ 
it holds that 
$ \beta \in \Cp{1}{}(L^{(n)}(V,W) \times V^n, W) $ 
and 
\begin{equation}
\label{eq:compose.derivative.2}
  \beta'(A,\mathbf{v})(\tilde{A},\mathbf{h})
  =
  \tilde{A}(v_1,v_2,\ldots,v_n)
  +
  \sum^n_{i=1}
  A(v_1,v_2,\ldots,v_{i-1},h_i,v_{i+1},v_{i+2},\ldots,v_n)
  .
\end{equation}
Combining~\eqref{eq:composition}--\eqref{eq:compose.derivative.2} with 
the chain rule yields that for all 
$ \mathbf{v}=(v_1,v_2,\ldots,v_n) $, 
$ \mathbf{h}=(h_1,h_2,\ldots,h_n) 
$
$
\in V^n $, 
$ v, h \in V $ 
it holds that 
$ \Phi \in \Cp{1}{}(V^{n+1},W) $
and 
\begin{equation}
\begin{split}
&
  \Phi'(v_1,v_2,\ldots,v_n,v)(h_1,h_2,\ldots,h_n,h)
\\&=
  \beta'(P(v_1,v_2,\ldots,v_n, v))P'(v_1,v_2,\ldots,v_n, v)(h_1,h_2,\ldots,h_n,h)
\\&=
  \beta'(P(v_1,v_2,\ldots,v_n, v))(\phi(v,h),\mathbf{h})
\\&=
  \beta'(\varphi^{(n)}(v),\mathbf{v})(\phi(v,h),\mathbf{h})
\\&=
  \phi(v,h)\mathbf{v}
  +
  \sum^n_{i=1}
  \varphi^{(n)}(v)(v_1,v_2,\ldots,v_{i-1},h_i,v_{i+1},v_{i+2},\ldots,v_n)
\\&=
  \varphi^{(n+1)}(v)(v_1,v_2,\ldots,v_n,h)
  +
  \sum^n_{i=1}
  \varphi^{(n)}(v)(v_1,v_2,\ldots,v_{i-1},h_i,v_{i+1},v_{i+2},\ldots,v_n)
  .
\end{split}
\end{equation}
This implies~\eqref{eq:total.derivative}. 
The proof of Lemma~\ref{lem:test.derivative} is thus completed.
\end{proof}
\begin{lemma}
\label{lem:derivative_formulas}
Assume the setting in Section~\ref{sec:global_setting}, 
let $ n \in \N $,
$ \varphi \in \Cb{n}(H,V) $, 
$ F \in \Cb{n}(H,H) $,
$ B \in \Cb{n}(H,HS(U,H)) $, 
let 
$
  X^{ k,\mathbf{u} }
  \colon
  [ 0 , T ] \times \Omega
  \to H
$, 
$
  \mathbf{u} \in H^{k+1}
$, 
$
  k \in \{ 0, 1, \dots, n \}
$,
be 
$(\mathcal{F}_t)_{t\in[0,T]}$/$ \mathcal{B}(H) $-predictable stochastic processes
which satisfy
for all
$
  k \in \{ 0, 1, \dots, n \}
$,
$
  \mathbf{u} = (u_0,u_1,\ldots,u_k) \in H^{k+1}
$, 
$ p \in (0,\infty) $,
$ t \in [0,T] $
that
$
  \sup_{s\in[0,T]}
  \E\big[\|X^{k,\mathbf{u}}_s\|^p_H\big]
  < \infty
$ 
and 
\begin{equation}
\label{eq:SEE.derivative}
\begin{split}
&
  [
  X_t^{k,\mathbf{u}}
  -
  e^{tA}
  \, \mathbbm{1}_{ \{ 0, 1 \} }(k) \, u_k  
  ]_{ \P,\mathcal{B}(H) }
\\ &
  =
  \int_0^t
    e^{ ( t - s ) A }
    \Bigg[
      \mathbbm{1}_{ \{ 0 \} }(k)
      \,
      F(X_s^{0,u_0})
\\&\quad+
      \sum_{ \varpi\in \Pi_k }
      F^{ ( \#_\varpi ) }( X_s^{ 0,u_0 } )
      \big(
        X_s^{ \#_{I^\varpi_1}, [ \mathbf{u} ]_1^{ \varpi } }
        ,
        X_s^{ \#_{I^\varpi_2}, [ \mathbf{u} ]_2^{ \varpi } }
        ,
        \dots
        ,
        X_s^{ \#_{I^\varpi_{\#_\varpi}}, [\mathbf{u} ]_{ \#_\varpi }^{ \varpi } }
      \big)
    \Bigg]
  \,{\bf ds}
\\ &
  +
  \int_0^t
    e^{ ( t - s ) A }
    \Bigg[
      \mathbbm{1}_{ \{ 0 \} }(k)
      \,
      B(X_s^{0,u_0})
\\&\quad+
      \sum_{ \varpi\in \Pi_k }
      B^{ ( \#_\varpi ) }( X_s^{ 0,u_0 } )
      \big(
        X_s^{ \#_{I^\varpi_1}, [ \mathbf{u} ]_1^{ \varpi } }
        ,
        X_s^{ \#_{I^\varpi_2}, [ \mathbf{u} ]_2^{ \varpi } }
        ,
        \dots
        ,
        X_s^{ \#_{I^\varpi_{\#_\varpi}}, [\mathbf{u} ]_{ \#_\varpi }^{ \varpi } }
      \big)
    \Bigg]
  \, \diffns W_s
  ,
\end{split}
\end{equation}
and let 
$
  \trans \colon [0,T] \times H \to V
$ 
be the function which satisfies for all 
$ t \in [0,T] $, 
$ x \in H $ 
that 
$
  \trans(t,x)
  =
  \E[\varphi(X^{0,x}_t)]
$.
Then 
\begin{enumerate}[(i)]
\item
\label{item:trans.derivative.integrability}
it holds for all 
$ k \in \{1,2,\ldots,n\} $, 
$
  \mathbf{u} = (u_0 , u_1 , \dots , u_k) \in
  H^{k+1}
$, 
$ t \in [ 0 , T ] $ 
that 
\begin{equation}
  \sum\limits_{
    \varpi \in \Pi_k
  }
  \E\Big[\big\|
    \varphi^{(\#_\varpi)}(X_t^{0,u_0})
    \big(
        X_t^{ \#_{I^\varpi_1}, [ \mathbf{u} ]_1^{ \varpi } }
        ,
        X_t^{ \#_{I^\varpi_2}, [ \mathbf{u} ]_2^{ \varpi } }
        ,
        \dots
        ,
        X_t^{ \#_{I^\varpi_{\#_\varpi}}, [\mathbf{u} ]_{ \#_\varpi }^{ \varpi } }
    \big)
  \big\|_V\Big]
  < \infty, 
\end{equation}
\item
\label{item:trans.smoothness}
it holds for all 
$ t \in [ 0 , T ] $ 
that 
$ \big( H \ni x \mapsto \trans(t,x) \in V \big) \in \Cb{n}(H,V) $,  
\item
\label{item:thm.representation}
it holds for all 
$ k \in \{1,2,\ldots,n\} $, 
$
  \mathbf{u} \in
  H^k
$, 
$ x \in H $, 
$ t \in [ 0 , T ] $ 
that 
\begin{equation}
\label{eq:derivative_formulas}
\begin{split}
&
    \big(
    \tfrac{\partial^k}{\partial x^k}
    \trans
    \big)(t,x)
  \mathbf{u}
\\&=
  \sum\limits_{
    \varpi \in \Pi_k
  }
  \E
  \Big[
    \varphi^{(\#_\varpi)}(X_t^{0,x})
    \big(
      X_t^{\#_{I^\varpi_1},[(x,\mathbf{u})]_1^{\varpi}}
      ,
      X_t^{\#_{I^\varpi_2},[(x,\mathbf{u})]_2^{\varpi}}
      ,
      \dots
      ,
      X_t^{\#_{I^\varpi_{\#_\varpi}},[(x,\mathbf{u})]_{\#_\varpi}^\varpi}
    \big)
  \Big]
  ,
\end{split}
\end{equation}
\item
\label{item:lem.derivative.apriori}
it holds for all 
$ p \in (0,\infty) $, 
$ k \in \{1,2,\ldots,n\} $, 
$
\boldsymbol{\delta}=(\delta_1,\delta_2,\dots,\delta_k)\in[0,\nicefrac{1}{2})^k
$, 
$ \alpha \in [0,1) $, 
$ \beta \in [0,\nicefrac{1}{2}) $ 
with 
$
\sum_{i=1}^k\delta_i
< \nicefrac{1}{2}
$ 
that 
\begin{equation}
  \sup_{x\in H}
  \sup_{\mathbf{u}=(u_1,u_2,\ldots,u_k)\in(\nzspace{H})^k}
  \sup_{t\in(0,T]}
  \left[
    \frac{
    	t^{\iota^{\boldsymbol{\delta},\alpha,\beta}_\N}
    	\|X^{k,(x,\mathbf{u})}_t\|_{\lpn{p}{\P}{H}}
    	}{
    	  \prod^k_{i=1} \|u_i\|_{H_{-\delta_i}}
    	}
  \right]
  < \infty,
\end{equation}
\item
\label{item:thm.derivative.bound}
it holds for all 
$ k \in \{1,2,\ldots,n\} $, 
$
  \boldsymbol{\delta}=(\delta_1,\delta_2,\dots,\delta_k)\in[0,\nicefrac{1}{2})^k
$, 
$ \alpha \in [0,1) $, 
$ \beta \in [0,\nicefrac{1}{2}) $ 
with 
$
  \sum_{i=1}^k\delta_i
  < \nicefrac{1}{2}
$ 
that 
\begin{equation}
\label{eq:bound_derivative_formulas}
\begin{split}
&
  \sup_{ v \in H }
  \sup_{
    \mathbf{u}=(u_1, u_2, \dots, u_k) \in 
    (\nzspace{H})^k
  }
  \sup_{
    t \in (0,T]
  }
  \Bigg[
  \frac{
    t^{\sum^k_{i=1} \delta_i}
    \,
    \big\|
    \big(
    \tfrac{\partial^k}{\partial x^k}
    \trans
    \big)(t,v)
  \mathbf{u}
    \big\|_V
  }{
    \prod^k_{ i=1 }
    \|
      u_i
    \|_{
      H_{ - \delta_i }
    }
  }
  \Bigg]
\\&\leq
  |T \vee 1|^{
    \lfloor k/2 \rfloor \,
    \min\{1-\alpha,\nicefrac{1}{2}-\beta\}
  } \,
    \|
      \varphi
    \|_{
      \Cb{k}(H,{V})
    }
\\&\cdot
  \sum_{
    \varpi \in \Pi_k
  }
  \prod_{
    I \in \varpi
  }
    \sup_{x\in H}
    \sup_{
      \mathbf{u}=(u_i)_{i\in I}
      \in (\nzspace{H})^{\#_I}
    }
    \sup_{t\in(0,T]}
    \Bigg[
    \frac{
      t^{
      \iota^{\boldsymbol{\delta},\alpha,\beta}_I
      }
      \|
        X_t^{\#_I,(x,\mathbf{u})}
      \|_{\mathcal{L}^{\#_\varpi}(\P;H)}
    }
    {
      \prod_{i\in I}
      \|u_i\|_{H_{-\delta_i}}
    }
  \Bigg]
<\infty,
\end{split}
\end{equation}
\item
\label{item:lem.0.lip}
it holds for all 
$ p \in (0,\infty) $
that
\begin{equation}
\sup_{\substack{x,y\in H,\\ x\neq y}}
\sup_{t\in(0,T]}
\left[
\frac{
	\|X^{0,x}_t-X^{0,y}_t\|_{\lpn{p}{\P}{H}}
}{
\|x-y\|_H
}
\right]
< \infty,
\end{equation}
\item
\label{item:lem.derivative.lip}
it holds for all 
$ p \in (0,\infty) $, 
$ k \in \{1,2,\ldots,n\} $, 
$
\boldsymbol{\delta}=(\delta_1,\delta_2,\dots,\delta_k)\in[0,\nicefrac{1}{2})^k
$, 
$ \alpha \in [0,1) $, 
$ \beta \in [0,\nicefrac{1}{2}) $ 
with 
$
\sum_{i=1}^k\delta_i
< \nicefrac{1}{2}
$ 
and
$
|F|_{\operatorname{Lip}^k(H,H_{-\alpha})}
+
|B|_{\operatorname{Lip}^k(H,HS(U,H_{-\beta}))}
< \infty
$ 
that 
\begin{equation}
\sup_{\substack{x,y\in H,\\ x\neq y}}
\sup_{\mathbf{u}=(u_1,u_2,\ldots,u_k)\in(\nzspace{H})^k}
\sup_{t\in(0,T]}
\left[
\frac{
	t^{\iota^{(\boldsymbol{\delta},0),\alpha,\beta}_\N}
	\|X^{k,(x,\mathbf{u})}_t-X^{k,(y,\mathbf{u})}_t\|_{\lpn{p}{\P}{H}}
}{
\|x-y\|_H
\prod^k_{i=1} \|u_i\|_{H_{-\delta_i}}
}
\right]
< \infty,
\end{equation}
and 
\item
\label{item:trans.lip}
it holds for all 
$ k \in \{1,2,\ldots,n\} $, 
$
  \boldsymbol{\delta}=(\delta_1, \delta_2, \ldots, \delta_k) \in
  [0,\nicefrac{1}{2})^k
$,
$ \alpha \in [0,1) $, 
$ \beta \in [0,\nicefrac{1}{2}) $ 
with 
$
  \sum_{i=1}^k\delta_i
  < \nicefrac{1}{2}
$ 
and 
$
  |F|_{\operatorname{Lip}^k(H,H_{-\alpha})}
  +
  |B|_{\operatorname{Lip}^k(H,HS(U,H_{-\beta}))}
  +
  |\varphi|_{\operatorname{Lip}^k(H,V)} < \infty
$ 
that 
\begin{equation}
\label{eq:trans.lip.statement}
\begin{split}
&
  \sup_{\substack{v, w \in H, \\ v \neq w}}
  \sup_{ \mathbf{u}=(u_1,u_2,\ldots,u_k) \in (\nzspace{H})^k }
  \sup_{t \in (0,T] }
  \left[
  \frac{
    t^{\sum^k_{i=1} \delta_i}
    \big\|
  \big[
    \big(
    \tfrac{\partial^k}{\partial x^k}
    \trans
    \big)(t,v)
  -
    \big(
    \tfrac{\partial^k}{\partial x^k}
    \trans
    \big)(t,w)
  \big]
    \mathbf{u}
    \big\|_V    
  }{
    \|v-w\|_H
    \prod^k_{i=1} \|u_i\|_{H_{-\delta_i}}
  }
  \right]
\\&\leq
    |T\vee 1|^{
      \lceil k/2 \rceil \, 
      \min\{1-\alpha,\nicefrac{1}{2}-\beta\}
    } \,
      \|\varphi\|_{ \operatorname{Lip}^k(H,V) }
\\&\cdot
    \sum_{
      \varpi \in \Pi_k
    }
    \Bigg\{
    \sup_{\substack{x,y \in H, \\ x \neq y}}
    \sup_{ t \in (0,T] }
    \Bigg[
      \frac{\| X_t^{0,x} - X_t^{0,y} \|_{\lpn{\#_\varpi+1}{\P}{H}}}{\|x-y\|_H}
    \Bigg]
\\&\cdot
    \prod_{I\in\varpi} 
    \sup_{ x \in H }
    \sup_{ \mathbf{u}=(u_i)_{i\in I} \in (\nzspace{H})^{\#_I} }
    \sup_{ t \in (0,T] }
    \Bigg[
    \frac{
      t^{ \iota^{\boldsymbol{\delta},\alpha,\beta}_I } \,
      \|X_t^{\#_I,(x,\mathbf{u})}\|_{\lpn{\#_\varpi+1}{\P}{H}}
    }{
      \prod_{i\in I}
      \|u_i\|_{H_{-\delta_i}}
    }
    \Bigg]
\\&+
    \sum_{I\in\varpi}
    \sup_{\substack{x,y \in H, \\ x \neq y}}
    \sup_{ \mathbf{u}=(u_i)_{i\in I} \in (\nzspace{H})^{\#_I} }
    \sup_{ t \in (0,T] }
    \Bigg[
    \frac{
      t^{ \iota^{(\boldsymbol{\delta},0),\alpha,\beta}_{I\cup\{k+1\}} } \,
      \|X_t^{\#_I,(x,\mathbf{u})}-X_t^{\#_I,(y,\mathbf{u})}\|_{\lpn{\#_\varpi}{\P}{H}}
    }{
      \|x-y\|_H
      \prod_{i\in I}
      \|u_i\|_{H_{-\delta_i}}
    }
    \Bigg]
\\&\cdot
    \prod_{ J \in \varpi \setminus \{I\} } 
    \sup_{ x \in H }
    \sup_{ \mathbf{u}=(u_i)_{i\in J} \in (\nzspace{H})^{\#_J} }
    \sup_{ t \in (0,T] }
    \Bigg[
    \frac{
      t^{ \iota^{\boldsymbol{\delta},\alpha,\beta}_J } \,
      \|X_t^{\#_J,(x,\mathbf{u})}\|_{\lpn{\#_\varpi}{\P}{H}}
    }{
      \prod_{i\in J}
      \|u_i\|_{H_{-\delta_i}}
    }
    \Bigg]
    \Bigg\}
    < \infty.
\end{split}
\end{equation}
\end{enumerate}
\end{lemma}

\begin{proof}
Throughout this proof 
let 
$ \alpha \in [0,1) $, 
$ \beta \in [0,\nicefrac{1}{2}) $ 
and let $ \deltaset{k} \in \mathcal{P}(\R^k) $, 
$ k \in \N $, 
be the sets which satisfy for all 
$ k \in \N $ 
that 
$
  \deltaset{k}=
  \{ 
    (\delta_1,\delta_2,\ldots,\delta_k) \in [0,\nicefrac{1}{2})^k \colon 
    \sum^k_{i=1} \delta_i < \nicefrac{1}{2}
  \}
$. 
Note that H\"{o}lder's inequality shows that for all 
$ k \in \{1,2,\ldots,n\} $, 
$ \mathbf{u} = (u_0,u_1,\ldots,u_k) \in H^{k+1} $, 
$ t \in [0,T] $ 
it holds that 
\begin{equation}
\label{eq:basic.hoelder}
\begin{split}
&
    \sum_{ \varpi \in \Pi_k }
    \E
    \Big[
      \big\|
      \varphi^{(\#_\varpi)}(X_t^{0,u_0})
      \big(
        X_t^{\#_{I^\varpi_1}, [\mathbf{u}]_1^{\varpi}}
        ,
        X_t^{\#_{I^\varpi_2}, [\mathbf{u}]_2^{\varpi}}
        ,
        \dots
        ,
        X_t^{\#_{I^\varpi_{\#_\varpi}}, [\mathbf{u}]_{\#_{\varpi}}^\varpi}
      \big)
      \big\|_V
    \Big]
\\&\leq
  \sum_{
    \varpi \in \Pi_k
  }
  |\varphi|_{\Cb{\#_\varpi}(H,V)}
  \prod^{\#_\varpi}_{i=1}
    \|X^{\#_{I^\varpi_i}, [\mathbf{u}]^\varpi_i}_t\|_{\lpn{\#_\varpi}{\P}{H}}
    .
\end{split}
\end{equation}
This, 
the assumption that 
$ \varphi \in \Cb{n}(H,V) $, 
and the assumption that 
$
  \forall \, k \in \{1,2,\ldots,n\},\,
  \mathbf{u} \in H^{k+1}, \,
  p \in (0,\infty)
  \colon
  \sup_{ t \in [0,T] }
  \E\big[\|X^{k,\mathbf{u}}_t\|^p_H\big] 
  < \infty
$ 
establish item~\eqref{item:trans.derivative.integrability}. 
Moreover, note that~\eqref{eq:basic.hoelder} implies that for all 
$ k \in \{1,2,\ldots,n\} $, 
$ \delta_1, \delta_2, \ldots, \delta_k \in [0,\infty) $, 
$ \mathbf{u} = (u_1,u_2,\ldots,u_k) \in (\nzspace{H})^k $, 
$ x \in H $, 
$ t \in [0,T] $ 
it holds that 
\begin{equation}
\label{eq:trans.holder.bound}
\begin{split}
&
  \frac{
    1
  }{
    \prod^k_{i=1} \|u_i\|_{ H_{-\delta_i} }
  } \,
    \sum_{
      \varpi \in \Pi_k
    }
    \E
    \Big[
      \big\|
      \varphi^{(\#_\varpi)}(X_t^{0,x})
      \big(
        X_t^{\#_{I^\varpi_1}, [(x,\mathbf{u})]_1^{\varpi}}
        ,
        X_t^{\#_{I^\varpi_2}, [(x,\mathbf{u})]_2^{\varpi}}
        ,
        \dots
        ,
        X_t^{\#_{I^\varpi_{\#_\varpi}}, [(x,\mathbf{u})]_{\#_{\varpi}}^\varpi}
      \big)
      \big\|_V
    \Big]
\\&\leq
  \sum_{
    \varpi \in \Pi_k
  }
  |\varphi|_{\Cb{\#_\varpi}(H,V)}
  \prod^{\#_\varpi}_{i=1}
  \frac{
    \|X^{\#_{I^\varpi_i}, [(x,\mathbf{u})]^\varpi_i}_t\|_{\lpn{\#_\varpi}{\P}{H}}
  }{
    \prod^{\#_{I^\varpi_i}}_{j=1} 
    \|u_{I^\varpi_{i,j}}\|_{ H_{-\delta_{I^\varpi_{i,j}}} }
  }.
\end{split}
\end{equation}
In addition, \eqref{eq:SEE.derivative} and item~(ii) of Theorem~2.1 
in~\cite{AnderssonJentzenKurniawan2016a} 
(with 
$ T = T $, 
$ \eta = \eta $, 
$ H = H $, 
$ U = U $, 
$ W = W $, 
$ A = A $, 
$ n = n $, 
$ F = F $, 
$ B = B $, 
$ \alpha = 0 $, 
$ \beta = 0 $, 
$ k = k $, 
$ p = p $, 
$ \boldsymbol{\delta} = (0,0,\ldots,0) \in \R^k $ 
for 
$ p \in [2,\infty) $, 
$ k \in \{1,2,\ldots,n\} $
in the notation of Theorem~2.1 
in~\cite{AnderssonJentzenKurniawan2016a})
ensure that for all 
$ k \in \{1,2,\ldots,n\} $, 
$ p \in [2,\infty) $, 
$ t \in [0,T] $
it holds that 
\begin{equation}
\label{eq:itemii}
  \sup_{ \mathbf{u} = (u_0,u_1,\ldots,u_k) \in H \times (\nzspace{H})^k }
  \Bigg[
  \frac{
    \|X^{k,\mathbf{u}}_t\|_{\lpn{p}{\P}{H}}
  }{
    \prod^k_{ i=1 }
    \|u_i\|_H
  }
  \Bigg]
  < \infty.
\end{equation}
This, Jensen's inequality, and~\eqref{eq:trans.holder.bound} 
(with 
$ k = k $, 
$ \delta_1 = 0 $,
$ \delta_2 = 0, \ldots$, 
$ \delta_k = 0 $ 
for $ k \in \{1,2,\ldots,n\} $
in the notation of~\eqref{eq:trans.holder.bound})
imply that for all 
$ k \in \{1,2,\ldots,n\} $, 
$ t \in [0,T] $ 
it holds that 
\begin{equation}
\label{eq:trans.derivative.holder}
\begin{split}
&
  \sup_{ x \in H }
  \sup_{ \mathbf{u}=(u_1,u_2,\ldots,u_k) \in (\nzspace{H})^k }
  \bigg(
  \tfrac{
    1
  }{
    \prod^k_{i=1} \|u_i\|_H
  }
\\&\cdot
    \smallsum\limits_{
      \varpi \in \Pi_k
    }
    \E
    \Big[
      \big\|
      \varphi^{(\#_\varpi)}(X_t^{0,x})
      \big(
        X_t^{\#_{I^\varpi_1}, [(x,\mathbf{u})]_1^{\varpi}}
        ,
        X_t^{\#_{I^\varpi_2}, [(x,\mathbf{u})]_2^{\varpi}}
        ,
        \dots
        ,
        X_t^{\#_{I^\varpi_{\#_\varpi}}, [(x,\mathbf{u})]_{\#_{\varpi}}^\varpi}
      \big)
      \big\|_V
    \Big]
  \bigg)
\\&\leq
  {\sum_{
    \varpi \in \Pi_k
  }}
  |\varphi|_{\Cb{\#_\varpi}(H,V)}
  {\prod_{ I \in \varpi }}
  \sup_{ x \in H }
  \sup_{ \mathbf{u} = (u_i)_{i\in I} \in (\nzspace{H})^{\#_I} }
  \Bigg(
  \frac{
    \|X^{\#_I,(x,\mathbf{u})}_t\|_{\lpn{\#_\varpi}{\P}{H}}
  }{
    \prod_{ i \in I }
    \|u_i\|_H
  }
  \Bigg)
  < \infty.
\end{split}
\end{equation}
Furthermore, \eqref{eq:SEE.derivative} and item~(iii) of Theorem~2.1 in~\cite{AnderssonJentzenKurniawan2016a} 
(with 
$ T = T $, 
$ \eta = \eta $, 
$ H = H $, 
$ U = U $, 
$ W = W $, 
$ A = A $, 
$ n = n $, 
$ F = F $, 
$ B = B $, 
$ \alpha = 0 $, 
$ \beta = 0 $, 
$ k = k $, 
$ p = p $, 
$ x = x $
for 
$ x \in H $, 
$ p \in [2,\infty) $, 
$ k \in \{1,2,\ldots,n\} $
in the notation of Theorem~2.1 
in~\cite{AnderssonJentzenKurniawan2016a}) 
assure that for all 
$ k \in \{1,2,\ldots,n\} $, 
$ p \in [2,\infty) $, 
$ x \in H $, 
$ t \in [0,T] $ 
it holds that 
\begin{equation}
\label{eq:klinear}
  \big(
    H^k \ni \mathbf{u} \mapsto
    [ X^{k,(x,\mathbf{u})}_t ]_{ \P, \mathcal{B}(H) } \in \lpnb{p}{\P}{H}
  \big)
  \in
  L^{(k)}( H, \lpnb{p}{\P}{H} )
  .
\end{equation}
This and~\eqref{eq:trans.derivative.holder} establish that for all 
$ k \in \{1,2,\ldots,n\} $, 
$ x \in H $, 
$ t \in [0,T] $ 
it holds that 
\begin{multline}
\label{eq:phi.klinear}
    \Big(
      H^k
      \ni \mathbf{u} \mapsto
  \smallsum_{
    \varpi \in \Pi_k
  }
  \E
  \Big[
    \varphi^{(\#_\varpi)}(X_t^{0,x})
    \big(
        X_t^{\#_{I^\varpi_1}, [(x,\mathbf{u})]_1^{\varpi}}
        ,
        X_t^{\#_{I^\varpi_2}, [(x,\mathbf{u})]_2^{\varpi}}
        ,
        \dots
        ,
        X_t^{\#_{I^\varpi_{\#_\varpi}}, [(x,\mathbf{u})]_{\#_{\varpi}}^\varpi}
    \big)
  \Big]
  \in V
    \Big)
\\
    \in L^{(k)}(H,V)
    .
\end{multline}
In the next step we prove that for all 
$ k \in \{1,2,\ldots,n\} $, 
$ t \in [0,T] $ 
it holds that 
\begin{multline}
\label{eq:trans.derivative.continuity}
  \Big(
    H \ni x \mapsto
    \Big(
      H^k
      \ni \mathbf{u} \mapsto
  \smallsum_{
    \varpi \in \Pi_k
  }
  \E
  \Big[
    \varphi^{(\#_\varpi)}(X_t^{0,x})
    \big(
        X_t^{\#_{I^\varpi_1}, [(x,\mathbf{u})]_1^{\varpi}}
        ,
        X_t^{\#_{I^\varpi_2}, [(x,\mathbf{u})]_2^{\varpi}}
        ,
        \dots
        ,
\\
        X_t^{\#_{I^\varpi_{\#_\varpi}}, [(x,\mathbf{u})]_{\#_{\varpi}}^\varpi}
    \big)
  \Big]
  \in V
    \Big)
    \in L^{(k)}(H,V)
  \Big)
  \in \Cp{}{}(H,L^{(k)}(H,V))
  .
\end{multline}
For this we observe that the triangle inequality and H\"{o}lder's inequality  
imply that for all 
$ k \in \{1,2,\ldots,n\} $, 
$ \delta_1, \delta_2, \ldots, \delta_k \in [0,\infty) $, 
$ \mathbf{u}=(u_1,u_2,\ldots,u_k) \in (\nzspace{H})^k $, 
$ x, y \in H $, 
$ t \in [ 0 , T ] $ 
it holds that 
\begin{equation}
\label{eq:trans.continuity.holder}
\begin{split}
&
    \tfrac{1}{
      \prod^k_{ i=1 } \|u_i\|_{ H_{-\delta_i} }
    } \,
      \bigg\|
    \smallsum\limits_{
      \varpi \in \Pi_k
    }
    \E
    \Big[
      \varphi^{(\#_\varpi)}(X_t^{0,x})
      \big(
        X_t^{\#_{I^\varpi_1}, [(x,\mathbf{u})]_1^{\varpi}}
        ,
        X_t^{\#_{I^\varpi_2}, [(x,\mathbf{u})]_2^{\varpi}}
        ,
        \dots
        ,
        X_t^{\#_{I^\varpi_{\#_\varpi}}, [(x,\mathbf{u})]_{\#_{\varpi}}^\varpi}
      \big)
\\&\quad-
      \varphi^{(\#_\varpi)}(X_t^{0,y})
      \big(
        X_t^{\#_{I^\varpi_1}, [(y,\mathbf{u})]_1^{\varpi}}
        ,
        X_t^{\#_{I^\varpi_2}, [(y,\mathbf{u})]_2^{\varpi}}
        ,
        \dots
        ,
        X_t^{\#_{I^\varpi_{\#_\varpi}}, [(y,\mathbf{u})]_{\#_{\varpi}}^\varpi}
      \big)
    \Big]
      \bigg\|_V
\\&\leq
    \tfrac{1}{
      \prod^k_{ i=1 } \|u_i\|_{ H_{-\delta_i} }
    }
    \smallsum\limits_{
      \varpi \in \Pi_k
    }
    \Big(
    \E
    \Big[
      \big\|
      \big[\varphi^{(\#_\varpi)}(X_t^{0,x})
      -
      \varphi^{(\#_\varpi)}(X_t^{0,y})\big]
      \big(
        X_t^{\#_{I^\varpi_1}, [(x,\mathbf{u})]_1^{\varpi}}
        ,
        X_t^{\#_{I^\varpi_2}, [(x,\mathbf{u})]_2^{\varpi}}
        ,
        \dots
        ,
\\&\quad
        X_t^{\#_{I^\varpi_{\#_\varpi}}, [(x,\mathbf{u})]_{\#_{\varpi}}^\varpi}
      \big)
      \big\|_V
    \Big]
+
    \E
    \Big[
      \big\|
      \varphi^{(\#_\varpi)}(X_t^{0,y})
      \big(
        X_t^{\#_{I^\varpi_1}, [(x,\mathbf{u})]_1^{\varpi}}
        ,
        X_t^{\#_{I^\varpi_2}, [(x,\mathbf{u})]_2^{\varpi}}
        ,
        \dots
        ,
        X_t^{\#_{I^\varpi_{\#_\varpi}}, [(x,\mathbf{u})]_{\#_{\varpi}}^\varpi}
      \big)
\\&\quad-
      \varphi^{(\#_\varpi)}(X_t^{0,y})
      \big(
        X_t^{\#_{I^\varpi_1}, [(y,\mathbf{u})]_1^{\varpi}}
        ,
        X_t^{\#_{I^\varpi_2}, [(y,\mathbf{u})]_2^{\varpi}}
        ,
        \dots
        ,
        X_t^{\#_{I^\varpi_{\#_\varpi}}, [(y,\mathbf{u})]_{\#_{\varpi}}^\varpi}
      \big)
      \big\|_V
    \Big]
    \Big)
\\&\leq
    \sum_{
      \varpi \in \Pi_k
    }
    \Bigg\{
    \|\varphi^{(\#_\varpi)}(X_t^{0,x})-\varphi^{(\#_\varpi)}(X_t^{0,y})\|_{\lpn{\#_\varpi+1}{\P}{L^{(\#_\varpi)}(H,V)}}
    \prod^{\#_\varpi}_{i=1} 
    \frac{
      \|X_t^{\#_{I^\varpi_i},[(x,\mathbf{u})]_i^{\varpi}}\|_{\lpn{\#_\varpi+1}{\P}{H}}
      }{
    \prod^{\#_{I^\varpi_i}}_{j=1} 
    \|u_{I^\varpi_{i,j}}\|_{ H_{-\delta_{I^\varpi_{i,j}}} }
      }
\\&\quad+
    |\varphi|_{\Cb{\#_\varpi}(H,V)}
    \sum^{\#_\varpi}_{i=1}
    \frac{\|X_t^{\#_{I^\varpi_i},[(x,\mathbf{u})]_i^{\varpi}}-X_t^{\#_{I^\varpi_i},[(y,\mathbf{u})]_i^{\varpi}}\|_{\lpn{\#_\varpi}{\P}{H}}}{
    \prod^{\#_{I^\varpi_i}}_{j=1} 
    \|u_{I^\varpi_{i,j}}\|_{ H_{-\delta_{I^\varpi_{i,j}}} }
    }
\\&\quad\cdot
    \Bigg[
    \prod^{i-1}_{j=1} 
    \frac{\|X_t^{\#_{I^\varpi_j},[(y,\mathbf{u})]_j^{\varpi}}\|_{\lpn{\#_\varpi}{\P}{H}}    }{
    \prod^{\#_{I^\varpi_j}}_{l=1} 
    \|u_{I^\varpi_{j,l}}\|_{ H_{-\delta_{I^\varpi_{j,l}}} }
    }
    \Bigg] \,
    \Bigg[
    \prod^{\#_\varpi}_{j=i+1} 
    \frac{\|X_t^{\#_{I^\varpi_j},[(x,\mathbf{u})]_j^{\varpi}}\|_{\lpn{\#_\varpi}{\P}{H}}    }{
    \prod^{\#_{I^\varpi_j}}_{l=1} 
    \|u_{I^\varpi_{j,l}}\|_{ H_{-\delta_{I^\varpi_{j,l}}} }
    }
    \Bigg]
    \Bigg\}
    .
\end{split}
\end{equation}
Next we note that \eqref{eq:SEE.derivative} and item~(iii) of Corollary~2.10 in~\cite{AnderssonJentzenKurniawan2016arXiv} 
(with 
$ H = H $, 
$ U = U $, 
$ T = T $, 
$ \eta = \eta $, 
$ \alpha = 0 $, 
$ \beta = 0 $, 
$ W = W $, 
$ A = A $, 
$ F = F $, 
$ B = B $, 
$ p = p $, 
$ \delta = 0 $
for 
$ p \in [2,\infty) $ 
in the notation of Corollary~2.10 in~\cite{AnderssonJentzenKurniawan2016arXiv})
ensure that for all 
$ p \in [2,\infty) $ 
it holds that 
\begin{equation}
\label{eq:cor2.10}
  \sup_{\substack{ x,y \in H, \, x \neq y }} \,
  \sup_{ t \in [0,T] }
  \frac{
    \|X^{0,x}_t - X^{0,y}_t\|_{\lpn{p}{\P}{H}}
  }{ \|x-y\|_H }
  < \infty.
\end{equation}
This implies that for all 
$ x \in H $, 
$ t \in [0,T] $ 
it holds that 
\begin{equation}
  \limsup\nolimits_{ H \ni y \to x }
  \E[\min\{1,
  \| X^{0,x}_t - X^{0,y}_t \|_H
  \}]
  =0.
\end{equation}
The fact that
$ \forall \, k \in \{1,2,\ldots,n\} \colon \varphi^{(k)} \in \Cp{}{}(H,L^{(k)}(H,V)) $ 
therefore assures that for all 
$ k \in \{1,2,\ldots,n\} $, 
$ x \in H $, 
$ t \in [ 0 , T ] $ 
it holds that 
\begin{equation}
  \limsup\nolimits_{ H \ni y \to x }
  \E[\min\{1,
  \|\varphi^{(k)}(X^{0,x}_t)-\varphi^{(k)}(X^{0,y}_t)\|_{ L^{(k)}(H,V) }
  \}]
  =0.
\end{equation}
Combining this and, e.g., Lemma~4.2 in 
Hutzenthaler et al.~\cite{HutzenthalerJentzenSalimova2016arXiv} 
(with 
$ I = \{\emptyset\} $, 
$ 
  (\Omega,\mathcal{F},\P) =
  (\Omega,\mathcal{F},\P) 
$, 
$ c=1 $, 
$
  X^m(\emptyset,\omega) = 
  \|\varphi^{(k)}(X^{0,y_m}_t(\omega))-\varphi^{(k)}(X^{0,y_0}_t(\omega))\|_{ L^{(k)}(H,V) }
$
for 
$ \omega \in \Omega $, 
$ t \in [0,T] $, 
$ m \in \N $, 
$ k \in \{1,2,\ldots,n\} $, 
$
  (y_l)_{ l \in \N_0 }
  \in 
  \{
    v \in \mathbb{M}( \N_0, H )
    \colon
    \limsup_{ l \to \infty }
    \|v_l-v_0\|_H
    =0
  \}
$
in the notation of Lemma~4.2 in 
Hutzenthaler et al.~\cite{HutzenthalerJentzenSalimova2016arXiv}) 
establishes that for all 
$ k \in \{1,2,\ldots,n\} $, 
$ \varepsilon \in (0,\infty) $, 
$ t \in [0,T] $
and all sequences 
$ (y_m)_{m\in\N_0} \subseteq H $ 
with 
$
    \limsup_{ m \to \infty }
    \|y_m-y_0\|_H
    =0
$ 
it holds that 
\begin{equation}
  \limsup\nolimits_{ m \to \infty }
  \P\big(
    \|\varphi^{(k)}(X^{0,y_m}_t)-\varphi^{(k)}(X^{0,y_0}_t)\|_{ L^{(k)}(H,V) }
    \geq \varepsilon
  \big)
  =0.
\end{equation}
Combining this, 
the fact that 
$
  \forall \, k \in \{1,2,\ldots,n\}
  \colon
  \sup_{ x \in H }
  \|\varphi^{(k)}(x)\|_{ L^{(k)}(H,V) }
  < \infty
$,
and, e.g., Proposition~4.5 in 
Hutzenthaler et al.~\cite{HutzenthalerJentzenSalimova2016arXiv} 
(with 
$ I = \{\emptyset\} $, 
$ 
  (\Omega,\mathcal{F},\P) =
  (\Omega,\mathcal{F},\P) 
$, 
$ p = p $, 
$ V = \R $, 
$
  X^m(\emptyset,\omega) = 
  \|\varphi^{(k)}( X^{0,y_m}_t(\omega) )-\varphi^{(k)}( X^{0,y_0}_t(\omega) )\|_{L^{(k)}(H,V)}
$
for 
$ \omega \in \Omega $, 
$ t \in [0,T] $, 
$ p \in [2,\infty) $, 
$ k \in \{1,2,\ldots,n\} $, 
$ m \in \N_0 $, 
$
  (y_l)_{ l \in \N_0 }
  \in 
  \{
    v \in \mathbb{M}( \N_0, H )
    \colon
    \limsup_{ l \to \infty }
    \|v_l-v_0\|_H
    =0
  \}
$
in the notation of Proposition~4.5 in 
Hutzenthaler et al.~\cite{HutzenthalerJentzenSalimova2016arXiv})
therefore shows that for all 
$ k \in \{1,2,\ldots,n\} $, 
$ p \in [2,\infty) $, 
$ t \in [ 0 , T ] $ 
and all sequences 
$
  (y_m)_{m \in \N_0} \subseteq H
$ 
with 
$ \limsup_{ m \to \infty } \|y_m-y_0\|_H = 0 $
it holds that 
\begin{equation}
\label{eq:trans.continuity.vitali}
  \limsup\nolimits_{ m \to \infty }
    \E\Big[\|\varphi^{(k)}(X_t^{0,y_m})-\varphi^{(k)}(X_t^{0,y_0})\|^p_{L^{(k)}(H,V)}\Big]
  =0.
\end{equation}
Furthermore, \eqref{eq:SEE.derivative} and item~(v) of Theorem~2.1 in~\cite{AnderssonJentzenKurniawan2016a} 
(with 
$ T = T $, 
$ \eta = \eta $, 
$ H = H $, 
$ U = U $, 
$ W = W $, 
$ A = A $, 
$ n = n $, 
$ F = F $, 
$ B = B $, 
$ \alpha = 0 $, 
$ \beta = 0 $, 
$ k = k $, 
$ p = p $
for 
$ p \in [2,\infty) $, 
$ k \in \{1,2,\ldots,n\} $ 
in the notation of Theorem~2.1 
in~\cite{AnderssonJentzenKurniawan2016a}) 
ensure that 
for all 
$ k \in \{1,2,\ldots,n\} $, 
$ p \in [2,\infty) $, 
$ t \in [0,T] $
it holds that 
\begin{multline}
\label{item:itemv}
  \big(
  H \ni x \mapsto
  \big(
    H^k \ni \mathbf{u} \mapsto
    [X^{k,(x,\mathbf{u})}_t]_{ \P, \mathcal{B}(H) } \in \lpnb{p}{\P}{H}
  \big)
  \in L^{(k)}(H, \lpnb{p}{\P}{H})
  \big) 
\\
  \in \Cp{}{}( H, L^{(k)}(H, \lpnb{p}{\P}{H}) )
  .
\end{multline}
Combining~\eqref{eq:trans.continuity.holder} 
(with 
$ k = k $, 
$ \delta_1 = 0 $, 
$ \delta_2 = 0,\ldots,$
$ \delta_k = 0 $ 
for $ k \in \{1,2,\ldots,n\} $
in the notation of~\eqref{eq:trans.continuity.holder})
with~\eqref{eq:itemii}, 
\eqref{eq:trans.continuity.vitali}, \eqref{item:itemv}, and Jensen's inequality
yields that for all 
$ k \in \{1,2,\ldots,n\} $, 
$ x \in H $, 
$ t \in [ 0 , T ] $ 
it holds that 
\begin{equation}
\begin{split}
&
    \limsup_{ H \ni y \to x }
    \sup_{ \mathbf{u} = (u_1,u_2,\ldots,u_k) \in (\nzspace{H})^k }
    \bigg(
    \tfrac{1}{
      \prod^k_{ i=1 } \|u_i\|_H
    }
\\&\cdot
      \bigg\|
    \smallsum\limits_{
      \varpi \in \Pi_k
    }
    \E
    \Big[
      \varphi^{(\#_\varpi)}(X_t^{0,x})
      \big(
        X_t^{\#_{I^\varpi_1}, [(x,\mathbf{u})]_1^{\varpi}}
        ,
        X_t^{\#_{I^\varpi_2}, [(x,\mathbf{u})]_2^{\varpi}}
        ,
        \dots
        ,
        X_t^{\#_{I^\varpi_{\#_\varpi}}, [(x,\mathbf{u})]_{\#_{\varpi}}^\varpi}
      \big)
\\&-
      \varphi^{(\#_\varpi)}(X_t^{0,y})
      \big(
        X_t^{\#_{I^\varpi_1}, [(y,\mathbf{u})]_1^{\varpi}}
        ,
        X_t^{\#_{I^\varpi_2}, [(y,\mathbf{u})]_2^{\varpi}}
        ,
        \dots
        ,
        X_t^{\#_{I^\varpi_{\#_\varpi}}, [(y,\mathbf{u})]_{\#_{\varpi}}^\varpi}
      \big)
    \Big]
      \bigg\|_V
      \bigg)
\\&\leq
    \sum_{
      \varpi \in \Pi_k
    }
    \Bigg\{\bigg(
    \limsup_{ H \ni y \to x }
    \Big[
    \|\varphi^{(\#_\varpi)}(X_t^{0,x})-\varphi^{(\#_\varpi)}(X_t^{0,y})\|_{\lpn{\#_\varpi+1}{\P}{L^{(\#_\varpi)}(H,V)}}
    \Big]
    \bigg)
\\&\cdot
    \Bigg(
    \prod_{ I \in \varpi }
    \sup_{ \mathbf{u} = (u_i)_{ i \in I } \in (\nzspace{H})^{\#_I} }
    \Bigg[
    \frac{
      \|X_t^{\#_I,(x,\mathbf{u})}\|_{\lpn{\#_\varpi+1}{\P}{H}}
      }{
    \prod_{ i \in I } 
    \|u_i\|_H
      }
    \Bigg]
    \Bigg)
\\&+
    |\varphi|_{\Cb{\#_\varpi}(H,V)}
    \sum_{ I \in \varpi }
    \Bigg(
    \limsup_{ H \ni y \to x }
    \sup_{ \mathbf{u} = (u_i)_{ i \in I } \in (\nzspace{H})^{\#_I} }
    \Bigg[
    \frac{\|X_t^{\#_I,(x,\mathbf{u})}-X_t^{\#_I,(y,\mathbf{u})}\|_{\lpn{\#_\varpi}{\P}{H}}}{
    \prod_{ i \in I } 
    \|u_i\|_H
    }
    \Bigg]
    \Bigg)
\\&\cdot
    \Bigg(
    \prod_{ J \in \varpi \setminus \{I\} } 
    \sup_{ y \in H }
    \sup_{ \mathbf{u} = (u_i)_{ i \in J } \in (\nzspace{H})^{\#_J} }
    \Bigg[
    \frac{\|X_t^{\#_J,(y,\mathbf{u})}\|_{\lpn{\#_\varpi}{\P}{H}}    }{
    \prod_{ i \in J } 
    \|u_i\|_H
    }
    \Bigg]
    \Bigg)
    \Bigg\}
    =0.
\end{split}
\end{equation}
This proves~\eqref{eq:trans.derivative.continuity}.
Next we claim that for all 
$ k \in \{1,2,\ldots,n\} $, 
$
  \mathbf{u} \in
  H^k
$, 
$ x \in H $, 
$ t \in [ 0 , T ] $ 
it holds that 
$
  \big( H \ni y \mapsto \trans(t,y) \in V \big) 
  \in \Cb{k}(H,V)
$ 
and 
\begin{equation}
\label{items:trans.regularity}
\begin{split}
    \big(
    \tfrac{\partial^k}{\partial x^k}
    \trans
    \big)(t,x)
  \mathbf{u}
  =
  \smallsum\limits_{
    \varpi \in \Pi_k
  }
  \E
  \Big[
    \varphi^{(\#_\varpi)}(X_t^{0,x})
    \big(
        X_t^{ \#_{I^\varpi_1}, [ (x,\mathbf{u}) ]_1^{ \varpi } }
        ,
        X_t^{ \#_{I^\varpi_2}, [ (x,\mathbf{u}) ]_2^{ \varpi } }
        ,
        \dots
        ,
        X_t^{ \#_{I^\varpi_{\#_\varpi}}, [(x,\mathbf{u})]_{ \#_\varpi }^{ \varpi } }
    \big)
  \Big].
\end{split}
\end{equation}

We now prove~\eqref{items:trans.regularity} by induction on 
$k\in\{1,2,\ldots,n\}$. 
For the base case $ k = 1 $ 
we note that~\eqref{eq:SEE.derivative}, Jensen's inequality, and items~(ix)--(x) of Theorem~2.1 
in~\cite{AnderssonJentzenKurniawan2016a} 
(with 
$ T = T $, 
$ \eta = \eta $, 
$ H = H $, 
$ U = U $, 
$ W = W $, 
$ A = A $, 
$ n = n $, 
$ F = F $, 
$ B = B $, 
$ \alpha = 0 $, 
$ \beta = 0 $, 
$ p = p $, 
$ t = t $
for 
$ t \in [0,T] $,  
$ p \in [2,\infty) $
in the notation of Theorem~2.1 
in~\cite{AnderssonJentzenKurniawan2016a}) 
ensure that for all 
$ p \in [1,\infty) $, 
$ x, u_1 \in H $, 
$ t \in [0,T] $
it holds that 
$
  \big(
    H \ni y \mapsto
    [X^{0,y}_t]_{ \P, \mathcal{B}(H) } \in \lpnb{p}{\P}{H}
  \big)
  \in \Cp{1}{}( H, \lpnb{p}{\P}{H} )
$ 
and 
\begin{equation}
\label{eq:itemvii}
  \big(
    \tfrac{d}{d x}
    [X^{0,x}_t]_{\P,\mathcal{B}(H)}
  \big) u_1
  = 
  [X^{1,(x,u_1)}_t]_{\P,\mathcal{B}(H)}
  .
\end{equation}
Lemma~\ref{lem:stoch.chain.rule} 
(with 
$ U = H $, 
$ V = H $, 
$ W = V $, 
$ (\Omega, \mathcal{F}, \P) = (\Omega, \mathcal{F}, \P) $, 
$ X^{m,\mathbf{u}} = X^{m,\mathbf{u}}_t $, 
$ \varphi  = \varphi $ 
for 
$ t \in [0,T] $, 
$ \mathbf{u} \in H^{m+1} $, $ m \in \{0,1\} $
in the notation of Lemma~\ref{lem:stoch.chain.rule})
therefore implies that for all 
$ x, u \in H $, 
$ t \in [0,T] $ 
it holds that 
\begin{equation}
  \big(
    H \ni y \mapsto
    \trans(t,y)=\E[\varphi(X^{0,y}_t)] \in V
  \big) 
  \in \Cp{1}{}(H,V)
\end{equation}
and 
\begin{equation}
    \big(
    \tfrac{\partial}{\partial x}
    \trans\big)(t,x)
  u
  =
  \E[ \varphi'(X^{0,x}_t) X^{1,(x,u)}_t ]
  .
\end{equation}
This and~\eqref{eq:trans.derivative.holder} 
prove~\eqref{items:trans.regularity} in the base case $k=1$.
For the induction step 
$ \{1,2,\ldots,n-1\} \ni k \to k+1 \in \{2,3,\ldots,n\} $ 
assume that there exists a natural number 
$ k \in \{1,2,\ldots,n-1\} $
such that~\eqref{items:trans.regularity} holds for 
$k=1$, $k=2,\ldots,k=k$, 
let 
$ \Phi_m \colon H^{m+1} \to V $, 
$ m \in \{1,2,\ldots,k\} $, 
be the functions which satisfy for all 
$ m \in \{1,2,\ldots,k\} $, 
$ \mathbf{u}=(u_1,u_2,\ldots,u_{m+1}) \in H^{m+1} $ 
that 
$
  \Phi_m(\mathbf{u})
  =
  \varphi^{(m)}(u_{m+1})(u_1,u_2,\ldots,u_m)
$, 
and let 
$ \mathbf{Y}^{ m,\mathbf{v}, \varpi, \mathbf{u} } \colon [0,T] \times \Omega \to H^{\#_\varpi+1} $, 
$ \mathbf{u} \in H^k $, 
$ \varpi \in \Pi_k $, 
$ \mathbf{v} \in H^{m+1} $, 
$ m \in \{0,1\} $, 
be the stochastic processes which satisfy for all 
$ \varpi \in \Pi_k $, 
$ \mathbf{u} \in H^k $, 
$ x, h \in H $, 
$ t \in [0,T] $ 
that 
\begin{equation}
  \mathbf{Y}^{ 0,x, \varpi, \mathbf{u} }_t
  =
  \big(
        X_t^{\#_{I^\varpi_1}, [(x,\mathbf{u})]_1^{\varpi}}
        ,
        X_t^{\#_{I^\varpi_2}, [(x,\mathbf{u})]_2^{\varpi}}
        ,
        \dots
        ,
        X_t^{\#_{I^\varpi_{\#_\varpi}}, [(x,\mathbf{u})]_{\#_{\varpi}}^\varpi}
        ,
  X^{0,x}_t
  \big) 
\end{equation}
and 
\begin{equation}
  \mathbf{Y}^{ 1,(x,h), \varpi, \mathbf{u} }_t
  =
  \big(
  X^{\#_{I^\varpi_1}+1,([(x,\mathbf{u})]^\varpi_{1},h)}_t,
  X^{\#_{I^\varpi_2}+1,([(x,\mathbf{u})]^\varpi_{2},h)}_t,
  \ldots,
  X^{\#_{I^\varpi_{\#_\varpi}}+1,([(x,\mathbf{u})]^\varpi_{\#_\varpi},h)}_t,
  X^{1,(x,h)}_t\big). 
\end{equation}
Next note that Lemma~\ref{lem:test.derivative} 
(with 
$ V = H $, 
$ W = V $, 
$ n = m $, 
$ \varphi = \varphi $, 
$ \Phi = \Phi_m $
for $ m \in \{1,2,\ldots,k\} $
in the notation of Lemma~\ref{lem:test.derivative})
shows that for all 
$ m \in \{1,2,\ldots,k\} $, 
$ \mathbf{u}=(u_1,u_2,\ldots,u_{m+1}) $, 
$ \mathbf{\tilde{u}}=(\tilde{u}_1,\tilde{u}_2,\ldots,\tilde{u}_{m+1}) \in H^{m+1} $ 
it holds that 
$
  \Phi_m \in \Cp{1}{}(H^{m+1},V)
$ 
and 
\begin{equation}
\label{eq:test.derivative}
\begin{split}
  \Phi'_m(\mathbf{u}) \mathbf{\tilde{u}}
&=
  \varphi^{(m+1)}(u_{m+1})(u_1,u_2,\ldots,u_m,\tilde{u}_{m+1})
\\&\quad+
  {\smallsum^m_{i=1}}
  \varphi^{(m)}(u_{m+1})(u_1,u_2,\ldots,u_{i-1},\tilde{u}_i,u_{i+1},u_{i+2},\ldots,u_m)
  .
\end{split}
\end{equation}	
This and H\"{o}lder's inequality imply that for all 
$ m \in \{1,2,\ldots,k\} $, 
$
  \mathbf{u}=(u_1,u_2,\ldots,u_{m+1}), \,
  \mathbf{\tilde{u}}=(\tilde{u}_1,\tilde{u}_2,\ldots,\tilde{u}_{m+1})
  \in H^{m+1}
$
it holds that 
\begin{equation}
\begin{split}
&
  \|\Phi'_m(\mathbf{u}) \mathbf{\tilde{u}}\|_V
\\&\leq
  |\varphi|_{\Cb{m+1}(H,V)} \,
  \|\tilde{u}_{m+1}\|_H
  \smallprod^m_{i=1} \|u_i\|_H
+
  \smallsum^m_{i=1}
  |\varphi|_{\Cb{m}(H,V)} \,
  \|\tilde{u}_i\|_H
  \smallprod_{ j \in \{1,2,\ldots,m\}\setminus\{i\} } 
  \|u_j\|_H
\\&\leq
  \|\mathbf{\tilde{u}}\|_{H^{m+1}}
  \big[
  |\varphi|^2_{\Cb{m+1}(H,V)}
  \smallprod^m_{i=1} \|u_i\|^2_H
+
  |\varphi|^2_{\Cb{m}(H,V)}
  \smallsum^m_{i=1}
  \smallprod_{ j \in \{1,2,\ldots,m\}\setminus\{i\} } 
  \|u_j\|^2_H  
  \big]^{1/2}
  .
\end{split}
\end{equation}
Hence, we obtain that for all 
$ m \in \{1,2,\ldots,k\} $, 
$
  \mathbf{u}=(u_1,u_2,\ldots,u_{m+1})
  \in H^{m+1}
$
it holds that 
\begin{equation}
\begin{split}
&
  \|\Phi'_m(\mathbf{u})\|_{L(H^{m+1},V)}
\leq
  \|\varphi\|_{\Cb{m+1}(H,V)}
  \big[
  \smallprod^m_{i=1} \|u_i\|^2_H
+
  \smallsum^m_{i=1}
  \smallprod_{ j \in \{1,2,\ldots,m\}\setminus\{i\} } 
  \|u_j\|^2_H  
  \big]^{1/2}  
\\&\leq
  \|\varphi\|_{\Cb{m+1}(H,V)}
\\&\quad\cdot
  \big[
  \smallprod^m_{i=1}  |\!\max\{1,\|\mathbf{u}\|_{H^{m+1}}\}|^2
+
  \smallsum^m_{i=1}
  \smallprod_{ j \in \{1,2,\ldots,m\}\setminus\{i\} } 
  |\!\max\{1,\|\mathbf{u}\|_{H^{m+1}}\}|^2
  \big]^{1/2}
\\&\leq
  \sqrt{m+1} \,
  \|\varphi\|_{\Cb{m+1}(H,V)} \,
  |\!\max\{1,\|\mathbf{u}\|_{H^{m+1}}\}|^m
  .
\end{split}
\end{equation}
This shows that for all 
$ m \in \{1,2,\ldots,k\} $ 
it holds that 
\begin{equation}
\label{eq:polynomial.order}
  \sup_{ \mathbf{u} \in H^{m+1} }
  \frac{
    \| \Phi'_m(\mathbf{u}) \|_{ L(H^{m+1},V) }
  }{
    |\!\max\{ 1, \|\mathbf{u}\|_{H^{m+1}} \}|^m
  }
  < \infty.
\end{equation}
Next we note that for all 
$ m \in \N $, 
$ p \in [1,\infty) $, 
$ Y_1,Y_2,\ldots,Y_m \in \lpn{0}{\P}{H} $ 
it holds that 
\begin{equation}
\label{eq:product.lp}
\begin{split}
  \|(Y_1,Y_2,\ldots,Y_m)\|_{\lpn{p}{\P}{H^m}}
&=
  \big\|[
    {\smallsum^m_{i=1}}
    \|Y_i\|^2_H
  ]^{1/2}\big\|_{\lpn{p}{\P}{\R}}
  \leq
  \big\|
    {\smallsum^m_{i=1}}
    \|Y_i\|_H
  \big\|_{\lpn{p}{\P}{\R}}
\\&\leq
  {\smallsum^m_{i=1}}
  \|Y_i\|_{\lpn{p}{\P}{H}}
  .
\end{split}
\end{equation}
This shows that for all 
$ m \in \{0,1\} $, 
$ p \in [1,\infty) $, 
$ \varpi \in \Pi_k $, 
$ \mathbf{u} \in H^k $, 
$ \mathbf{v} \in H^{m+1} $, 
$ t \in [0,T] $
it holds that 
\begin{equation}
\label{eq:vector.integrability}
  \mathbf{Y}^{m,\mathbf{v},\varpi,\mathbf{u}}_t 
  \in \lpn{p}{\P}{H^{\#_\varpi+1}}.
\end{equation} 
Next observe that \eqref{eq:klinear}, \eqref{eq:product.lp}, and Jensen's inequality imply that for all 
$ p \in [1,\infty) $, 
$ \varpi \in \Pi_k $, 
$ \mathbf{u} \in H^k $, 
$ x \in H $, 
$ t \in [0,T] $ 
it holds that 
\begin{equation}
\label{eq:vector.bound.linear}
  \big(
    H \ni h \mapsto
    \big[ \mathbf{Y}^{1,(x,h),\varpi,\mathbf{u}}_t \big]_{ \P, \mathcal{B}(H^{\#_\varpi+1}) }
    \in \lpnb{p}{\P}{H^{\#_\varpi+1}}
  \big)
  \in L(H,\lpnb{p}{\P}{H^{\#_\varpi+1}})
  .
\end{equation}
Furthermore, we note that~\eqref{eq:SEE.derivative} and item~(vi) of Theorem~2.1 
in~\cite{AnderssonJentzenKurniawan2016a} 
(with 
$ T = T $, 
$ \eta = \eta $, 
$ H = H $, 
$ U = U $, 
$ W = W $, 
$ A = A $, 
$ n = n $, 
$ F = F $, 
$ B = B $, 
$ \alpha = 0 $, 
$ \beta = 0 $, 
$ k = m $, 
$ p = p $, 
$ x = x $
for 
$ x \in H $, 
$ p \in [2,\infty) $,  
$ m \in \{2,3,\ldots,k+1\} $ 
in the notation of Theorem~2.1 
in~\cite{AnderssonJentzenKurniawan2016a}) 
ensure that for all 
$ m \in \{2,3,\ldots,k+1\} $, 
$ p \in [2,\infty) $, 
$ x \in H $, 
$ t \in [0,T] $
it holds that 
\begin{equation}
\label{eq:itemvii.higher}
  \limsup_{ \nzspace{H} \ni u_m \to 0 }
  \sup_{ \mathbf{u}=(u_1,u_2,\ldots,u_{m-1}) \in (\nzspace{H})^{m-1} }
  \Bigg[
  \frac{
    \|
      X^{m-1,(x+u_m,\mathbf{u})}_t - X^{m-1,(x,\mathbf{u})}_t
      - X^{m,(x,\mathbf{u},u_m)}_t
    \|_{\lpn{p}{\P}{H}}
  }{
    \prod^m_{ i=1 }
    \|u_i\|_H
  }
  \Bigg]
  =0.
\end{equation}
Combining~\eqref{eq:itemvii} and~\eqref{eq:itemvii.higher} with~\eqref{eq:product.lp} and Jensen's inequality yields that for all 
$ p \in [1,\infty) $, 
$ \varpi \in \Pi_k $, 
$ \mathbf{u} \in H^k $, 
$ x \in H $, 
$ t \in [0,T] $ 
it holds that 
\begin{equation}
\label{eq:vector.diff}
\begin{split}
&
  \limsup_{ \nzspace{H} \ni h \to 0 }
  \Bigg[
  \frac{
    \big\|
    \mathbf{Y}^{0,x+h,\varpi,\mathbf{u}}_t
    -
    \mathbf{Y}^{0,x,\varpi,\mathbf{u}}_t
    -
    \mathbf{Y}^{1,(x,h),\varpi,\mathbf{u}}_t
    \big\|_{\lpn{p}{\P}{H^{\#_\varpi+1}}}
  }{ \|h\|_H }
  \Bigg]
\\&\leq
  \limsup_{ \nzspace{H} \ni h \to 0 }
  \Bigg[
  \frac{
    \|X^{0,x+h}_t - X^{0,x}_t - X^{1,(x,h)}_t\|_{ \lpn{p}{\P}{H} }
  }{
    \|h\|_H
  }
  \Bigg]
\\&+
  \sum^{\#_\varpi}_{ i=1 }
  \limsup_{ \nzspace{H} \ni h \to 0 }
  \Bigg[
  \frac{
    \big\|
    X^{\#_{I^\varpi_i},[(x+h,\mathbf{u})]^\varpi_{i}}_t
    -
    X^{\#_{I^\varpi_i},[(x,\mathbf{u})]^\varpi_{i}}_t-
    X^{\#_{I^\varpi_i}+1,([(x,\mathbf{u})]^\varpi_{i},h)}_t
    \big\|_{\lpn{p}{\P}{H}}
  }{ \|h\|_H }
  \Bigg]
  =0.
\end{split}
\end{equation}
In addition, combining~\eqref{item:itemv} with~\eqref{eq:product.lp} and Jensen's inequality
yields that for all 
$ p \in [1,\infty) $, 
$ \varpi \in \Pi_k $, 
$ \mathbf{u} \in H^k $, 
$ x \in H $, 
$ t \in [0,T] $ 
it holds that 
\begin{equation}
\label{eq:vector.continuity}
\begin{split}
&
  \limsup_{ H \ni y \to x }
  \sup_{ h \in \nzspace{H} }
  \left[
  \frac{
    \| \mathbf{Y}^{1,(x,h),\varpi,\mathbf{u}}_t - \mathbf{Y}^{1,(y,h),\varpi,\mathbf{u}}_t \|_{\lpn{p}{\P}{H^{\#_\varpi+1}}}
  }{\|h\|_H}
  \right]
\\&\leq
  \limsup_{ H \ni y \to x }
  \sup_{ h \in \nzspace{H} }
  \left[
  \frac{
    \| X^{1,(x,h)}_t - X^{1,(y,h)}_t \|_{\lpn{p}{\P}{H}}
  }{\|h\|_H}
  \right]
\\&+
  \sum^{\#_\varpi}_{i=1}
  \limsup_{ H \ni y \to x }
  \sup_{ h \in \nzspace{H} }
  \left[
  \frac{
    \| X^{\#_{I^\varpi_i}+1,([x,\mathbf{u}]^\varpi_i,h)}_t - X^{\#_{I^\varpi_i}+1,([y,\mathbf{u}]^\varpi_i,h)}_t \|_{\lpn{p}{\P}{H}}
  }{\|h\|_H}
  \right]
  =0.
\end{split}
\end{equation}
Combining~\eqref{eq:vector.bound.linear} and~\eqref{eq:vector.diff} 
hence yields that for all 
$ p \in [1,\infty) $, 
$ \varpi \in \Pi_k $, 
$ \mathbf{u} \in H^k $, 
$ x, h \in H $, 
$ t \in [0,T] $ 
it holds that 
\begin{equation}
  \big(
  H \ni y \mapsto
    \big[ \mathbf{Y}^{0,y,\varpi,\mathbf{u}}_t \big]_{ \P, \mathcal{B}(H^{\#_\varpi+1}) }
    \in \lpnb{p}{\P}{H^{\#_\varpi+1}}
  \big) 
  \in \Cp{1}{}( H,\lpnb{p}{\P}{H^{\#_\varpi+1}}  )
\end{equation}
and 
\begin{equation}
  \big(
  \tfrac{\partial}{ \partial x }
  \big[ \mathbf{Y}^{0,x,\varpi,\mathbf{u}}_t \big]_{ \P, \mathcal{B}(H^{\#_\varpi+1}) }
  \big) h
  =
  \big[ \mathbf{Y}^{1,(x,h),\varpi,\mathbf{u}}_t \big]_{ \P, \mathcal{B}(H^{\#_\varpi+1}) }
  .
\end{equation}
This, \eqref{eq:test.derivative}, \eqref{eq:polynomial.order}, 
and Lemma~\ref{lem:stoch.chain.rule}
(with 
$ U = H $, 
$ V = H^{\#_\varpi+1} $, 
$ W = V $, 
$ (\Omega, \mathcal{F}, \P) = (\Omega, \mathcal{F}, \P) $, 
$ 
  X^{0,x} = 
  \mathbf{Y}^{0,x,\varpi,\mathbf{u}}_t
$, 
$ 
  X^{1,(x,h)} = 
  \mathbf{Y}^{1,(x,h),\varpi,\mathbf{u}}_t
$, 
$ \varphi = \Phi_{\#_\varpi} $ 
for 
$ t \in [0,T] $, 
$ x, h \in H $, 
$ \mathbf{u} \in H^k $, 
$ \varpi \in \Pi_k $
in the notation of Lemma~\ref{lem:stoch.chain.rule})
assure that 
\begin{enumerate}[(a)]
\item
\label{item:trans.higher.derivative.smooth}
it holds for all 
$ \varpi \in \Pi_k $, 
$ \mathbf{u} \in H^k $, 
$ t \in [0,T] $ 
that 
\begin{multline}
  \Big(
    H \ni x \mapsto
    \E\Big[\Phi_{\#_\varpi}\big(\mathbf{Y}^{0,x,\varpi,\mathbf{u}}_t\big)\Big]
\\=
    \E\Big[
      \varphi^{(\#_\varpi)}(X_t^{0,x})
      \big(
        X_t^{\#_{I^\varpi_1}, [(x,\mathbf{u})]_1^{\varpi}}
        ,
        X_t^{\#_{I^\varpi_2}, [(x,\mathbf{u})]_2^{\varpi}}
        ,
        \dots
        ,
        X_t^{\#_{I^\varpi_{\#_\varpi}}, [(x,\mathbf{u})]_{\#_{\varpi}}^\varpi}
      \big)
    \Big]
    \in V
  \Big) 
  \in \Cp{1}{}(H,V)
\end{multline}
and 
\item
\label{item:trans.higher.derivative.rep}
it holds for all 
$ \varpi \in \Pi_k $, 
$ \mathbf{u} \in H^k $, 
$ x, u_{k+1} \in H $, 
$ t \in [0,T] $ 
that 
\begin{equation}
\label{eq:induction.rep.III}
\begin{split}
&\!\!\!\!\!\!\!\!\!\!\!\!\!\!\!\!\!
    \Big(
    \tfrac{\partial}{\partial x}
    \E\Big[
      \varphi^{(\#_\varpi)}(X_t^{0,x})
      \big(
        X_t^{ \#_{I^\varpi_1}, [ (x,\mathbf{u}) ]_1^{ \varpi } }
        ,
        X_t^{ \#_{I^\varpi_2}, [ (x,\mathbf{u}) ]_2^{ \varpi } }
        ,
        \dots
        ,
        X_t^{ \#_{I^\varpi_{\#_\varpi}}, [ (x,\mathbf{u}) ]_{ \#_\varpi }^{ \varpi } }
      \big)
    \Big]
    \Big) u_{k+1}
\\&\!\!\!\!\!\!\!\!\!\!\!\!\!\!\!\!\!=
\Big(
\tfrac{\partial}{\partial x}
\E\Big[
\Phi_{\#_\varpi}\big(\mathbf{Y}^{0,x,\varpi,\mathbf{u}}_t\big)
\Big]
\Big) u_{k+1}
\\&\!\!\!\!\!\!\!\!\!\!\!\!\!\!\!\!\!=
\E\Big[
\Phi'_{\#_\varpi}\big(\mathbf{Y}^{0,x,\varpi,\mathbf{u}}_t\big)
\mathbf{Y}^{1,(x,u_{k+1}),\varpi,\mathbf{u}}_t
\Big]
\\&\!\!\!\!\!\!\!\!\!\!\!\!\!\!\!\!\!=
  \E\Big[
    \varphi^{(\#_\varpi+1)}(X_t^{0,x})
    \big(
        X_t^{ \#_{I^\varpi_1}, [ (x,\mathbf{u}) ]_1^{ \varpi } }
        ,
        X_t^{ \#_{I^\varpi_2}, [ (x,\mathbf{u}) ]_2^{ \varpi } }
        ,
        \dots
        ,
        X_t^{ \#_{I^\varpi_{\#_\varpi}}, [ (x,\mathbf{u}) ]_{ \#_\varpi }^{ \varpi } }
      ,
      X_t^{1,(x,u_{k+1})}
    \big)
  \Big]
\\&\!\!\!\!\!\!\!\!\!\!\!\!\!\!\!\!\!\quad+
  \smallsum^{\#_\varpi}_{i=1}
  \E\Big[
    \varphi^{(\#_\varpi)}(X_t^{0,x})
    \big(
      X_t^{\#_{I^\varpi_1},[(x,\mathbf{u})]_1^{\varpi}}
      ,
      X_t^{\#_{I^\varpi_2},[(x,\mathbf{u})]_2^{\varpi}}
      ,
      \dots
      ,
      X_t^{\#_{I^\varpi_{i-1}},[(x,\mathbf{u})]_{i-1}^{\varpi}}
      ,
\\&\!\!\!\!\!\!\!\!\!\!\!\!\!\!\!\!\!\quad
      X_t^{\#_{I^\varpi_i}+1,([(x,\mathbf{u})]_i^{\varpi},u_{k+1})}
      ,
      X_t^{\#_{I^\varpi_{i+1}},[(x,\mathbf{u})]_{i+1}^{\varpi}}
      ,
      X_t^{\#_{I^\varpi_{i+2}},[(x,\mathbf{u})]_{i+2}^{\varpi}}
      ,
      \ldots
      ,
      X_t^{\#_{I^\varpi_{\#_\varpi}},[(x,\mathbf{u})]_{\#_\varpi}^\varpi}
    \big)
  \Big]
  .
\end{split}
\end{equation}
\end{enumerate}
Combining item~\eqref{item:trans.higher.derivative.smooth} with the induction hypothesis shows that for all 
$ \mathbf{u}\in H^k $, 
$ t \in [0,T] $ 
it holds that 
\begin{multline}
\Big(
H \ni x \mapsto
{\smallsum\limits_{\varpi\in\Pi_k}}
    \E\Big[
    \varphi^{(\#_\varpi)}(X_t^{0,x})
    \big(
    X_t^{\#_{I^\varpi_1}, [(x,\mathbf{u})]_1^{\varpi}}
    ,
    X_t^{\#_{I^\varpi_2}, [(x,\mathbf{u})]_2^{\varpi}}
    ,
    \dots
    ,
    X_t^{\#_{I^\varpi_{\#_\varpi}}, [(x,\mathbf{u})]_{\#_{\varpi}}^\varpi}
    \big)
    \Big]
\\=
    \big(\tfrac{\partial^k}{\partial x^k}\phi\big)(t,x) \mathbf{u}
    \in V
\Big) 
\in \Cp{1}{}(H,V).
\end{multline}
Item~\eqref{item:trans.higher.derivative.rep} hence proves that for all 
$\mathbf{u}\in H^k$, 
$x,h\in H$, 
$t\in[0,T]$ 
it holds that 
\begin{equation}
\label{eq:induction.derivative}
\begin{split}
&
  \big(\tfrac{d}{d x}\big[\big(\tfrac{\partial^k}{\partial x^k}\phi\big)(t,x) \mathbf{u}\big]\big) h
\\&=
\sum_{\varpi\in\Pi_k}\bigg\{
\E\Big[
\varphi^{(\#_\varpi+1)}(X_t^{0,x})
\big(
X_t^{ \#_{I^\varpi_1}, [ (x,\mathbf{u}) ]_1^{ \varpi } }
,
X_t^{ \#_{I^\varpi_2}, [ (x,\mathbf{u}) ]_2^{ \varpi } }
,
\dots
,
X_t^{ \#_{I^\varpi_{\#_\varpi}}, [ (x,\mathbf{u}) ]_{ \#_\varpi }^{ \varpi } }
,
X_t^{1,(x,h)}
\big)
\Big]
\\&\quad+
\smallsum^{\#_\varpi}_{i=1}
\E\Big[
\varphi^{(\#_\varpi)}(X_t^{0,x})
\big(
X_t^{\#_{I^\varpi_1},[(x,\mathbf{u})]_1^{\varpi}}
,
X_t^{\#_{I^\varpi_2},[(x,\mathbf{u})]_2^{\varpi}}
,
\dots
,
X_t^{\#_{I^\varpi_{i-1}},[(x,\mathbf{u})]_{i-1}^{\varpi}}
,
\\&\quad
X_t^{\#_{I^\varpi_i}+1,([(x,\mathbf{u})]_i^{\varpi},h)}
,
X_t^{\#_{I^\varpi_{i+1}},[(x,\mathbf{u})]_{i+1}^{\varpi}}
,
X_t^{\#_{I^\varpi_{i+2}},[(x,\mathbf{u})]_{i+2}^{\varpi}}
,
\ldots
,
X_t^{\#_{I^\varpi_{\#_\varpi}},[(x,\mathbf{u})]_{\#_\varpi}^\varpi}
\big)
\Big]\bigg\}
.
\end{split}
\end{equation}
In addition, note that
\begin{equation}
\begin{split}
&
\Pi_{k+1}
=
\Big\{ 
\varpi \cup \big\{\{k+1\}\big\}
\colon 
\varpi \in \Pi_k 
\Big\}
\\&
\biguplus
\Big\{
\big\{
I^\varpi_1, I^\varpi_2, \ldots, I^\varpi_{i-1}, 
I^\varpi_i \cup \{k+1\}, I^\varpi_{i+1}, I^\varpi_{i+2}, 
\ldots, I^\varpi_{\#_\varpi}
\big\}
\colon
i \in \{1,2,\ldots,\#_\varpi\}, \,
\varpi \in \Pi_k
\Big\}
.
\end{split}
\end{equation}
This ensures that for all 
$ \mathbf{u}=(u_0,u_1,\ldots,u_k) \in H^{k+1} $, 
$ h \in H $, 
$ t \in [0,T] $
it holds that 
\begin{equation}
\label{eq:derivative.relation}
\begin{split}
&
  \smallsum_{
    \varpi \in \Pi_{k+1}
  }
    \varphi^{(\#_\varpi)}(X_t^{0,u_0})
    \big(
      X_t^{\#_{I^\varpi_1},[(\mathbf{u},h)]_1^{\varpi}}
      ,
      X_t^{\#_{I^\varpi_2},[(\mathbf{u},h)]_2^{\varpi}}
      ,
      \dots
      ,
      X_t^{\#_{I^\varpi_{\#_\varpi}},[(\mathbf{u},h)]_{\#_\varpi}^\varpi}
    \big)
\\&=
  \smallsum_{
    \varpi \in \Pi_k
  }
  \Big[
    \varphi^{(\#_\varpi+1)}(X_t^{0,u_0})
    \big(
        X_t^{ \#_{I^\varpi_1}, [ \mathbf{u} ]_1^{ \varpi } }
        ,
        X_t^{ \#_{I^\varpi_2}, [ \mathbf{u} ]_2^{ \varpi } }
        ,
        \dots
        ,
        X_t^{ \#_{I^\varpi_{\#_\varpi}}, [\mathbf{u} ]_{ \#_\varpi }^{ \varpi } }
      ,
      X_t^{1,(u_0,h)}
    \big)
\\&\quad+
  \smallsum^{\#_\varpi}_{ i=1 }
    \varphi^{(\#_\varpi)}(X_t^{0,u_0})
    \big(
      X_t^{\#_{I^\varpi_1},[\mathbf{u}]_1^{\varpi}}
      ,
      X_t^{\#_{I^\varpi_2},[\mathbf{u}]_2^{\varpi}}
      ,
      \dots
      ,
      X_t^{\#_{I^\varpi_{i-1}},[\mathbf{u}]_{i-i}^{\varpi}}
      ,
      X_t^{\#_{I^\varpi_i}+1,([\mathbf{u}]_i^{\varpi},h)}
      ,
\\&\quad
      X_t^{\#_{I^\varpi_{i+1}},[\mathbf{u}]_{i+1}^{\varpi}}
      ,
      X_t^{\#_{I^\varpi_{i+2}},[\mathbf{u}]_{i+2}^{\varpi}}
      ,
      \dots
      ,
      X_t^{\#_{I^\varpi_{\#_\varpi}},[\mathbf{u}]_{\#_\varpi}^\varpi}
    \big)
    \Big]
    .
\end{split}
\end{equation}
Combining this with~\eqref{eq:induction.derivative} establishes 
that for all 
$\mathbf{u}\in H^k$, 
$x,h\in H$, 
$t\in[0,T]$ 
it holds that 
\begin{equation}
\begin{split}
&
\big(\tfrac{d}{d x}\big[\big(\tfrac{\partial^k}{\partial x^k}\phi\big)(t,x) \mathbf{u}\big]\big) h
\\&=
\sum_{\varpi\in\Pi_{k+1}}
\E\Big[\varphi^{(\#_\varpi)}(X^{0,u_0}_t)\big(
X^{\#_{I^\varpi_1},[(x,\mathbf{u},h)]^\varpi_1}_t,
X^{\#_{I^\varpi_2},[(x,\mathbf{u},h)]^\varpi_2}_t,
\ldots,
X^{\#_{I^\varpi_{\#_\varpi}},[(x,\mathbf{u},h)]^\varpi_{\#_\varpi}}_t,
\big)\Big].
\end{split}
\end{equation}
Hence, we obtain that for all 
$ \mathbf{u} = (u_1,u_2,\ldots,u_k) \in H^k $, 
$ x \in H $, 
$ t \in [0,T] $ 
it holds that 
\begin{equation}
\begin{split}
&
  \limsup_{ \nzspace{H} \ni h \to 0 }
  \bigg(
  \frac{1}{ \|h\|_H }
  \Big\|
    \big(
    \tfrac{\partial^k}{\partial x^k}
    \trans
    \big)(t,x+h)
    \mathbf{u}
  -
    \big(
    \tfrac{\partial^k}{\partial x^k}
    \trans
    \big)(t,x)
    \mathbf{u}
\\&-
  \smallsum_{
    \varpi \in \Pi_{k+1}
  }
  \E
  \Big[
    \varphi^{(\#_\varpi)}(X_t^{0,x})
    \big(
      X_t^{\#_{I^\varpi_1},[(x,\mathbf{u},h)]_1^{\varpi}}
      ,
      X_t^{\#_{I^\varpi_2},[(x,\mathbf{u},h)]_2^{\varpi}}
      ,
      \dots
      ,
      X_t^{\#_{I^\varpi_{\#_\varpi}},[(x,\mathbf{u},h)]_{\#_\varpi}^\varpi}
    \big)
  \Big]
  \Big\|_V
  \bigg)
  =0.
\end{split}
\end{equation}
Combining~\eqref{eq:trans.derivative.continuity} and Lemma~\ref{lem:partial.diff} 
(with 
$ V = H $, 
$ W = V $, 
$ n = k $, 
$ f = ( H \ni x \mapsto \trans(t,x) \in V) $, 
$
  g =
  (
    H \ni x \mapsto
    (
      H^{k+1}
      \ni \mathbf{u} \mapsto
  \smallsum_{
    \varpi \in \Pi_{k+1}
  }
  \E
  [
    \varphi^{(\#_\varpi)}(X_t^{0,x})
    (
        X_t^{\#_{I^\varpi_1}, [(x,\mathbf{u})]_1^{\varpi}}
        ,
        X_t^{\#_{I^\varpi_2}, [(x,\mathbf{u})]_2^{\varpi}}
        ,
        \dots
        ,
$
$
        X_t^{\#_{I^\varpi_{\#_\varpi}}, [(x,\mathbf{u})]_{\#_{\varpi}}^\varpi}
    )
  ]
  \in V
    )
    \in L^{(k+1)}(H,V)
  )
$ 
for $ t \in [0,T] $ 
in the notation of Lemma~\ref{lem:partial.diff}) 
therefore shows that for all 
$\mathbf{u}\in H^{k+1}$, 
$x\in H$, 
$t\in[0,T]$
it holds that 
$
  (H \ni y \mapsto \phi(t,y) \in V)
  \in \Cp{k+1}{}(H,V)
$
and 
\begin{equation}
\begin{split}
&
  \big(\tfrac{\partial^{k+1}}{\partial x^{k+1}}\phi\big)(t,x) \mathbf{u}
\\&=
  \sum_{
  	\varpi \in \Pi_{k+1}
  }
  \E
  \Big[
  \varphi^{(\#_\varpi)}(X_t^{0,x})
  \big(
  X_t^{\#_{I^\varpi_1},[(x,\mathbf{u})]_1^{\varpi}}
  ,
  X_t^{\#_{I^\varpi_2},[(x,\mathbf{u})]_2^{\varpi}}
  ,
  \dots
  ,
  X_t^{\#_{I^\varpi_{\#_\varpi}},[(x,\mathbf{u})]_{\#_\varpi}^\varpi}
  \big)
  \Big].
\end{split}
\end{equation}
This and~\eqref{eq:trans.derivative.holder} prove~\eqref{items:trans.regularity} in the case $k+1$.
Induction thus completes the proof of~\eqref{items:trans.regularity}.

Next observe that item~\eqref{item:trans.smoothness} and item~\eqref{item:thm.representation} follow immediately from~\eqref{items:trans.regularity}.
It thus remains to prove items~\eqref{item:lem.derivative.apriori}--\eqref{item:trans.lip}. 
To prove item~\eqref{item:lem.derivative.apriori} we first note that
\eqref{eq:SEE.derivative} and item~(ii) of Theorem~2.1 in~\cite{AnderssonJentzenKurniawan2016a}
(with 
$ T = T $, 
$ \eta = \eta $, 
$ H = H $, 
$ U = U $, 
$ W = W $, 
$ A = A $, 
$ n = n $, 
$ F = F $, 
$ B = B $, 
$ \alpha = \alpha $, 
$ \beta = \beta $, 
$ k = k $, 
$ p = p $, 
$ \boldsymbol{\delta} = \boldsymbol{\delta} $
for 
$ \boldsymbol{\delta} \in \deltaset{k} $, 
$ p \in [2,\infty) $, 
$ k \in \{1,2,\ldots,n\} $ 
in the notation of Theorem~2.1 
in~\cite{AnderssonJentzenKurniawan2016a}) 
ensure that for all 
$ k \in \{1,2,\ldots,n\} $, 
$ p \in [2,\infty) $, 
$ \boldsymbol{\delta} = (\delta_1,\delta_2,\ldots,\delta_k) \in \deltaset{k} $ 
it holds that 
\begin{equation}
\label{eq:itemii.rough}
  \sup_{ \mathbf{u}=(u_0,u_1,\ldots,u_k) \in H \times
  (\nzspace{H})^k }
  \sup_{ t\in(0,T] }
  \Bigg[
  \frac{
    t^{\iota^{\boldsymbol{\delta},\alpha,\beta}_\N}\,
    \|X^{k,\mathbf{u}}_t\|_{\lpn{p}{\P}{H}}
  }{
    \prod^k_{ i = 1 } 
    \|u_i\|_{H_{-\delta_i}}
  }
  \Bigg]
  < \infty
  .
\end{equation}
This and Jensen's inequality establish item~\eqref{item:lem.derivative.apriori}.
Moreover, observe that for all 
$ k \in \N $, 
$
  \boldsymbol{\delta}=(\delta_1, \delta_2, \ldots, \delta_k) \in \R^k
$, 
$ \varpi \in \Pi_k $ 
it holds that 
\begin{equation}
\label{eq:time.bound.I}
\begin{split}
&
  \sup_{ t \in (0,T] }
   \prod_{ I \in \varpi }
   t^{ ( - \iota^{ \boldsymbol{\delta}, \alpha, \beta }_I + \sum_{ i \in I } \delta_i ) }  
   =
  \sup_{ t \in (0,T] }
  t^{ \min\{1-\alpha,\nicefrac{1}{2}-\beta\} \sum_{I\in\varpi} \1_{[2,\infty)}(\#_I) }
\\&=
  T^{ \min\{1-\alpha,\nicefrac{1}{2}-\beta\} \sum_{I\in\varpi} \1_{[2,\infty)}(\#_I) }  
  \leq
  |T\vee 1|^{ \lfloor k/2 \rfloor \min\{ 1-\alpha, 1/2-\beta \} }.
\end{split}
\end{equation}
Combining~\eqref{eq:trans.holder.bound} with 
item~\eqref{item:thm.representation},  
\eqref{eq:itemii.rough}, \eqref{eq:time.bound.I}, and Jensen's inequality
yields that for all 
$ k \in \{1,2,\ldots,n\} $, 
$
  \boldsymbol{\delta}=(\delta_1, \delta_2, \ldots, \delta_k) \in
  \deltaset{k}
$ 
it holds that 
\begin{equation}
\label{eq:trans.derivative.bound}
\begin{split}
&
  \sup_{ v \in H }
  \sup_{
    \mathbf{u}=(u_1, u_2, \dots, u_k) \in 
    (\nzspace{H})^k
  }
  \sup_{
    t \in (0,T]
  }
  \Bigg[
  \frac{
    t^{\sum^k_{i=1} \delta_i}
    \,
    \big\|
    \big(
    \tfrac{\partial^k}{\partial x^k}
    \trans
    \big)(t,v)
  \mathbf{u}
    \big\|_V
  }{
    \prod^k_{ i=1 }
    \|
      u_i
    \|_{
      H_{ - \delta_i }
    }
  }
  \Bigg]
\\&\leq
  \sum_{
    \varpi \in \Pi_k
  }
  |\varphi|_{\Cb{\#_\varpi}(H,V)}
  \Bigg(
    \sup_{ t \in (0,T] }
  \left[
    \prod_{ I \in \varpi }
    t^{ ( - \iota^{ \boldsymbol{\delta}, \alpha, \beta }_I + \sum_{ i \in I } \delta_i ) }
  \right]
  \Bigg)
\\&\quad\cdot
  \Bigg(
  \prod_{ I\in\varpi }
  \sup_{ x \in H }
  \sup_{ \mathbf{u}=(u_i)_{i\in I} \in 
  (\nzspace{H})^{\#_I} }
  \sup_{ t\in(0,T] }
  \Bigg[
  \frac{
    t^{\iota^{\boldsymbol{\delta},\alpha,\beta}_I}\,
    \|X^{\#_I,(x,\mathbf{u})}_t\|_{\lpn{\#_\varpi}{\P}{H}}
  }{
    \prod_{i \in I} 
    \|u_i\|_{H_{-\delta_i}}
  }
  \Bigg]\Bigg)
\\&\leq
    |T\vee 1|^{ \lfloor k/2 \rfloor \min\{ 1-\alpha, 1/2-\beta \} } \,
    \|\varphi\|_{ \Cb{k}(H,V) }
\\&\quad\cdot
  \sum_{
    \varpi \in \Pi_k
  }
  \prod_{ I\in\varpi }
  \sup_{ x \in H }
  \sup_{ \mathbf{u}=(u_i)_{i\in I} \in 
  (\nzspace{H})^{\#_I} }
  \sup_{ t\in(0,T] }
  \Bigg[
  \frac{
    t^{\iota^{\boldsymbol{\delta},\alpha,\beta}_I}\,
    \|X^{\#_I,(x,\mathbf{u})}_t\|_{\lpn{\#_\varpi}{\P}{H}}
  }{
    \prod_{i \in I} 
    \|u_i\|_{H_{-\delta_i}}
  }
  \Bigg]
  < \infty
  .
\end{split}
\end{equation}
This proves item~\eqref{item:thm.derivative.bound}. 
Next we observe that~\eqref{eq:SEE.derivative} and item~(iv) of Theorem~2.1 in~\cite{AnderssonJentzenKurniawan2016a} 
(with 
$ T = T $, 
$ \eta = \eta $, 
$ H = H $, 
$ U = U $, 
$ W = W $, 
$ A = A $, 
$ n = n $, 
$ F = F $, 
$ B = B $, 
$ \alpha = \alpha $, 
$ \beta = \beta $, 
$ k = k $, 
$ p = p $, 
$ \boldsymbol{\delta} = \boldsymbol{\delta} $
for 
$ \boldsymbol{\delta} \in \deltaset{k} $, 
$ p \in [2,\infty) $, 
$
  k \in \{l\in\{1,2,\ldots,n\}\colon
  |F|_{\operatorname{Lip}^l(H,H_{-\alpha})} 
  + 
  |B|_{\operatorname{Lip}^l(H,HS(U,H_{-\beta}))} < \infty 
  \}
$ 
in the notation of Theorem~2.1 
in~\cite{AnderssonJentzenKurniawan2016a}) 
ensure that for all 
$ k \in \{1,2,\ldots,n\} $, 
$ p \in [2,\infty) $, 
$ \boldsymbol{\delta} = (\delta_1,\delta_2,\ldots,\delta_k) \in \deltaset{k} $ 
with 
$ 
|F|_{\operatorname{Lip}^k(H,H_{-\alpha})} 
+ 
|B|_{\operatorname{Lip}^k(H,HS(U,H_{-\beta}))} < \infty 
$ 
it holds that 
\begin{equation}
\label{eq:itemiii}
    \sup_{\substack{x,y \in H, \\ x \neq y}}
    \sup_{ \mathbf{u}=(u_1,u_2,\ldots,u_k) \in (\nzspace{H})^k }
    \sup_{ t \in (0,T] }
    \Bigg[
    \frac{
      t^{ \iota^{(\boldsymbol{\delta},0),\alpha,\beta}_\N } \,
      \|X_t^{k,(x,\mathbf{u})}-X_t^{k,(y,\mathbf{u})}\|_{\lpn{p}{\P}{H}}
    }{
      \|x-y\|_H
      \prod^k_{ i=1}
      \|u_i\|_{H_{-\delta_i}}
    }
    \Bigg]
    < \infty.
\end{equation}
Combining this and~\eqref{eq:cor2.10} with Jensen's inequality establish items~\eqref{item:lem.0.lip} and~\eqref{item:lem.derivative.lip}.
Moreover, note that for all 
$ k \in \N $, 
$
\boldsymbol{\delta}=(\delta_1, \delta_2, \ldots, \delta_k) \in \R^k
$, 
$ \varpi \in \Pi_k $, 
$ I \in \varpi $
it holds that 
\begin{equation}
\label{eq:time.bound.II}
\begin{split}
&
  \sup_{ t \in (0,T] }
  \Bigg[
  t^{ ( - \iota^{ (\boldsymbol{\delta},0), \alpha, \beta }_{I\cup\{k+1\}} + \sum_{ i \in I } \delta_i ) }
  \prod_{ J \in \varpi \setminus \{I\} }
  t^{ ( - \iota^{ \boldsymbol{\delta}, \alpha, \beta }_J + \sum_{ i \in J } \delta_i ) }
  \Bigg]
\\&=
  \sup_{ t \in (0,T] }
  t^{ \min\{1-\alpha,\nicefrac{1}{2}-\beta\} \, [1+\sum_{J\in\varpi\setminus\{I\}} \1_{[2,\infty)}(\#_J)] }
\\&=
  T^{ \min\{1-\alpha,\nicefrac{1}{2}-\beta\} \, [1+\sum_{J\in\varpi\setminus\{I\}} \1_{[2,\infty)}(\#_J)] }
\leq
  |T\vee 1|^{ \lceil k/2 \rceil \min\{ 1-\alpha, 1/2-\beta \} }.
\end{split}
\end{equation}
Furthermore, note that 
item~\eqref{item:stoch.derivative.representation} and~\eqref{eq:trans.continuity.holder} 
imply that for all 
$ k \in \{1,2,\ldots,n\} $, 
$
  \boldsymbol{\delta}=(\delta_1, \delta_2, \ldots, \delta_k) \in
  \deltaset{k}
$ 
with 
$
  |\varphi|_{\operatorname{Lip}^k(H,V)}
  < \infty
$ 
it holds that 
\begin{equation}
\label{eq:trans.lip}
\begin{split}
&
  \sup_{\substack{v,w \in H, \\ v \neq w}}
  \sup_{ \mathbf{u}=(u_1,u_2,\ldots,u_k) \in (\nzspace{H})^k }
  \sup_{t \in (0,T] }
  \left[
  \frac{
    t^{\sum^k_{i=1} \delta_i}
    \big\|
  \big[
    \big(
    \tfrac{\partial^k}{\partial x^k}
    \trans
    \big)(t,v)
  -
    \big(
    \tfrac{\partial^k}{\partial x^k}
    \trans
    \big)(t,w)
  \big]
    \mathbf{u}
    \big\|_V    
  }{
    \|v-w\|_H
    \prod^k_{i=1} \|u_i\|_{H_{-\delta_i}}
  }
  \right]
\\&\leq
    \sum_{
      \varpi \in \Pi_k
    }
    \Bigg\{
    |\varphi|_{ \operatorname{Lip}^{\#_\varpi}(H,V) }
  \Bigg(
    \sup_{ t \in (0,T] }
  \left[
    \prod_{ I \in \varpi }
    t^{ ( - \iota^{ \boldsymbol{\delta}, \alpha, \beta }_I + \sum_{ i \in I } \delta_i ) }
  \right]
  \Bigg)
  \Bigg(
    \sup_{\substack{x,y \in H, \\ x \neq y}}
    \sup_{ t \in (0,T] }
    \Bigg[
      \frac{\| X_t^{0,x} - X_t^{0,y} \|_{\lpn{\#_\varpi+1}{\P}{H}}}{\|x-y\|_H}
    \Bigg]\Bigg)
\\&\cdot
    \Bigg(
    \prod_{I\in\varpi} 
    \sup_{ x \in H }
    \sup_{ \mathbf{u}=(u_i)_{i\in I} \in (\nzspace{H})^{\#_I} }
    \sup_{ t \in (0,T] }
    \Bigg[
    \frac{
      t^{ \iota^{\boldsymbol{\delta},\alpha,\beta}_I } \,
      \|X_t^{\#_I,(x,\mathbf{u})}\|_{\lpn{\#_\varpi+1}{\P}{H}}
    }{
      \prod_{i\in I}
      \|u_i\|_{H_{-\delta_i}}
    }
    \Bigg]\Bigg)
\\&+
    |\varphi|_{\Cb{\#_\varpi}(H,V)}
    \sum_{I\in\varpi}
    \Bigg(
    \sup_{ t \in (0,T] }
  \Bigg[
    t^{ ( - \iota^{ (\boldsymbol{\delta},0), \alpha, \beta }_{I\cup\{k+1\}} + \sum_{ i \in I } \delta_i ) }
    \prod_{ J \in \varpi \setminus \{I\} }
    t^{ ( - \iota^{ \boldsymbol{\delta}, \alpha, \beta }_J + \sum_{ i \in J } \delta_i ) }
  \Bigg]\Bigg)
\\&\cdot
    \Bigg(
    \sup_{\substack{x,y \in H, \\ x \neq y}}
    \sup_{ \mathbf{u}=(u_i)_{i\in I} \in (\nzspace{H})^{\#_I} }
    \sup_{ t \in (0,T] }
    \Bigg[
    \frac{
      t^{ \iota^{(\boldsymbol{\delta},0),\alpha,\beta}_{I\cup\{k+1\}} } \,
      \|X_t^{\#_I,(x,\mathbf{u})}-X_t^{\#_I,(y,\mathbf{u})}\|_{\lpn{\#_\varpi}{\P}{H}}
    }{
      \|x-y\|_H
      \prod_{i\in I}
      \|u_i\|_{H_{-\delta_i}}
    }
    \Bigg]\Bigg)
\\&\cdot
    \Bigg(
    \prod_{ J \in \varpi \setminus \{I\} } 
    \sup_{ x \in H }
    \sup_{ \mathbf{u}=(u_i)_{i\in J} \in (\nzspace{H})^{\#_J} }
    \sup_{ t \in (0,T] }
    \Bigg[
    \frac{
      t^{ \iota^{\boldsymbol{\delta},\alpha,\beta}_J } \,
      \|X_t^{\#_J,(x,\mathbf{u})}\|_{\lpn{\#_\varpi}{\P}{H}}
    }{
      \prod_{i\in J}
      \|u_i\|_{H_{-\delta_i}}
    }
    \Bigg]\Bigg)
    \Bigg\}
    .
\end{split}
 \end{equation}
Combining~\eqref{eq:trans.lip} with~\eqref{eq:cor2.10}, \eqref{eq:itemii.rough}, \eqref{eq:time.bound.I}, \eqref{eq:itemiii}, \eqref{eq:time.bound.II}, and Jensen's inequality establishes item~\eqref{item:trans.lip}. 
The proof of Lemma~\ref{lem:derivative_formulas} is thus completed.
\end{proof}

\begin{theorem}
	\label{thm:derivative_formulas}
	Assume the setting in Section~\ref{sec:global_setting} 
	and let $ n \in \N $,
	$ \varphi \in \Cb{n}(H,V) $, 
	$ F \in \Cb{n}(H,H) $,
	$ B \in \Cb{n}(H,HS(U,H)) $. 
	Then 
	\begin{enumerate}[(i)]
		\item
		\label{thm.existence}
		it holds that there exist up-to-modifications unique
		$(\mathcal{F}_t)_{t\in[0,T]}$/$ \mathcal{B}(H) $-predictable stochastic processes
		$
		X^{ k,\mathbf{u} }
		\colon
		[ 0 , T ] \times \Omega
		\to H
		$, 
		$
		\mathbf{u} \in H^{k+1}
		$, 
		$
		k \in \{ 0, 1, \dots, n \}
		$,
		which satisfy
		for all
		$
		k \in \{ 0, 1, \dots, n \}
		$,
		$
		\mathbf{u} = (u_0,u_1,\ldots,u_k) \in H^{k+1}
		$, 
		$ p \in (0,\infty) $,
		$ t \in [0,T] $
		that
		$
		\sup_{s\in[0,T]}
		\E\big[\|X^{k,\mathbf{u}}_s\|^p_H\big]
		< \infty
		$ 
		and 
		\begin{equation}
		\begin{split}
		&
		[
		X_t^{k,\mathbf{u}}
		-
		e^{tA}
		\, \mathbbm{1}_{ \{ 0, 1 \} }(k) \, u_k  
		]_{ \P,\mathcal{B}(H) }
		\\ &
		=
		\int_0^t
		e^{ ( t - s ) A }
		\Bigg[
		\mathbbm{1}_{ \{ 0 \} }(k)
		\,
		F(X_s^{0,u_0})
		\\&\quad+
		\sum_{ \varpi\in \Pi_k }
		F^{ ( \#_\varpi ) }( X_s^{ 0,u_0 } )
		\big(
		X_s^{ \#_{I^\varpi_1}, [ \mathbf{u} ]_1^{ \varpi } }
		,
		X_s^{ \#_{I^\varpi_2}, [ \mathbf{u} ]_2^{ \varpi } }
		,
		\dots
		,
		X_s^{ \#_{I^\varpi_{\#_\varpi}}, [\mathbf{u} ]_{ \#_\varpi }^{ \varpi } }
		\big)
		\Bigg]
		\,{\bf ds}
		\\ &
		+
		\int_0^t
		e^{ ( t - s ) A }
		\Bigg[
		\mathbbm{1}_{ \{ 0 \} }(k)
		\,
		B(X_s^{0,u_0})
		\\&\quad+
		\sum_{ \varpi\in \Pi_k }
		B^{ ( \#_\varpi ) }( X_s^{ 0,u_0 } )
		\big(
		X_s^{ \#_{I^\varpi_1}, [ \mathbf{u} ]_1^{ \varpi } }
		,
		X_s^{ \#_{I^\varpi_2}, [ \mathbf{u} ]_2^{ \varpi } }
		,
		\dots
		,
		X_s^{ \#_{I^\varpi_{\#_\varpi}}, [\mathbf{u} ]_{ \#_\varpi }^{ \varpi } }
		\big)
		\Bigg]
		\, \diffns W_s
		,
		\end{split}
		\end{equation}
		\item
		\label{item:thm.TPS}
		it holds that there exists a unique function
		$
		\trans \colon [0,T] \times H \to V
		$ which satisfies for all 
		$ t \in [0,T] $, 
		$ x \in H $ 
		that 
		$
		\trans(t,x)
		=
		\E[\varphi(X^{0,x}_t)]
		$,
		\item
		\label{item:thm.trans.smoothness}
		it holds for all 
		$ t \in [ 0 , T ] $ 
		that 
		$ \big( H \ni x \mapsto \trans(t,x) \in V \big) \in \Cb{n}(H,V) $,  
		\item
		\label{item:thm.trans.derivative.integrability}
		it holds for all 
		$ k \in \{1,2,\ldots,n\} $, 
		$
		\mathbf{u} = (u_0 , u_1 , \dots , u_k) \in
		H^{k+1}
		$, 
		$ t \in [ 0 , T ] $ 
		that 
		\begin{equation}
		\sum\limits_{
			\varpi \in \Pi_k
		}
		\E\Big[\big\|
		\varphi^{(\#_\varpi)}(X_t^{0,u_0})
		\big(
		X_t^{ \#_{I^\varpi_1}, [ \mathbf{u} ]_1^{ \varpi } }
		,
		X_t^{ \#_{I^\varpi_2}, [ \mathbf{u} ]_2^{ \varpi } }
		,
		\dots
		,
		X_t^{ \#_{I^\varpi_{\#_\varpi}}, [\mathbf{u} ]_{ \#_\varpi }^{ \varpi } }
		\big)
		\big\|_V\Big]
		< \infty, 
		\end{equation}
		\item
		\label{item:thm2.representation}
		it holds for all 
		$ k \in \{1,2,\ldots,n\} $, 
		$
		\mathbf{u} \in
		H^k
		$, 
		$ x \in H $, 
		$ t \in [ 0 , T ] $ 
		that 
		\begin{equation}
		\begin{split}
		&
		\big(
		\tfrac{\partial^k}{\partial x^k}
		\trans
		\big)(t,x)
		\mathbf{u}
		\\&=
		\sum\limits_{
			\varpi \in \Pi_k
		}
		\E
		\Big[
		\varphi^{(\#_\varpi)}(X_t^{0,x})
		\big(
		X_t^{\#_{I^\varpi_1},[(x,\mathbf{u})]_1^{\varpi}}
		,
		X_t^{\#_{I^\varpi_2},[(x,\mathbf{u})]_2^{\varpi}}
		,
		\dots
		,
		X_t^{\#_{I^\varpi_{\#_\varpi}},[(x,\mathbf{u})]_{\#_\varpi}^\varpi}
		\big)
		\Big]
		,
		\end{split}
		\end{equation}
\item
\label{item:thm.derivative.apriori}
it holds for all 
$ p \in (0,\infty) $, 
$ k \in \{1,2,\ldots,n\} $, 
$
\boldsymbol{\delta}=(\delta_1,\delta_2,\dots,\delta_k)\in[0,\nicefrac{1}{2})^k
$, 
$ \alpha \in [0,1) $, 
$ \beta \in [0,\nicefrac{1}{2}) $ 
with 
$
\sum_{i=1}^k\delta_i
< \nicefrac{1}{2}
$ 
that 
\begin{equation}
\sup_{x\in H}
\sup_{\mathbf{u}=(u_1,u_2,\ldots,u_k)\in(\nzspace{H})^k}
\sup_{t\in(0,T]}
\left[
\frac{
	t^{\iota^{\boldsymbol{\delta},\alpha,\beta}_\N}
	\|X^{k,(x,\mathbf{u})}_t\|_{\lpn{p}{\P}{H}}
}{
\prod^k_{i=1} \|u_i\|_{H_{-\delta_i}}
}
\right]
< \infty,
\end{equation}
\item
\label{item:thm2.derivative.bound}
it holds for all 
$ k \in \{1,2,\ldots,n\} $, 
$
\boldsymbol{\delta}=(\delta_1,\delta_2,\dots,\delta_k)\in[0,\nicefrac{1}{2})^k
$, 
$ \alpha \in [0,1) $, 
$ \beta \in [0,\nicefrac{1}{2}) $ 
with 
$
\sum_{i=1}^k\delta_i
< \nicefrac{1}{2}
$ 
that 
\begin{equation}
\begin{split}
&
\sup_{ v \in H }
\sup_{
	\mathbf{u}=(u_1, u_2, \dots, u_k) \in 
	(\nzspace{H})^k
}
\sup_{
	t \in (0,T]
}
\Bigg[
\frac{
	t^{\sum^k_{i=1} \delta_i}
	\,
	\big\|
	\big(
	\tfrac{\partial^k}{\partial x^k}
	\trans
	\big)(t,v)
	\mathbf{u}
	\big\|_V
}{
\prod^k_{ i=1 }
\|
u_i
\|_{
	H_{ - \delta_i }
}
}
\Bigg]
\\&\leq
|T \vee 1|^{
	\lfloor k/2 \rfloor \,
	\min\{1-\alpha,\nicefrac{1}{2}-\beta\}
} \,
\|
\varphi
\|_{
	\Cb{k}(H,{V})
}
\\&\cdot
\sum_{
	\varpi \in \Pi_k
}
\prod_{
	I \in \varpi
}
\sup_{x\in H}
\sup_{
	\mathbf{u}=(u_i)_{i\in I}
	\in (\nzspace{H})^{\#_I}
}
\sup_{t\in(0,T]}
\Bigg[
\frac{
	t^{
		\iota^{\boldsymbol{\delta},\alpha,\beta}_I
	}
	\|
	X_t^{\#_I,(x,\mathbf{u})}
	\|_{\mathcal{L}^{\#_\varpi}(\P;H)}
}
{
	\prod_{i\in I}
	\|u_i\|_{H_{-\delta_i}}
}
\Bigg]
<\infty,
\end{split}
\end{equation}
\item
\label{item:thm.0.lip}
it holds for all 
$ p \in (0,\infty) $
that
\begin{equation}
\sup_{\substack{x,y\in H,\\ x\neq y}}
\sup_{t\in(0,T]}
\left[
\frac{
	\|X^{0,x}_t-X^{0,y}_t\|_{\lpn{p}{\P}{H}}
}{
\|x-y\|_H
}
\right]
< \infty,
\end{equation}
\item
\label{item:thm.derivative.lip}
it holds for all 
$ p \in (0,\infty) $, 
$ k \in \{1,2,\ldots,n\} $, 
$
\boldsymbol{\delta}=(\delta_1,\delta_2,\dots,\delta_k)\in[0,\nicefrac{1}{2})^k
$, 
$ \alpha \in [0,1) $, 
$ \beta \in [0,\nicefrac{1}{2}) $ 
with 
$
\sum_{i=1}^k\delta_i
< \nicefrac{1}{2}
$ 
and
$
|F|_{\operatorname{Lip}^k(H,H_{-\alpha})}
+
|B|_{\operatorname{Lip}^k(H,HS(U,H_{-\beta}))}
< \infty
$ 
that 
\begin{equation}
\sup_{\substack{x,y\in H,\\ x\neq y}}
\sup_{\mathbf{u}=(u_1,u_2,\ldots,u_k)\in(\nzspace{H})^k}
\sup_{t\in(0,T]}
\left[
\frac{
	t^{\iota^{(\boldsymbol{\delta},0),\alpha,\beta}_\N}
	\|X^{k,(x,\mathbf{u})}_t-X^{k,(y,\mathbf{u})}_t\|_{\lpn{p}{\P}{H}}
}{
\|x-y\|_H
\prod^k_{i=1} \|u_i\|_{H_{-\delta_i}}
}
\right]
< \infty,
\end{equation}
and 
\item
\label{item:thm.trans.lip}
it holds for all 
$ k \in \{1,2,\ldots,n\} $, 
$
\boldsymbol{\delta}=(\delta_1, \delta_2, \ldots, \delta_k) \in
[0,\nicefrac{1}{2})^k
$,
$ \alpha \in [0,1) $, 
$ \beta \in [0,\nicefrac{1}{2}) $ 
with 
$
\sum_{i=1}^k\delta_i
< \nicefrac{1}{2}
$ 
and 
$
|F|_{\operatorname{Lip}^k(H,H_{-\alpha})}
+
|B|_{\operatorname{Lip}^k(H,HS(U,H_{-\beta}))}
+
|\varphi|_{\operatorname{Lip}^k(H,V)} < \infty
$ 
that 
\begin{equation}
\begin{split}
&
\sup_{\substack{v, w \in H, \\ v \neq w}}
\sup_{ \mathbf{u}=(u_1,u_2,\ldots,u_k) \in (\nzspace{H})^k }
\sup_{t \in (0,T] }
\left[
\frac{
	t^{\sum^k_{i=1} \delta_i}
	\big\|
	\big[
	\big(
	\tfrac{\partial^k}{\partial x^k}
	\trans
	\big)(t,v)
	-
	\big(
	\tfrac{\partial^k}{\partial x^k}
	\trans
	\big)(t,w)
	\big]
	\mathbf{u}
	\big\|_V    
}{
\|v-w\|_H
\prod^k_{i=1} \|u_i\|_{H_{-\delta_i}}
}
\right]
\\&\leq
|T\vee 1|^{
	\lceil k/2 \rceil \, 
	\min\{1-\alpha,\nicefrac{1}{2}-\beta\}
} \,
\|\varphi\|_{ \operatorname{Lip}^k(H,V) }
\\&\cdot
\sum_{
	\varpi \in \Pi_k
}
\Bigg\{
\sup_{\substack{x,y \in H, \\ x \neq y}}
\sup_{ t \in (0,T] }
\Bigg[
\frac{\| X_t^{0,x} - X_t^{0,y} \|_{\lpn{\#_\varpi+1}{\P}{H}}}{\|x-y\|_H}
\Bigg]
\\&\cdot
\prod_{I\in\varpi} 
\sup_{ x \in H }
\sup_{ \mathbf{u}=(u_i)_{i\in I} \in (\nzspace{H})^{\#_I} }
\sup_{ t \in (0,T] }
\Bigg[
\frac{
	t^{ \iota^{\boldsymbol{\delta},\alpha,\beta}_I } \,
	\|X_t^{\#_I,(x,\mathbf{u})}\|_{\lpn{\#_\varpi+1}{\P}{H}}
}{
\prod_{i\in I}
\|u_i\|_{H_{-\delta_i}}
}
\Bigg]
\\&+
\sum_{I\in\varpi}
\sup_{\substack{x,y \in H, \\ x \neq y}}
\sup_{ \mathbf{u}=(u_i)_{i\in I} \in (\nzspace{H})^{\#_I} }
\sup_{ t \in (0,T] }
\Bigg[
\frac{
	t^{ \iota^{(\boldsymbol{\delta},0),\alpha,\beta}_{I\cup\{k+1\}} } \,
	\|X_t^{\#_I,(x,\mathbf{u})}-X_t^{\#_I,(y,\mathbf{u})}\|_{\lpn{\#_\varpi}{\P}{H}}
}{
\|x-y\|_H
\prod_{i\in I}
\|u_i\|_{H_{-\delta_i}}
}
\Bigg]
\\&\cdot
\prod_{ J \in \varpi \setminus \{I\} } 
\sup_{ x \in H }
\sup_{ \mathbf{u}=(u_i)_{i\in J} \in (\nzspace{H})^{\#_J} }
\sup_{ t \in (0,T] }
\Bigg[
\frac{
	t^{ \iota^{\boldsymbol{\delta},\alpha,\beta}_J } \,
	\|X_t^{\#_J,(x,\mathbf{u})}\|_{\lpn{\#_\varpi}{\P}{H}}
}{
\prod_{i\in J}
\|u_i\|_{H_{-\delta_i}}
}
\Bigg]
\Bigg\}
< \infty.
\end{split}
\end{equation}
\end{enumerate}
\end{theorem}
\begin{proof}
Note that items~\eqref{thm.existence} and~\eqref{item:thm.TPS} follow immediately from item~(i) of Theorem~2.1 in~\cite{AnderssonJentzenKurniawan2016a} 
(with 
$ T = T $, 
$ \eta = \eta $, 
$ H = H $, 
$ U = U $, 
$ W = W $, 
$ A = A $, 
$ n = n $, 
$ F = F $, 
$ B = B $, 
$ \alpha = 0 $, 
$ \beta = 0 $
in the notation of Theorem~2.1 in~\cite{AnderssonJentzenKurniawan2016a}).
Moreover, observe that items~\eqref{item:thm.trans.smoothness}--\eqref{item:thm.trans.lip} follow directly from items~\eqref{item:trans.derivative.integrability}--\eqref{item:trans.lip} of Lemma~\ref{lem:derivative_formulas}. The proof of Theorem~\ref{thm:derivative_formulas} is thus completed.
\end{proof}

\section{Regularity of transition semigroups for mollified stochastic evolution equations}

\begin{lemma}
\label{lem:mollified.transition}
Assume the setting in Section~\ref{sec:global_setting}, 
let $ n \in \N $, 
$ \alpha \in [0,1) $, 
$ \beta \in [0,\nicefrac{1}{2}) $, 
$ F \in \Cb{n}(H,H_{-\alpha}) $,
$ B \in \Cb{n}(H,HS(U,H_{-\beta})) $, 
$ \varphi \in \Cb{n}(H,V) $, 
let 
$
  X^{ \varepsilon, k, \mathbf{u} }
  \colon
  [ 0 , T ] \times \Omega
  \to H
$, 
$
  \mathbf{u} \in H^{k+1}
$, 
$
  k \in \{ 0, 1, \dots, n \}
$,
$ \varepsilon \in (0,T] $, 
be $(\mathcal{F}_t)_{t\in[0,T]}$/$ \mathcal{B}(H) $-predictable stochastic processes which satisfy for all 
$
  k \in \{ 0, 1, \dots, n \}
$,
$
  \mathbf{u} = (u_0,u_1,\ldots,u_k) \in H^{k+1}
$, 
$ \varepsilon \in (0,T] $, 
$ p \in (0,\infty) $,
$ t \in [0,T] $
that
$
  \sup_{s\in[0,T]}
  \E\big[\|X^{\varepsilon,k,\mathbf{u}}_s\|^p_H\big]
  < \infty
$ 
and 
\begin{equation}
\label{eq:SEE.derivative.mollified}
\begin{split}
&\!\!\!\!
  [
  X_t^{\varepsilon,k,\mathbf{u}}
  -
  e^{tA}
  \, \mathbbm{1}_{ \{ 0, 1 \} }(k) \, u_k  
  ]_{ \P,\mathcal{B}(H) }
\\ &\!\!\!\!
  =
  \int_0^t
    e^{ ( t - s + \varepsilon ) A }
    \Bigg[
      \mathbbm{1}_{ \{ 0 \} }(k)
      \,
      F(X_s^{\varepsilon,0,u_0})
\\&\quad+
      \sum_{ \varpi\in \Pi_k }
      F^{ ( \#_\varpi ) }( X_s^{ \varepsilon,0,u_0 } )
      \big(
        X_s^{ \varepsilon, \#_{I^\varpi_1}, [ \mathbf{u} ]_1^{ \varpi } }
        ,
        X_s^{ \varepsilon, \#_{I^\varpi_2}, [ \mathbf{u} ]_2^{ \varpi } }
        ,
        \dots
        ,
        X_s^{ \varepsilon, \#_{I^\varpi_{\#_\varpi}}, [\mathbf{u} ]_{ \#_\varpi }^{ \varpi } }
      \big)
    \Bigg]
  \,{\bf ds}
\\ &\!\!\!\!
  +
  \int_0^t
    e^{ ( t - s + \varepsilon ) A }
    \Bigg[
      \mathbbm{1}_{ \{ 0 \} }(k)
      \,
      B(X_s^{\varepsilon,0,u_0})
\\&\quad+
      \sum_{ \varpi\in \Pi_k }
      B^{ ( \#_\varpi ) }( X_s^{ \varepsilon,0,u_0 } )
      \big(
        X_s^{ \varepsilon, \#_{I^\varpi_1}, [ \mathbf{u} ]_1^{ \varpi } }
        ,
        X_s^{ \varepsilon, \#_{I^\varpi_2}, [ \mathbf{u} ]_2^{ \varpi } }
        ,
        \dots
        ,
        X_s^{ \varepsilon, \#_{I^\varpi_{\#_\varpi}}, [\mathbf{u} ]_{ \#_\varpi }^{ \varpi } }
      \big)
    \Bigg]
  \, \diffns W_s
  ,
\end{split}
\end{equation}
and let 
$
  \transmol{\varepsilon} \colon [0,T] \times H \to V
$, 
$ \varepsilon \in (0,T] $, 
be the functions which satisfy for all 
$ \varepsilon \in (0,T] $, 
$ t \in [0,T] $, 
$ x \in H $ 
that 
$
  \transmol{\varepsilon}(t,x)
  =
  \E[\varphi(X^{\varepsilon,0,x}_t)]
$.
Then 
\begin{enumerate}[(i)]
\item
\label{item:mollified.transition.derivative}
it holds for all 
$ \varepsilon \in (0,T] $, 
$ t \in [0,T] $
that 
$
  \big(
  H \ni x \mapsto
  \transmol{\varepsilon}(t,x) \in V
  \big)
  \in \Cb{n}(H,V)
$, 
\item
\label{item:mollified.transition.bound}
it holds for all 
$ k \in \{1,2,\ldots,n\} $, 
$ \delta_1, \delta_2, \ldots, \delta_k \in [0,\nicefrac{1}{2})  $ 
with 
$ \sum^k_{i=1} \delta_i < \nicefrac{1}{2} $ 
that 
\begin{equation}
  \sup_{ \varepsilon, t \in (0,T] }
  \sup_{ v \in H }
  \sup_{ \mathbf{u}=(u_1,u_2,\ldots,u_k) \in (\nzspace{H})^k }
  \left[
  \frac{
    t^{ \sum^k_{i=1} \delta_i } \,
    \big\|
    \big(
    \tfrac{\partial^k}{\partial x^k}
    \transmol{\varepsilon}
    \big)(t,v)
    \mathbf{u}
    \big\|_V
  }{
    \prod^k_{i=1} \|u_i\|_{ H_{-\delta_i} }
  }
  \right]
  < \infty, 
\end{equation}
and 
\item
\label{item:mollified.transition.lip}
it holds for all 
$ k \in \{1,2,\ldots,n\} $, 
$ \delta_1, \delta_2, \ldots, \delta_k \in [0,\nicefrac{1}{2})  $ 
with 
$ \sum^k_{i=1} \delta_i < \nicefrac{1}{2} $ 
and 
$
  |F|_{\operatorname{Lip}^k(H,H_{-\alpha})} +
  |B|_{\operatorname{Lip}^k(H,HS(U,H_{-\beta}))} + 
  |\varphi|_{\operatorname{Lip}^k(H,V)} < \infty 
$ 
that 
\begin{equation}
\begin{split}
&\!\!
  \sup_{ \varepsilon, t \in (0,T] }
  \sup_{\substack{ v,w \in H, \\ v \neq w }}
  \sup_{ \mathbf{u}=(u_1,u_2,\ldots,u_k) \in (\nzspace{H})^k }
  \left[
  \frac{
    t^{ \sum^k_{i=1} \delta_i } \,
    \big\|
  \big[
    \big(
    \tfrac{\partial^k}{\partial x^k}
    \transmol{\varepsilon}
    \big)(t,v)
      -
    \big(
    \tfrac{\partial^k}{\partial x^k}
    \transmol{\varepsilon}
    \big)(t,w)
  \big]
    \mathbf{u}
    \big\|_V
  }{
    \|v-w\|_H
    \prod^k_{i=1} \|u_i\|_{ H_{-\delta_i} }
  }
  \right]
< \infty. 
\end{split}
\end{equation}
\end{enumerate}
\end{lemma}
\begin{proof}
Throughout this proof let 
$ \deltaset{k} \in \mathcal{P}(\R^k) $, 
$ k \in \N $, 
be the sets which satisfy for all 
$ k \in \N $ 
that 
$
  \deltaset{k}=
  \{ 
    (\delta_1,\delta_2,\ldots,\delta_k) \in [0,\nicefrac{1}{2})^k \colon 
    \sum^k_{i=1} \delta_i < \nicefrac{1}{2}
  \}
$, 
let 
$
  \chi_r
  \in [0,\infty)
$, 
$ r \in [0,1] $, 
be the real numbers 
which satisfy for all 
$ r \in [0,1] $ 
that 
$
  \chi_r
  =
  \sup_{ t \in (0,T] }
    t^r
    \,
    \|
      ( \eta - A )^r
      e^{ t A }
    \|_{ L( H ) }
$ 
(see, e.g., Lemma~11.36 in Renardy \& Rogers~\cite{rr93}), 
let 
$
  \mathbbm{B} \colon (0,\infty)^2 \to (0,\infty)
$
be the function which satisfies for all $ x, y \in (0,\infty) $ that
$
  \mathbbm{B}( x, y ) = \int_0^1 t^{ (x - 1) } \left( 1 - t \right)^{ (y - 1) }
  \diffns t
$ 
(Beta function), 
let 
$
  \mathrm{E}_{ a, b } \colon [0,\infty) \to [0,\infty)
$, 
$ a, b \in (-\infty,1) $, 
be the functions which satisfy for all
$ a, b \in (-\infty,1) $, 
$
  x \in [0,\infty)
$
that
$
  \mathrm{E}_{ a, b }[ x ]
  =
  1
  +
  \sum_{ n = 1 }^{ \infty }
  x^n
  \prod_{ k = 0 }^{ n - 1 }
  \mathbb{B}\big(
    1-b
    ,
    k(1-b) + 1-a
  \big)
$
(generalized exponential function; cf.\ Chapter~7 in Henry~\cite{h81}
and~(16) in~\cite{AnderssonJentzenKurniawan2016arXiv}), 
let 
$F_\varepsilon\colon H \to H$, $\varepsilon\in(0,T]$, 
and 
$B_\varepsilon\colon H \to HS(U,H)$, $\varepsilon\in(0,T]$, 
be the functions which satisfy for all 
$\varepsilon\in(0,T]$, 
$x\in H$, 
$u\in U$ 
that 
\begin{equation}
  F_\varepsilon(x)
  =
  e^{\varepsilon A} F(x)
  \qquad
  \text{and}
  \qquad
  B_\varepsilon(x)u
  =
  e^{\varepsilon A} B(x)u
  ,
\end{equation}
and let 
$
  \Theta^{\lambda}_{p}
  \colon
  [0,\infty)^2 \to [0,\infty]
$, 
$ p \in [1,\infty) $, 
$ \lambda \in (-\infty,1) $, 
and 
$
\supertheta^{\lambda}_{p}
\colon
[0,\infty)^2 \to [0,\infty]
$, 
$ p \in [1,\infty) $, 
$ \lambda \in (-\infty,1) $, 
be the functions which satisfy for all 
$ \lambda \in (-\infty,1) $, 
$ p \in [1,\infty) $, 
$ L, \hat{L} \in [0,\infty) $
that
{\small
\begin{equation}
\begin{split}
&
  \Theta_{p}^{ \lambda }( L, \hat{L} )
  =
\\ &
\begin{cases}
  \sqrt{2}
  \,
  \bigg|
    E_{
      2 \lambda, \max\{ \alpha, 2 \beta \}
    }\bigg[
      \Big|
        \tfrac{
        \chi_\alpha\,
        L\,
          \sqrt{2}\,
          T^{ (1 - \alpha) }
        }
        {
          \sqrt{1-\alpha}
        }
        +
        \chi_\beta\,
        \hat{L}\,
        \sqrt{ p \, ( p - 1 ) \, T^{ (1 - 2\beta) } }
      \Big|^2
    \bigg]
  \bigg|^{1/2}
&
  \colon
  ( \lambda, \hat{L} )
  \in ( -\infty, \frac{ 1 }{ 2 } ) \times (0,\infty)
\\[1ex]
  E_{\lambda,\alpha}\!\left[
    \chi_\alpha\,
    L\,
    T^{ (1 - \alpha) }
  \right]
&
  \colon
  \hat{L} = 0
\\[1ex]
  \infty
&
  \colon
  \text{otherwise}
\end{cases}
\end{split}
\end{equation}
}(see, e.g., (17) in~\cite{AnderssonJentzenKurniawan2016arXiv})
and 
\begin{equation}
  \supertheta^\lambda_p(L,\hat{L})
  =
	  \sup_{x\in[0,L]}
	  \sup_{y\in[0,\hat{L}]}
	  \Theta^\lambda_p(x,y)  
	  .
\end{equation}
Note that for all 
	$ \lambda \in (-\infty,\nicefrac{1}{2}) $, 
	$ p \in [1,\infty) $, 
	$ L,\hat{L} \in [0,\infty) $ 
	it holds that 
\begin{equation}
\label{item:theta}
	  \supertheta^\lambda_p(L,\hat{L})
	  =
	  \max\!\big\{
	    \Theta^\lambda_p(L,\hat{L}),
	    \Theta^\lambda_p(L,0)
	  \big\}
	  < \infty.
\end{equation}
Moreover, observe that for all 
$ \varepsilon \in (0,T] $
it holds that 
\begin{equation}
\label{item:smooth}
  F_\varepsilon
  \in \Cb{n}(H,H)
  \qquad\text{and}\qquad
  B_\varepsilon \in \Cb{n}(H,HS(U,H))
  .
\end{equation}
In addition, note that for all 
$ k \in \{1,2,\ldots,n\} $, 
$ \varepsilon \in (0,T] $
it holds that 
\begin{equation}
\label{item:Cb}
\begin{split}
&
	  |F_\varepsilon|_{\Cb{k}(H,H_{-\alpha})}
	  \leq
	  \chi_0 \, |F|_{\Cb{k}(H,H_{-\alpha})} < \infty
	  \qquad\text{and}
\\&
	  |B_\varepsilon|_{\Cb{k}(H,HS(U,H_{-\beta}))}
	  \leq
	  \chi_0 \, |B|_{\Cb{k}(H,HS(U,H_{-\beta}))} < \infty.
\end{split}
\end{equation}
Furthermore, note that for all 
$ k \in \{0,1,\ldots,n\} $, 
$ \varepsilon \in (0,T] $
it holds that 
\begin{equation}
\label{item:Lip}
\begin{split}
&
	|F_\varepsilon|_{\operatorname{Lip}^k(H,H_{-\alpha})}
	\leq
	\chi_0 \, |F|_{\operatorname{Lip}^k(H,H_{-\alpha})}
	  \qquad\text{and}
\\&
	|B_\varepsilon|_{\operatorname{Lip}^k(H,HS(U,H_{-\beta}))}
	\leq
	\chi_0 \, |B|_{\operatorname{Lip}^k(H,HS(U,H_{-\beta}))}.  
\end{split}
\end{equation}
Item~\eqref{item:trans.smoothness} of Lemma~\ref{lem:derivative_formulas} 
(with 
$ n = n $, 
$ \varphi = \varphi $, 
$ F = F_\varepsilon $, 
$ B = B_\varepsilon $, 
$ X^{k,\mathbf{u}} = X^{\varepsilon,k,\mathbf{u}} $, 
$ \phi = \phi_\varepsilon $
$ t = t $
for 
$ t \in [0,T] $, 
$ \varepsilon \in (0,T] $, 
$ \mathbf{u} \in H^{ k + 1 } $, 
$ k \in \{0,1,\ldots,n\} $
in the notation of Lemma~\ref{lem:derivative_formulas})
and~\eqref{item:smooth} 
prove
item~\eqref{item:mollified.transition.derivative}.
Next we combine~\eqref{eq:SEE.derivative.mollified} and item~(iii) of Corollary~2.10 in~\cite{AnderssonJentzenKurniawan2016arXiv} 
(with 
$ H = H $, 
$ U = U $, 
$ T = T $, 
$ \eta = \eta $, 
$ \alpha = \alpha $, 
$ \beta = \beta $, 
$ W = W $, 
$ A = A $, 
$ F = (H \ni x \mapsto F_\varepsilon(x) \in H_{-\alpha}) $, 
$ B = (H \ni x \mapsto (U \ni u \mapsto B_\varepsilon(x)u \in H_{-\beta}) \in HS(U,H_{-\beta})) $, 
$ p = p $, 
$ \delta = 0 $
for 
$ \varepsilon \in (0,T] $, 
$ p \in [2,\infty) $
in the notation of Corollary~2.10 in~\cite{AnderssonJentzenKurniawan2016arXiv}) 
with~\eqref{item:theta} and~\eqref{item:Lip}
to obtain that for all $ p \in [2,\infty) $ it holds that 
\begin{equation}
\label{eq:lip.0th.derivative}
\begin{split}
&
\sup_{ \varepsilon, t \in (0,T] }
\sup_{\substack{x,y \in H, \\ x \neq y}}
\bigg[
\frac{\| X_t^{\varepsilon,0,x} - X_t^{\varepsilon,0,y} \|_{\lpn{p}{\P}{H}}}{\|x-y\|_H}
\bigg]
\\&\leq 
\chi_0
\sup_{\varepsilon\in(0,T]}
\Theta^0_p\big(|F_\varepsilon|_{\operatorname{Lip}^0(H,H_{-\alpha})},|B_\varepsilon|_{\operatorname{Lip}^0(H,HS(U,H_{-\beta}))}\big)
\\&\leq 
\chi_0
\sup_{\varepsilon\in(0,T]}
\supertheta^0_p\big(|F_\varepsilon|_{\operatorname{Lip}^0(H,H_{-\alpha})},|B_\varepsilon|_{\operatorname{Lip}^0(H,HS(U,H_{-\beta}))}\big)
\\&\leq 
\chi_0 \,
\supertheta^0_p\big(\chi_0 \, |F|_{\operatorname{Lip}^0(H,H_{-\alpha})},\chi_0 \, |B|_{\operatorname{Lip}^0(H,HS(U,H_{-\beta}))}\big)
<\infty
.
\end{split}
\end{equation}
Next we claim that 
\begin{enumerate}[(a)]
	\item
	\label{eq:apriori.higher.derivatives}
	it holds for all 
	$ k \in \{1,2,\ldots,n\} $, 
	$ p \in [2,\infty) $, 
	$ \boldsymbol{\delta} = (\delta_1,\delta_2,\ldots,\delta_k) \in \deltaset{k} $ 
	that 
	\begin{equation}
	\sup_{ \varepsilon, t \in (0,T] }
	\sup_{x\in H}
	\sup_{
		\mathbf{u}=(u_1,u_2,\ldots,u_k)
		\in (\nzspace{H})^k
	}
	\Bigg[
	\frac{
		t^{
			\iota^{\boldsymbol{\delta},\alpha,\beta}_\N
		}
		\|
		X_t^{\varepsilon,k,(x,\mathbf{u})}
		\|_{\mathcal{L}^{p}(\P;H)}
	}
	{
		\prod^k_{i=1}
		\|u_i\|_{H_{-\delta_i}}
	}
	\Bigg]
	< \infty
	\end{equation}
	and
	\item
	\label{eq:lip.higher.derivatives}
	it holds for all 
	$ k \in \{1,2,\ldots,n\} $, 
	$ p \in [2,\infty) $, 
	$ \boldsymbol{\delta} = (\delta_1,\delta_2,\ldots,\delta_k) \in \deltaset{k} $ 
	with 
	$ 
	|F|_{\operatorname{Lip}^k(H,H_{-\alpha})} 
	+ 
	|B|_{\operatorname{Lip}^k(H,HS(U,H_{-\beta}))} 
	< \infty 
	$
	that 
	\begin{equation}
	\sup_{ \varepsilon, t \in (0,T] }
	\sup_{\substack{x,y \in H, \\ x \neq y}}
	\sup_{
		\mathbf{u}=(u_1,u_2,\ldots,u_k)
		\in (\nzspace{H})^k
	}
	\Bigg[
	\frac{
		t^{ \iota^{(\boldsymbol{\delta},0),\alpha,\beta}_\N } \,
		\|X_t^{\varepsilon,k,(x,\mathbf{u})}-X_t^{\varepsilon,k,(y,\mathbf{u})}\|_{\lpn{p}{\P}{H}}
	}{
	\|x-y\|_H
	\prod^k_{i=1}
	\|u_i\|_{H_{-\delta_i}}
}
\Bigg]
< \infty.
\end{equation}
\end{enumerate}

We now prove item~\eqref{eq:apriori.higher.derivatives} and 
item~\eqref{eq:lip.higher.derivatives} by induction on 
$ k \in \{1,2,\ldots,n\} $. 
For the base case $k=1$ we combine ~\eqref{eq:SEE.derivative.mollified} and item~(ii) of 
Theorem~2.1 in~\cite{AnderssonJentzenKurniawan2016a} 
(with 
$ T = T $, 
$ \eta = \eta $, 
$ H = H $, 
$ U = U $, 
$ W = W $, 
$ A = A $, 
$ n = n $, 
$ F = F_\varepsilon $, 
$ B = B_\varepsilon $, 
$ \alpha = \alpha $, 
$ \beta = \beta $, 
$ k = 1 $, 
$ p = p $, 
$ \boldsymbol{\delta} = \delta $
for 
$ \varepsilon \in (0,T] $, 
$ \delta \in [0,\nicefrac{1}{2}) $, 
$ p \in [2,\infty) $, 
in the notation of Theorem~2.1 in~\cite{AnderssonJentzenKurniawan2016a})
with~\eqref{item:theta}--\eqref{item:Cb}
to obtain that for all 
$ p \in [2,\infty) $, 
$ \delta \in [0,\nicefrac{1}{2}) $
it holds that 
\begin{equation}
\label{eq:mollified.apriori.1}
\begin{split}
&
\sup_{ \varepsilon, t \in (0,T] }
\sup_{x\in H}
\sup_{
	u \in \nzspace{H}
}
\Bigg[
\frac{
	t^{
		\delta
	} \,
	\|
	X_t^{\varepsilon,1,(x,u)}
	\|_{\mathcal{L}^{p}(\P;H)}
}
{
	\|u\|_{H_{-\delta}}
}
\Bigg]
\\&\leq
\chi_\delta
\sup_{\varepsilon\in(0,T]}
\Theta_p^\delta( |F_\varepsilon|_{ \Cb{1}( H, H_{-\alpha} ) } 
, 
|B_\varepsilon|_{ \Cb{1}( H, HS( U, H_{-\beta} ) ) } )
\\&\leq
\chi_\delta
\sup_{\varepsilon\in(0,T]}
\supertheta_p^\delta( |F_\varepsilon|_{ \Cb{1}( H, H_{-\alpha} ) } 
, 
|B_\varepsilon|_{ \Cb{1}( H, HS( U, H_{-\beta} ) ) } )
\\&\leq
\chi_\delta \,
\supertheta_p^\delta( \chi_0 \, |F|_{ \Cb{1}( H, H_{-\alpha} ) } 
, 
\chi_0 \, |B|_{ \Cb{1}( H, HS( U, H_{-\beta} ) ) } )
< \infty.
\end{split}
\end{equation}
Moreover, combining~\eqref{eq:SEE.derivative.mollified} and item~(iv) of Theorem~2.1 in~\cite{AnderssonJentzenKurniawan2016a} 
(with 
$ T = T $, 
$ \eta = \eta $, 
$ H = H $, 
$ U = U $, 
$ W = W $, 
$ A = A $, 
$ n = n $, 
$ F = F_\varepsilon $, 
$ B = B_\varepsilon $, 
$ \alpha = \alpha $, 
$ \beta = \beta $, 
$ k = 1 $, 
$ p = p $, 
$ \boldsymbol{\delta} = \delta $
for 
$ \varepsilon \in (0,T] $, 
$ \delta \in [0,\nicefrac{1}{2}) $, 
$ 
p \in \{r\in[2,\infty)\colon
|F|_{\operatorname{Lip}^1(H,H_{-\alpha})} 
+ 
|B|_{\operatorname{Lip}^1(H,HS(U,H_{-\beta}))} < \infty 
\} 
$ 
in the notation of Theorem~2.1 
in~\cite{AnderssonJentzenKurniawan2016a})
with~\eqref{item:theta}--\eqref{item:Lip} and~\eqref{eq:mollified.apriori.1}
assures that for all 
$ p \in [2,\infty) $, 
$ \delta \in [0,\nicefrac{1}{2}) $ 
with 
$ 
|F|_{\operatorname{Lip}^1(H,H_{-\alpha})} 
+ 
|B|_{\operatorname{Lip}^1(H,HS(U,H_{-\beta}))} < \infty 
$
it holds that 
\begin{equation}
\begin{split}
&
\sup_{ \varepsilon, t \in (0,T] }
\sup_{\substack{x,y \in H, \\ x \neq y}}
\sup_{
	u \in \nzspace{H}
}
\Bigg[
\frac{
	t^{ \iota^{(\delta,0),\alpha,\beta}_\N } \,
	\|X_t^{\varepsilon,1,(x,u)}-X_t^{\varepsilon,1,(y,u)}\|_{\lpn{p}{\P}{H}}
}{
\|x-y\|_H \,
\|u\|_{H_{-\delta}}
}
\Bigg]
\\&\leq
| T \vee 1 |
\sup_{\varepsilon\in(0,T]}
\Bigg\{
\Theta_p^{ 
	\iota^{(\delta,0),\alpha,\beta}_\N
}\!\big(
| F_\varepsilon |_{ \Cb{1}( H , H_{-\alpha} ) }
,
| B_\varepsilon |_{ \Cb{1}( H , HS( U , H_{-\beta} ) ) }
\big)
\\&\cdot
\chi_0\,
\Theta^0_p
\big(
| F_\varepsilon |_{ \Cb{1}( H , H_{-\alpha} ) }
,
| B_\varepsilon |_{ \Cb{1}( H , HS( U , H_{-\beta} ) ) }
\big)
\sup_{ t \in (0,T] }
\sup_{x\in H}
\sup_{
	u \in \nzspace{H}
}
\Bigg[
\frac{
	t^\delta\,
	\|
	X_t^{ \varepsilon,1,(x,u) }
	\|_{ \lpn{2p}{\P}{H} }
}{
\| u \|_{ H_{ -\delta } }
}
\Bigg]
\\&\cdot
\bigg[
\chi_\alpha \,
\mathbbm{B}\big(
1 - \alpha
,
1 - \delta
\big)
\| F_\varepsilon \|_{ 
	\operatorname{Lip}^1( H, H_{-\alpha} ) 
}
+
\chi_\beta \,
\sqrt{
	\tfrac{p \, ( p - 1 )}{2}
	\,
	\mathbbm{B}\big(
	1 - 2\beta
	,
	1 
	- 
	2 \delta
	\big)
} \,
\| B_\varepsilon \|_{ 
	\operatorname{Lip}^1( H, HS( U, H_{-\beta} ) ) 
}
\bigg]
\Bigg\}
\\&\leq
| T \vee 1 | \, |\chi_0|^2 \,
\supertheta_p^{ 
	\iota^{(\delta,0)}_\N
}\!\big(
\chi_0 \, | F |_{ \Cb{1}( H , H_{-\alpha} ) }
,
\chi_0 \, | B |_{ \Cb{1}( H , HS( U , H_{-\beta} ) ) }
\big)
\\&\cdot
\supertheta^0_p
\big(
\chi_0 \, | F |_{ \Cb{1}( H , H_{-\alpha} ) }
,
\chi_0 \, | B |_{ \Cb{1}( H , HS( U , H_{-\beta} ) ) }
\big)
\sup_{ \varepsilon, t \in (0,T] }
\sup_{x\in H}
\sup_{
	u \in \nzspace{H}
}
\Bigg[
\frac{
	t^\delta\,
	\|
	X_t^{ \varepsilon,1,(x,u) }
	\|_{ \lpn{2p}{\P}{H} }
}{
\| u \|_{ H_{ -\delta } }
}
\Bigg]
\\&\cdot
\bigg[
\chi_\alpha \,
\mathbbm{B}\big(
1 - \alpha
,
1 - \delta
\big)
\| F \|_{ 
	\operatorname{Lip}^1( H, H_{-\alpha} ) 
}
+
\chi_\beta \,
\sqrt{
	\tfrac{p \, ( p - 1 )}{2}
	\,
	\mathbbm{B}\big(
	1 - 2\beta
	,
	1 
	- 
	2 \delta
	\big)
} \,
\| B \|_{ 
	\operatorname{Lip}^1( H, HS( U, H_{-\beta} ) ) 
}
\bigg]
\\&< \infty.
\end{split}
\end{equation}
This and~\eqref{eq:mollified.apriori.1} establish item~\eqref{eq:apriori.higher.derivatives} and item~\eqref{eq:lip.higher.derivatives} in the base case $k=1$.
For the induction step 
$ \{1,2,\ldots,n-1\} \ni k \to k+1 \in \{2,3,\ldots,n\} $ 
assume that there exists a natural number 
$ k \in \{1,2,\ldots,n-1\} $
such that item~\eqref{eq:apriori.higher.derivatives} and item~\eqref{eq:lip.higher.derivatives} hold for all 
$ l \in \{1,2,\ldots,k\} $. 
Observe that~\eqref{eq:SEE.derivative.mollified}, item~(ii) of Theorem~2.1 in~\cite{AnderssonJentzenKurniawan2016a} 
(with 
$ T = T $, 
$ \eta = \eta $, 
$ H = H $, 
$ U = U $, 
$ W = W $, 
$ A = A $, 
$ n = n $, 
$ F = F_\varepsilon $, 
$ B = B_\varepsilon $, 
$ \alpha = \alpha $, 
$ \beta = \beta $, 
$ k = k+1 $, 
$ p = p $, 
$ \boldsymbol{\delta} = \boldsymbol{\delta} $
for 
$ \varepsilon \in (0,T] $, 
$ \boldsymbol{\delta} \in \deltaset{k+1} $, 
$ p \in [2,\infty) $
in the notation of Theorem~2.1 
in~\cite{AnderssonJentzenKurniawan2016a}), 
the induction step, and~\eqref{item:theta}--\eqref{item:Cb}
imply that for all 
$ p \in [2,\infty) $, 
$ \boldsymbol{\delta} = (\delta_1,\delta_2,\ldots,\delta_{k+1}) \in \deltaset{k+1} $ 
it holds that 
\begin{equation}
\label{eq:mollified.apriori.n}
\begin{split}
&
\sup_{ \varepsilon, t \in (0,T] }
\sup_{x\in H}
\sup_{
	\mathbf{u}=(u_1,u_2,\ldots,u_{k+1})
	\in (\nzspace{H})^{k+1}
}
\Bigg[
\frac{
	t^{
		\iota^{\boldsymbol{\delta},\alpha,\beta}_\N
	}
	\|
	X_t^{\varepsilon,k+1,(x,\mathbf{u})}
	\|_{\mathcal{L}^{p}(\P;H)}
}
{
	\prod^{k+1}_{i=1}
	\|u_i\|_{H_{-\delta_i}}
}
\Bigg]
\\&\leq
|T\vee 1|^{k+1}
\sup_{\varepsilon\in(0,T]}
\bigg\{
\Theta_p^{
	\iota^{ \boldsymbol{\delta} }_\N
}( |F_\varepsilon|_{ \Cb{1}( H, H_{-\alpha} ) } 
, 
|B_\varepsilon|_{ \Cb{1}( H, HS( U, H_{-\beta} ) ) } )
\\&\quad\cdot
\bigg[
\chi_\alpha \,
\mathbbm{B}\big(
1 - \alpha
,
1 - \smallsum_{ i = 1 }^{ k+1 } \delta_i
\big)
\| F_\varepsilon \|_{ 
	\Cb{ k+1 }( H, H_{-\alpha} ) 
}
\\&\quad+
\chi_\beta \,
\sqrt{
	\tfrac{p \, ( p - 1 )}{2}
	\,
	\mathbbm{B}\big(
	1 - 2\beta
	,
	1 
	- 
	2 \smallsum_{ i = 1 }^{ k+1 } \delta_i
	\big)
} \,
\| B_\varepsilon \|_{ 
	\Cb{ k+1 }( H, HS( U, H_{-\beta} ) ) 
}
\bigg]
\\ & \quad
\cdot
\sum_{ \varpi \in \Pi_{k+1}^{ * } }
\prod_{ I \in \varpi }
\sup\limits_{ t \in (0,T] }
\sup_{ x \in H }
\sup_{ \mathbf{u} = ( u_i )_{ i \in I  } \in (\nzspace{H})^{\#_I} }
\bigg[
\displaystyle
\frac{
	t^{
		\iota^{ \boldsymbol{\delta},\alpha,\beta }_I
	} \,
	\|
	X_t^{ \varepsilon, \#_I, (x,\mathbf{u}) }
	\|_{
		\mathcal{L}^{ p \, \#_\varpi }( \P ; H )
	}
}{
\prod_{ i \in I }
\| u_i \|_{ H_{ - \delta_i } }
}
\bigg]
\bigg\}
\\&\leq
|T\vee 1|^{k+1} \, \chi_0 \,
\supertheta_p^{
	\iota^{ \boldsymbol{\delta} }_\N
}( \chi_0 \, |F|_{ \Cb{1}( H, H_{-\alpha} ) } 
, 
\chi_0 \, |B|_{ \Cb{1}( H, HS( U, H_{-\beta} ) ) } )
\\&\quad\cdot
\bigg[
\chi_\alpha \,
\mathbbm{B}\big(
1 - \alpha
,
1 - \smallsum_{ i = 1 }^{ k+1 } \delta_i
\big)
\| F \|_{ 
	\Cb{ k+1 }( H, H_{-\alpha} ) 
}
\\&\quad+
\chi_\beta \,
\sqrt{
	\tfrac{p \, ( p - 1 )}{2}
	\,
	\mathbbm{B}\big(
	1 - 2\beta
	,
	1 
	- 
	2 \smallsum_{ i = 1 }^{ k+1 } \delta_i
	\big)
} \,
\| B \|_{ 
	\Cb{ k+1 }( H, HS( U, H_{-\beta} ) ) 
}
\bigg]
\\ & \quad
\cdot
\sum_{ \varpi \in \Pi_{k+1}^{ * } }
\prod_{ I \in \varpi }
\sup\limits_{ \varepsilon, t \in (0,T] }
\sup_{ x \in H }
\sup_{ \mathbf{u} = ( u_i )_{ i \in I  } \in (\nzspace{H})^{\#_I} }
\bigg[
\displaystyle
\frac{
	t^{
		\iota^{ \boldsymbol{\delta},\alpha,\beta }_I
	} \,
	\|
	X_t^{ \varepsilon, \#_I, (x,\mathbf{u}) }
	\|_{
		\mathcal{L}^{ p \, \#_\varpi }( \P ; H )
	}
}{
\prod_{ i \in I }
\| u_i \|_{ H_{ - \delta_i } }
}
\bigg]
< \infty.
\end{split}
\end{equation}
Furthermore, note that~\eqref{eq:SEE.derivative.mollified}, item~(iv) of Theorem~2.1 in~\cite{AnderssonJentzenKurniawan2016a} 
(with 
$ T = T $, 
$ \eta = \eta $, 
$ H = H $, 
$ U = U $, 
$ W = W $, 
$ A = A $, 
$ n = n $, 
$ F = F_\varepsilon $, 
$ B = B_\varepsilon $, 
$ \alpha = \alpha $, 
$ \beta = \beta $, 
$ k = k+1 $, 
$ p = p $, 
$ \boldsymbol{\delta} = \boldsymbol{\delta} $
for 
$ \varepsilon \in (0,T] $, 
$ \boldsymbol{\delta} \in \deltaset{k+1} $, 
$ p \in \{r\in[2,\infty)\colon
|F|_{\operatorname{Lip}^{k+1}(H,H_{-\alpha})} 
+ 
|B|_{\operatorname{Lip}^{k+1}(H,HS(U,H_{-\beta}))} < \infty 
\} $ 
in the notation of Theorem~2.1 
in~\cite{AnderssonJentzenKurniawan2016a}), 
\eqref{item:smooth}, and \eqref{item:Lip}
imply that for all 
$ p \in [2,\infty) $, 
$ \boldsymbol{\delta} = (\delta_1,\delta_2,\ldots,\delta_{k+1}) \in \deltaset{k+1} $ 
with 
$ 
|F|_{\operatorname{Lip}^{k+1}(H,H_{-\alpha})} 
+ 
|B|_{\operatorname{Lip}^{k+1}(H,HS(U,H_{-\beta}))} < \infty 
$
it holds that 
\begin{equation}
\begin{split}
&
\sup_{ \varepsilon, t \in (0,T] }
\sup_{\substack{x,y \in H, \\ x \neq y}}
\sup_{
	\mathbf{u}=(u_1,u_2,\ldots,u_{k+1})
	\in (\nzspace{H})^{k+1}
}
\Bigg[
\frac{
	t^{ \iota^{(\boldsymbol{\delta},0),\alpha,\beta}_\N } \,
	\|X_t^{\varepsilon,k+1,(x,\mathbf{u})}-X_t^{\varepsilon,k+1,(y,\mathbf{u})}\|_{\lpn{p}{\P}{H}}
}{
\|x-y\|_H
\prod^{k+1}_{i=1}
\|u_i\|_{H_{-\delta_i}}
}
\Bigg]
\\&\leq
| T \vee 1 |^{k+1}
\sup_{\varepsilon\in(0,T]}
\Bigg\{
\Theta_p^{ 
	\iota^{(\boldsymbol{\delta},0)}_\N
}\big(
| F_\varepsilon |_{ \Cb{1}( H , H_{-\alpha} ) }
,
| B_\varepsilon |_{ \Cb{1}( H , HS( U , H_{-\beta} ) ) }
\big)
\\&\cdot
\Bigg(
\chi_0\,
\Theta^0_p
\big(
| F_\varepsilon |_{ \Cb{1}( H , H_{-\alpha} ) }
,
| B_\varepsilon |_{ \Cb{1}( H , HS( U , H_{-\beta} ) ) }
\big)
\\&\cdot
\sum_{
	\varpi \in \Pi_{k+1}
}
\prod_{ I\in\varpi }
\sup_{ t\in (0,T] }
\sup_{ x \in H }
\sup_{ \mathbf{u} = ( u_i )_{ i \in I  } \in (\nzspace{H})^{\#_I} }
\bigg[
\displaystyle
\frac{
	t^{
		\iota^{ \boldsymbol{\delta},\alpha,\beta }_I
	} \,
	\|
	X_t^{ \varepsilon, \#_I, (x,\mathbf{u}) }
	\|_{
		\mathcal{L}^{ p (\#_\varpi+1) }( \P ; H )
	}
}{
\prod_{ i \in I }
\| u_i \|_{ H_{ - \delta_i } }
}
\bigg]
\\&+
\sum_{ \varpi \in \Pi^*_{k+1} }
\sum_{ I \in \varpi }
\sup_{ t \in (0,T] }
\sup_{\substack{
		x,y \in H,
		\\
		x\neq y
	}}
	\sup_{\mathbf{u}=(u_i)_{i\in I}\in (\nzspace{H})^{\#_I}}
	\Bigg[
	\frac{
		t^{
			\iota_{ I \cup \{k+2\} }^{ (\boldsymbol{\delta},0),\alpha,\beta } 
		}
		\|
		X_t^{ \varepsilon,\#_I, ( x, \mathbf{u} ) }
		-
		X_t^{ \varepsilon,\#_I, ( y, \mathbf{u} ) }
		\|_{
			\mathcal{L}^{ p \#_\varpi }( \P ; H )
		}
	}{
	\| x-y \|_{ H }
	\prod_{i\in I}
	\| u_i \|_{ H_{ - \delta_i } }
}
\Bigg]
\\&\cdot
\prod_{ J\in\varpi\setminus\{I\} }
\sup_{ t\in (0,T] }
\sup_{ x \in H }
\sup_{ \mathbf{u} = ( u_i )_{ i \in J  } \in (\nzspace{H})^{\#_J} }
\Bigg[
\frac{
	t^{ \iota^{ \boldsymbol{\delta},\alpha,\beta }_J }\,
	\|
	X_t^{ \varepsilon,\#_J, (x,\mathbf{u}) }
	\|_{ \lpn{p\#_\varpi}{\P}{H} }
}{
\prod_{ i\in J }
\| u_i \|_{ H_{ -\delta_i } }
}
\Bigg]
\Bigg)
\\&\cdot
\bigg[
\chi_\alpha\,
\mathbb{B}\big(
1-\alpha
,
1-\smallsum^{k+1}_{ i=1 }\delta_i
\big) \,
\|F_\varepsilon\|_{ \operatorname{Lip}^{k+1}(H,H_{-\alpha}) }
\\&+
\chi_\beta\,
\sqrt{
	\tfrac{p\,(p-1)}{2}
	\,
	\mathbb{B}\big(
	1-2\beta
	,
	1
	-
	2\smallsum^{k+1}_{i=1}\delta_i
	\big)
}\,
\|B_\varepsilon\|_{ \operatorname{Lip}^{k+1}(H,HS(U,H_{-\beta})) }
\bigg]
\Bigg\}
.
\end{split}
\end{equation}
This, the induction hypothesis,
\eqref{item:theta}, \eqref{item:Cb}, \eqref{item:Lip}, 
and~\eqref{eq:mollified.apriori.n}
imply that for all 
$ p \in [2,\infty) $, 
$ \boldsymbol{\delta} = (\delta_1,\delta_2,\ldots,\delta_{k+1}) \in \deltaset{k+1} $ 
with 
$ 
|F|_{\operatorname{Lip}^{k+1}(H,H_{-\alpha})} 
+ 
|B|_{\operatorname{Lip}^{k+1}(H,HS(U,H_{-\beta}))} < \infty 
$
it holds that 
\begin{equation}
\label{eq:mollified.lip.n}
\begin{split}
&
\sup_{ \varepsilon, t \in (0,T] }
\sup_{\substack{x,y \in H, \\ x \neq y}}
\sup_{
	\mathbf{u}=(u_1,u_2,\ldots,u_{k+1})
	\in (\nzspace{H})^{k+1}
}
\Bigg[
\frac{
	t^{ \iota^{(\boldsymbol{\delta},0),\alpha,\beta}_\N } \,
	\|X_t^{\varepsilon,k+1,(x,\mathbf{u})}-X_t^{\varepsilon,k+1,(y,\mathbf{u})}\|_{\lpn{p}{\P}{H}}
}{
\|x-y\|_H
\prod^{k+1}_{i=1}
\|u_i\|_{H_{-\delta_i}}
}
\Bigg]
\\&\leq
| T \vee 1 |^{k+1} \, \chi_0
\,
\supertheta_p^{ 
	\iota^{(\boldsymbol{\delta},0)}_\N
}\!\big(
\chi_0 \, | F |_{ \Cb{1}( H , H_{-\alpha} ) }
,
\chi_0 \, | B |_{ \Cb{1}( H , HS( U , H_{-\beta} ) ) }
\big)
\\&\cdot
\Bigg(
\chi_0\,
\supertheta^0_p
\big(
\chi_0 \, | F |_{ \Cb{1}( H , H_{-\alpha} ) }
,
\chi_0 \, | B |_{ \Cb{1}( H , HS( U , H_{-\beta} ) ) }
\big)
\\&\cdot
\sum_{
	\varpi \in \Pi_{k+1}
}
\prod_{ I\in\varpi }
\sup_{ \varepsilon,t\in (0,T] }
\sup_{ x \in H }
\sup_{ \mathbf{u} = ( u_i )_{ i \in I  } \in (\nzspace{H})^{\#_I} }
\Bigg[
\frac{
	t^{ \iota^{ \boldsymbol{\delta},\alpha,\beta }_I }\,
	\|
	X_t^{ \varepsilon, \#_I, (x,\mathbf{u}) }
	\|_{ \lpn{p(\#_\varpi+1)}{\P}{H} }
}{
\prod_{ i\in I }
\| u_i \|_{ H_{ -\delta_i } }
}
\Bigg]
\\&+
\sum_{ \varpi \in \Pi^*_{k+1} }
\sum_{ I \in \varpi }
\sup_{ \varepsilon, t \in (0,T] }
\sup_{\substack{
		x,y \in H,
		\\
		x\neq y
	}}
	\sup_{\mathbf{u}=(u_i)_{i\in I}\in (\nzspace{H})^{\#_I}}
	\Bigg[
	\frac{
		t^{
			\iota_{ I \cup \{k+2\} }^{ (\boldsymbol{\delta},0),\alpha,\beta } 
		}
		\|
		X_t^{ \varepsilon, \#_I, ( x, \mathbf{u} ) }
		-
		X_t^{ \varepsilon, \#_I, ( y, \mathbf{u} ) }
		\|_{
			\mathcal{L}^{ p \#_\varpi }( \P ; H )
		}
	}{
	\| x-y \|_{ H }
	\prod_{i\in I}
	\| u_i \|_{ H_{ - \delta_i } }
}
\Bigg]
\\&\cdot
\prod_{ J\in\varpi\setminus\{I\} }
\sup_{ \varepsilon, t\in (0,T] }
\sup_{ x \in H }
\sup_{ \mathbf{u} = ( u_i )_{ i \in J  } \in (\nzspace{H})^{\#_J} }
\Bigg[
\frac{
	t^{ \iota^{ \boldsymbol{\delta},\alpha,\beta }_J }\,
	\|
	X_t^{ \varepsilon, \#_J, (x,\mathbf{u}) }
	\|_{ \lpn{p\#_\varpi}{\P}{H} }
}{
\prod_{ i\in J }
\| u_i \|_{ H_{ -\delta_i } }
}
\Bigg]
\Bigg)
\\&\cdot
\bigg[
\chi_\alpha\,
\mathbb{B}\big(
1-\alpha
,
1-\smallsum^{k+1}_{ i=1 }\delta_i
\big) \,
\|F\|_{ \operatorname{Lip}^{k+1}(H,H_{-\alpha}) }
\\&+
\chi_\beta\,
\sqrt{
	\tfrac{p\,(p-1)}{2}
	\,
	\mathbb{B}\big(
	1-2\beta
	,
	1
	-
	2\smallsum^{k+1}_{i=1}\delta_i
	\big)
}\,
\|B\|_{ \operatorname{Lip}^{k+1}(H,HS(U,H_{-\beta})) }
\bigg]
< \infty.
\end{split}
\end{equation}
Combining~\eqref{eq:mollified.apriori.n} with~\eqref{eq:mollified.lip.n}
proves item~\eqref{eq:apriori.higher.derivatives} and item~\eqref{eq:lip.higher.derivatives} in the case $k+1$. 
Induction hence establishes item~\eqref{eq:apriori.higher.derivatives} and item~\eqref{eq:lip.higher.derivatives}.

Next note that item~\eqref{item:thm.derivative.bound} of Lemma~\ref{lem:derivative_formulas} 
(with 
$ n = n $, 
$ \varphi = \varphi $, 
$ F = F_\varepsilon $, 
$ B = B_\varepsilon $, 
$ X^{m,\mathbf{u}} = X^{\varepsilon,m,\mathbf{u}} $, 
$ \phi=\phi_\varepsilon $, 
$ k = k $, 
$ \boldsymbol{\delta} = \boldsymbol{\delta} $, 
$ \alpha = \alpha $, 
$ \beta = \beta $
for 
$ \boldsymbol{\delta} \in \deltaset{k} $, 
$ k \in \{1,2,\ldots,n\} $, 
$ \mathbf{u} \in H^{ m + 1 } $, 
$ m \in \{0,1,\ldots,n\} $, 
$ \varepsilon \in (0,T] $ 
in the notation of Lemma~\ref{lem:derivative_formulas}), 
\eqref{item:smooth},
item~\eqref{eq:apriori.higher.derivatives}, and Jensen's inequality
ensure that for all 
$k\in\{1,2,\ldots,n\}$, 
$ \boldsymbol{\delta}=(\delta_1,\delta_2,\ldots,\delta_k) \in \deltaset{k} $ 
it holds that 
\begin{equation}
\begin{split}
&
\sup_{ \varepsilon, t \in (0,T] }
\sup_{ v \in H }
\sup_{ \mathbf{u}=(u_1,u_2,\ldots,u_k) \in (\nzspace{H})^k }
\left[
\frac{
	t^{ \sum^k_{i=1} \delta_i } \,
	\big\|
	\big(
	\tfrac{\partial^k}{\partial x^k}
	\transmol{\varepsilon}
	\big)(t,v)
	\mathbf{u}
	\big\|_V
}{
\prod^k_{i=1} \|u_i\|_{ H_{-\delta_i} }
}
\right]
\\&\leq
|T \vee 1|^{
	\lfloor k/2 \rfloor \,
	\min\{1-\alpha,\nicefrac{1}{2}-\beta\}
} \,
\|
\varphi
\|_{
	\Cb{k}(H,{V})
}
\\&\cdot
\sum_{
	\varpi \in \Pi_k
}
\prod_{
	I \in \varpi
}
\sup_{\varepsilon, t\in(0,T]}
\sup_{x\in H}
\sup_{
	\mathbf{u}=(u_i)_{i\in I}
	\in (\nzspace{H})^{\#_I}
}
\Bigg[
\frac{
	t^{
		\iota^{\boldsymbol{\delta},\alpha,\beta}_I
	}
	\|
	X_t^{\varepsilon,\#_I,(x,\mathbf{u})}
	\|_{\mathcal{L}^{\#_\varpi}(\P;H)}
}
{
	\prod_{i\in I}
	\|u_i\|_{H_{-\delta_i}}
}
\Bigg]
<\infty.
\end{split}
\end{equation}
This proves item~\eqref{item:mollified.transition.bound}.
It thus remains to prove item~\eqref{item:mollified.transition.lip}. 
For this we combine 
item~\eqref{item:trans.lip} of Lemma~\ref{lem:derivative_formulas} 
(with 
$ n = n $, 
$ \varphi = \varphi $, 
$ F = F_\varepsilon $, 
$ B = B_\varepsilon $, 
$ X^{m,\mathbf{u}} = X^{\varepsilon,m,\mathbf{u}} $, 
$ \phi=\phi_\varepsilon $, 
$ k = k $, 
$ \boldsymbol{\delta} = \boldsymbol{\delta} $, 
$ \alpha = \alpha $, 
$ \beta = \beta $
for 
$ \boldsymbol{\delta} \in \deltaset{k} $, 
$ k \in \{l\in\{1,2,\ldots,n\}\colon
|F|_{\operatorname{Lip}^l(H,H_{-\alpha})} 
+ 
|B|_{\operatorname{Lip}^l(H,HS(U,H_{-\beta}))} 
+
|\varphi|_{\operatorname{Lip}^l(H,V)} 
< \infty 
\} $, 
$ \mathbf{u} \in H^{ m + 1 } $, 
$ m \in \{0,1,\ldots,n\} $, 
$ \varepsilon \in (0,T] $
in the notation of Lemma~\ref{lem:derivative_formulas})
with~\eqref{item:smooth}, \eqref{item:Lip}, \eqref{eq:lip.0th.derivative}, item~\eqref{eq:apriori.higher.derivatives}, item~\eqref{eq:lip.higher.derivatives}, 
and Jensen's inequality to obtain that for all 
$k\in\{1,2,\ldots,n\}$, 
$ \boldsymbol{\delta}=(\delta_1,\delta_2,\ldots,\delta_k) \in \deltaset{k} $ 
with 
$ 
|F|_{\operatorname{Lip}^k(H,H_{-\alpha})} 
+ 
|B|_{\operatorname{Lip}^k(H,HS(U,H_{-\beta}))} 
+
|\varphi|_{\operatorname{Lip}^k(H,V)} 
< \infty 
$
it holds that 
\begin{equation}
\begin{split}
&
\sup_{ \varepsilon, t \in (0,T] }
\sup_{\substack{ v,w \in H, \\ v \neq w }}
\sup_{ \mathbf{u}=(u_1,u_2,\ldots,u_k) \in (\nzspace{H})^k }
\left[
\frac{
	t^{ \sum^k_{i=1} \delta_i } \,
	\big\|
	\big[
	\big(
	\tfrac{\partial^k}{\partial x^k}
	\transmol{\varepsilon}
	\big)(t,v)
	-
	\big(
	\tfrac{\partial^k}{\partial x^k}
	\transmol{\varepsilon}
	\big)(t,w)
	\big]
	\mathbf{u}
	\big\|_V
}{
\|v-w\|_H
\prod^k_{i=1} \|u_i\|_{ H_{-\delta_i} }
}
\right]
\\&\leq
|T\vee 1|^{
	\lceil k/2 \rceil \, 
	\min\{1-\alpha,\nicefrac{1}{2}-\beta\}
} \,
\|\varphi\|_{ \operatorname{Lip}^k(H,V) }
\\&\cdot
\sum_{
	\varpi \in \Pi_k
}
\Bigg\{
\sup_{ \varepsilon, t \in (0,T] }
\sup_{\substack{x,y \in H, \\ x \neq y}}
\Bigg[
\frac{\| X_t^{\varepsilon,0,x} - X_t^{\varepsilon,0,y} \|_{\lpn{\#_\varpi+1}{\P}{H}}}{\|x-y\|_H}
\Bigg]
\\&\cdot
\prod_{I\in\varpi} 
\sup_{ \varepsilon,t \in (0,T] }
\sup_{ x \in H }
\sup_{ \mathbf{u}=(u_i)_{i\in I} \in (\nzspace{H})^{\#_I} }
\Bigg[
\frac{
	t^{ \iota^{\boldsymbol{\delta},\alpha,\beta}_I } \,
	\|X_t^{\varepsilon,\#_I,(x,\mathbf{u})}\|_{\lpn{\#_\varpi+1}{\P}{H}}
}{
\prod_{i\in I}
\|u_i\|_{H_{-\delta_i}}
}
\Bigg]
\\&+
\sum_{I\in\varpi}
\sup_{ \varepsilon,t \in (0,T] }
\sup_{\substack{x,y \in H, \\ x \neq y}}
\sup_{ \mathbf{u}=(u_i)_{i\in I} \in (\nzspace{H})^{\#_I} }
\Bigg[
\frac{
	t^{ \iota^{(\boldsymbol{\delta},0),\alpha,\beta}_{I\cup\{k+1\}} } \,
	\|X_t^{\varepsilon,\#_I,(x,\mathbf{u})}-X_t^{\varepsilon,\#_I,(y,\mathbf{u})}\|_{\lpn{\#_\varpi}{\P}{H}}
}{
\|x-y\|_H
\prod_{i\in I}
\|u_i\|_{H_{-\delta_i}}
}
\Bigg]
\\&\cdot
\prod_{ J \in \varpi \setminus \{I\} } 
\sup_{ \varepsilon,t \in (0,T] }
\sup_{ x \in H }
\sup_{ \mathbf{u}=(u_i)_{i\in J} \in (\nzspace{H})^{\#_J} }
\Bigg[
\frac{
	t^{ \iota^{\boldsymbol{\delta},\alpha,\beta}_J } \,
	\|X_t^{\varepsilon,\#_J,(x,\mathbf{u})}\|_{\lpn{\#_\varpi}{\P}{H}}
}{
\prod_{i\in J}
\|u_i\|_{H_{-\delta_i}}
}
\Bigg]
\Bigg\}
< \infty.
\end{split}
\end{equation}
This proves item~\eqref{item:mollified.transition.lip}. 
The proof of Theorem~\ref{thm:mollified.transition} is thus completed.
\end{proof}

\begin{corollary}
	\label{thm:mollified.transition}
	Assume the setting in Section~\ref{sec:global_setting} 
	and let $ n \in \N $, 
	$ \alpha \in [0,1) $, 
	$ \beta \in [0,\nicefrac{1}{2}) $, 
	$ F \in \Cb{n}(H,H_{-\alpha}) $,
	$ B \in \Cb{n}(H,HS(U,H_{-\beta})) $, 
	$ \varphi \in \Cb{n}(H,V) $. 
	Then 
	\begin{enumerate}[(i)]
		\item
		\label{item:cor.mollified.existence}
		it holds that there exist up-to-modifications unique $(\mathcal{F}_t)_{t\in[0,T]}$/$ \mathcal{B}(H) $-predictable stochastic processes
		$
		X^{ \varepsilon, x }
		\colon
		[ 0 , T ] \times \Omega
		\to H
		$, 
		$ \varepsilon \in (0,T] $, 
		$
		x \in H
		$, 
		which satisfy for all 
		$ \varepsilon \in (0,T] $, 
		$ p \in (0,\infty) $,
		$
		x \in H
		$, 
		$ t \in [0,T] $
		that
		$
		\sup_{s\in[0,T]}
		\E\big[\|X^{\varepsilon,x}_s\|^p_H\big]
		< \infty
		$ 
		and 
		\begin{equation}
		\begin{split}
		&
		[
		X_t^{\varepsilon,x}
		-
		e^{tA} x
		]_{ \P,\mathcal{B}(H) }
		=
		\int_0^t
		e^{ ( t - s + \varepsilon ) A }
		F(X_s^{\varepsilon,x})
		\,{\bf ds}
		+
		\int_0^t
		e^{ ( t - s + \varepsilon ) A }
		B(X_s^{\varepsilon,x})
		\, \diffns W_s
		,
		\end{split}
		\end{equation}
		\item
		\label{item:cor.mollified.TPS}
		it holds that there exist unique functions 
		$
		\transmol{\varepsilon} \colon [0,T] \times H \to V
		$, 
		$ \varepsilon \in (0,T] $, 
		which satisfy for all 
		$ \varepsilon \in (0,T] $, 
		$ t \in [0,T] $, 
		$ x \in H $ 
		that 
		$
		\transmol{\varepsilon}(t,x)
		=
		\E[\varphi(X^{\varepsilon,x}_t)]
		$,
		\item
		\label{item:cor.mollified.transition.derivative}
		it holds for all 
		$ \varepsilon \in (0,T] $, 
		$ t \in [0,T] $
		that 
		$
		\big(
		H \ni x \mapsto
		\transmol{\varepsilon}(t,x) \in V
		\big)
		\in \Cb{n}(H,V)
		$, 
		\item
		\label{item:cor.mollified.transition.bound}
		it holds for all 
		$ k \in \{1,2,\ldots,n\} $, 
		$ \delta_1, \delta_2, \ldots, \delta_k \in [0,\nicefrac{1}{2})  $ 
		with 
		$ \sum^k_{i=1} \delta_i < \nicefrac{1}{2} $ 
		that 
		\begin{equation}
		\sup_{ \varepsilon, t \in (0,T] }
		\sup_{ v \in H }
		\sup_{ \mathbf{u}=(u_1,u_2,\ldots,u_k) \in (\nzspace{H})^k }
		\left[
		\frac{
			t^{ \sum^k_{i=1} \delta_i } \,
			\big\|
			\big(
			\tfrac{\partial^k}{\partial x^k}
			\transmol{\varepsilon}
			\big)(t,v)
			\mathbf{u}
			\big\|_V
		}{
		\prod^k_{i=1} \|u_i\|_{ H_{-\delta_i} }
	}
	\right]
	< \infty, 
	\end{equation}
	and 
	\item
	\label{item:cor.mollified.transition.lip}
	it holds for all 
	$ k \in \{1,2,\ldots,n\} $, 
	$ \delta_1, \delta_2, \ldots, \delta_k \in [0,\nicefrac{1}{2})  $ 
	with 
	$ \sum^k_{i=1} \delta_i < \nicefrac{1}{2} $ 
	and 
	$
	|F|_{\operatorname{Lip}^k(H,H_{-\alpha})} +
	|B|_{\operatorname{Lip}^k(H,HS(U,H_{-\beta}))} + 
	|\varphi|_{\operatorname{Lip}^k(H,V)} < \infty 
	$ 
	that 
	\begin{equation}
	\begin{split}
	&\!\!
	\sup_{ \varepsilon, t \in (0,T] }
	\sup_{\substack{ v,w \in H, \\ v \neq w }}
	\sup_{ \mathbf{u}=(u_1,u_2,\ldots,u_k) \in (\nzspace{H})^k }
	\left[
	\frac{
		t^{ \sum^k_{i=1} \delta_i } \,
		\big\|
		\big[
		\big(
		\tfrac{\partial^k}{\partial x^k}
		\transmol{\varepsilon}
		\big)(t,v)
		-
		\big(
		\tfrac{\partial^k}{\partial x^k}
		\transmol{\varepsilon}
		\big)(t,w)
		\big]
		\mathbf{u}
		\big\|_V
	}{
	\|v-w\|_H
	\prod^k_{i=1} \|u_i\|_{ H_{-\delta_i} }
}
\right]
< \infty. 
\end{split}
\end{equation}
\end{enumerate}
\end{corollary}
\begin{proof}
Note that items~\eqref{item:cor.mollified.existence} and~\eqref{item:cor.mollified.TPS} follow immediately from item~(i) of Theorem~2.1 in~\cite{AnderssonJentzenKurniawan2016a} 
(with 
$ T = T $, 
$ \eta = \eta $, 
$ H = H $, 
$ U = U $, 
$ W = W $, 
$ A = A $, 
$ n = n $, 
$ F = F $, 
$ B = B $, 
$ \alpha = 0 $, 
$ \beta = 0 $
in the notation of Theorem~2.1 
in~\cite{AnderssonJentzenKurniawan2016a}) 
and item~(i) of Corollary~2.10 in~\cite{AnderssonJentzenKurniawan2016arXiv}.
Moreover, observe that items~\eqref{item:cor.mollified.transition.derivative}--\eqref{item:cor.mollified.transition.lip} follow directly from items~\eqref{item:mollified.transition.derivative}--\eqref{item:mollified.transition.lip} of Lemma~\ref{lem:mollified.transition}. 
The proof of Corollary~\ref{thm:mollified.transition} is thus completed.
\end{proof}

\section*{Acknowledgements}

Stig Larsson and Christoph Schwab are gratefully acknowledged for some useful comments.
This project has been supported through the SNSF-Research project 200021\_156603 
"Numerical approximations of nonlinear stochastic ordinary and partial differential equations".

\bibliographystyle{acm}
\bibliography{Bib/bibfile}
\end{document}